\pdfoutput=1
\documentclass[a4paper]{amsart}
\usepackage[hmarginratio={1:1},vmarginratio={1:1},lmargin=100.0pt,tmargin=90.0pt]{geometry}
\usepackage{ifdraft}
\ifdraft{\usepackage[draft]{showkeys}}{\usepackage[final]{showkeys}}
\usepackage{calligra}
\usepackage[T1]{fontenc}
\usepackage{lmodern}
\usepackage[ugly]{nicefrac}

\usepackage{mathtools}
\mathtoolsset{mathic}
\usepackage[final]{microtype}
\usepackage{multicol}
\usepackage{tabularx,tabu}
\usepackage{latexsym,exscale,enumitem,amsfonts,amssymb,mathtools}
\usepackage{amsmath,amsthm,amsfonts,amssymb,amscd,textcomp,bbm}
\usepackage{stmaryrd}
\SetSymbolFont{stmry}{bold}{U}{stmry}{m}{n}
\usepackage[normalem]{ulem}
\usepackage{thmtools}
\usepackage{fancybox}
\usepackage[table]{xcolor}
\usepackage{mathrsfs}
\DeclareMathAlphabet{\mathscrbf}{OMS}{mdugm}{b}{n}
\DeclareFontEncoding{LS1}{}{}
\DeclareFontSubstitution{LS1}{stix}{m}{n}
\DeclareMathAlphabet{\mathpzc}{LS1}{stixscr}{m}{n}

\usepackage{caption} 
\makeatletter
\newcommand{\leqnomode}{\tagsleft@true\let\veqno\@@leqno}
\newcommand{\reqnomode}{\tagsleft@false\let\veqno\@@eqno}
\makeatother

\usepackage[all]{xy}
\SelectTips{cm}{}

\usepackage{tikz}
\tikzset{anchorbase/.style={baseline={([yshift=-0.5ex]current bounding box.center)}}}
\usetikzlibrary{decorations.markings}
\usetikzlibrary{decorations.pathreplacing}
\usetikzlibrary{arrows.meta,shapes,positioning,matrix,calc}
\usetikzlibrary{decorations.pathmorphing}
\usetikzlibrary{shapes.callouts}
\tikzstyle directed=[postaction={decorate,decoration={markings,
mark=at position #1 with {\arrow{>}}}}]
\tikzstyle rdirected=[postaction={decorate,decoration={markings,
mark=at position #1 with {\arrow{<}}}}]
\tikzset{
partial ellipse/.style args={#1:#2:#3}{
insert path={+ (#1:#3) arc (#1:#2:#3)}
}
}

\allowdisplaybreaks

\makeatletter
\newcommand{\setword}[2]{%
\phantomsection
#1\def\@currentlabel{\unexpanded{#1}}\label{#2}%
}
\makeatother

\newcommand{\neatfrac}[2]{{}^{#1}\! / \!{}_{#2}}

\newcommand{\bigslant}[2]{{\raisebox{.2em}{$#1$}\!\left/\raisebox{-.2em}{$#2$}\right.}}

\newcommand{\overunder}[3]{\raisebox{-.01cm}{$
\xy
(0,1.9)*{{\scriptstyle#1}};
(0,0)*{#3};
(0,-1.2)*{{\scriptstyle#2}};
\endxy$}
}

\makeatletter
\providecommand{\leftsquigarrow}{%
\mathrel{\mathpalette\reflect@squig\relax}%
}
\newcommand{\reflect@squig}[2]{%
\reflectbox{$\m@th#1\rightsquigarrow$}%
}
\makeatother

\def\mystrut(#1,#2){\vrule height #1 depth #2 width 0pt}

\newcolumntype{C}[1]{%
>{\mystrut(1ex,1ex)\centering}%
p{#1}%
<{}}

\newcolumntype{L}[1]{%
>{\mystrut(1ex,1ex)}%
p{#1}%
<{}}

\definecolor{mycolor}{rgb}{0.9,0,0}
\definecolor{colormy}{rgb}{0.75,0,0}
\definecolor{nonecolor}{rgb}{0.9,0,0.9}
\definecolor{cupgray}{gray}{0.5}
\definecolor{paragray}{gray}{0.3}
\definecolor{myblue}{rgb}{0,0,0.8}
\definecolor{mygreen}{rgb}{0,0.7,0}
\definecolor{myred}{rgb}{0.9,0,0}
\definecolor{myyellow}{rgb}{0.9,0.9,0}
\definecolor{CMWcolor}{rgb}{0.8,0.2,0}

\definecolor{orchid}{RGB}{143,40,194}
\definecolor{lava}{RGB}{207,16,32}
\definecolor{mydarkblue}{RGB}{10,10,150}

\newcommand{\neatspec}[2]{\raisebox{.1cm}{$\underset{{\color{paragray}(#1)}}{#2}$}}

\newcommand{\C}{\mathbb{C}}
\newcommand{\Z}{\mathbb{Z}}
\newcommand{\R}{\mathbb{R}}

\newcommand{\osymbol}{\mathrm{o}}
\newcommand{\psymbol}{\mathrm{p}}

\newcommand{\parepi}{\boldsymbol{\tau}_{\osymbol}}
\newcommand{\parep}{\boldsymbol{\tau}_{\psymbol}}
\newcommand{\parom}{\boldsymbol{\omega}_+}
\newcommand{\parome}{\boldsymbol{\omega}_-}
\newcommand{\parto}{\boldsymbol{\alpha}}
\newcommand{\parha}{\boldsymbol{\beta}}

\newcommand{\parsign}{\boldsymbol{\varepsilon}}
\newcommand{\paro}{\boldsymbol{\omega}}

\newcommand{\parameter}{\mathtt{P}}
\newcommand{\parameterS}{\mathtt{Q}}

\newcommand{\specs}{\mathtt{p}}

\newcommand{\specS}{\mathtt{q}}

\newcommand{\F}{\boldsymbol{\mathfrak{F}}}
\newcommand{\Ffno}{{\boldsymbol{\mathcal{F}}}}
\newcommand{\Fker}{{\Ffno^{\mathrm{ker}}}}
\newcommand{\Ff}{\tilde{\boldsymbol{\mathcal{F}}}}
\newcommand{\field}{P}
\newcommand{\ring}{Q}
\newcommand{\sTQFT}{\mathcal{Z}}

\newcommand{\Mo}{\mathcal{A}}
\newcommand{\M}{\mathcal{W}}

\newcommand{\fmod}{\field\text{-}\boldsymbol{\mathcal{M}\mathrm{od}}^{\mathrm{fin}}_{\mathrm{free}}}
\newcommand{\newmod}{\field\text{-}\boldsymbol{\mathcal{M}\mathrm{od}}}

\newcommand{\oneinsert}[1]{\boldsymbol{1}_{#1}}

\DeclareRobustCommand{\sltwo}[1][2]{\mathfrak{sl}_{#1}}
\DeclareRobustCommand{\gltwo}[1][2]{\mathfrak{gl}_{#1}}

\newcommand{\pmandnot}{\star}

\newcommand{\webalg}{\mathfrak{W}}
\newcommand{\webalgpm}{\mathfrak{W}^\pm}
\newcommand{\webalgboth}{\mathfrak{W}^\pmandnot}

\newcommand{\webcatpm}{\mathfrak{W}^\pm[\parameter]\text{-}\boldsymbol{\mathcal{B}\mathrm{im}}^p_{\mathrm{gr}}}
\newcommand{\webcatS}{\mathfrak{W}[\parameterS]\text{-}\boldsymbol{\mathcal{B}\mathrm{im}}^p_{\mathrm{gr}}}

\newcommand{\based}{\flat}

\newcommand{\webcatcirc}{\mathfrak{W}_{\based}[\parameter]\text{-}\boldsymbol{\mathcal{B}\mathrm{im}}^p_{\mathrm{gr}}}
\newcommand{\webcatScirc}{\mathfrak{W}_{\based}[\parameterS]\text{-}\boldsymbol{\mathcal{B}\mathrm{im}}^p_{\mathrm{gr}}}
\newcommand{\webcatcircpm}{\mathfrak{W}_{\based}^\pm[\parameter]\text{-}\boldsymbol{\mathcal{B}\mathrm{im}}^p_{\mathrm{gr}}}
\newcommand{\webcatScircpm}{\mathfrak{W}_{\based}^\pm[\parameterS]\text{-}\boldsymbol{\mathcal{B}\mathrm{im}}^p_{\mathrm{gr}}}
\newcommand{\webcatcircboth}{\mathfrak{W}_{\based}^\pmandnot[\parameter]\text{-}\boldsymbol{\mathcal{B}\mathrm{im}}^p_{\mathrm{gr}}}

\newcommand{\webcatcircgl}{\mathfrak{W}_{\based}^\pm[\gltwo]\text{-}\boldsymbol{\mathcal{B}\mathrm{im}}^p_{\mathrm{gr}}}
\newcommand{\webcatcircglpm}{\mathfrak{W}_{\based}^\pm[\gltwo]\text{-}\boldsymbol{\mathcal{B}\mathrm{im}}^p_{\mathrm{gr}}}

\newcommand{\somealg}{\mathrm{A}}
\newcommand{\someotheralg}{\mathrm{B}}

\newcommand{\web}{\mathrm{w}}
\newcommand{\cc}[1]{\mathrm{a}(#1)}
\newcommand{\CUP}{\mathrm{CUP}}
\newcommand{\CUPB}{\mathrm{Cup}}

\newcommand{\basedd}{\circ}

\newcommand{\CUPbasis}[3]{{}_{#1}\mathbb{B}^{\basedd}(#3)_{#2}}
\newcommand{\CUPBasis}[1]{\mathbb{B}^{\basedd}(#1)}
\newcommand{\Cupbasis}[3]{{}_{#1}\mathbb{B}(#3)_{#2}}
\newcommand{\CupBasis}[1]{\mathbb{B}(#1)}

\newcommand{\stacked}{D}
\newcommand{\stackedd}{\tilde D}
\newcommand{\coeff}{\mathrm{coeff}}
\newcommand{\coeffb}{\overline{\mathrm{coeff}}}
\newcommand{\clapping}{\mathrm{cl}}

\newcommand{\Modpgr}[1]{#1\text{-}\boldsymbol{\mathcal{B}\mathrm{im}}^p_{\mathrm{gr}}}

\newcommand{\Y}{\mathbbm{bl}}
\newcommand{\bY}{\mathbbm{bl}_{\diamond}}

\newcommand{\Mod}[1]{#1\text{-}\boldsymbol{\mathcal{M}\mathrm{od}}}
\newcommand{\biMod}[1]{#1\text{-}\boldsymbol{\mathcal{B}\mathrm{im}}}

\newcommand{\Iso}[2]{\Phi_{#1}^{#2}}
\newcommand{\Isonew}[2]{\Psi_{#1}^{#2}}

\newcommand{\Hom}{\mathrm{Hom}}
\newcommand{\twoHom}{\mathrm{H}\mathrm{om}}
\newcommand{\End}{\mathrm{End}}
\newcommand{\twoEnd}{\mathcal{E}\mathrm{nd}}
\newcommand{\Kar}{\mathcal{K}\mathrm{ar}}

\newcommand{\down}{{\scriptstyle\vee}}
\newcommand{\up}{{\scriptstyle\wedge}}

\newcommand{\block}{\mathtt{bl}}
\newcommand{\bblock}{\mathtt{bl}_{\diamond}}

\newcommand{\length}{\mathrm{d}}

\newcommand{\Arcalg}{\mathfrak{A}[\parameterS]}
\newcommand{\ArcalgS}{\mathfrak{A}}

\newcommand{\compmatch}{\vec{\Lambda}}
\newcommand{\compmatcht}{\vec{\mathrm{t}}}

\newcommand{\op}{\operatorname}
\newcommand{\pos}{\mathrm{p}}

\newcommand{\dummy}{{\scriptstyle\bigstar}}

\newcommand{\factor}{\chi}

\newcommand{\KBN}{\boldsymbol{\mathrm{KBN}}}
\newcommand{\CMW}{\boldsymbol{\mathrm{CMW}}}
\newcommand{\Ca}{\boldsymbol{\mathrm{Ca}}}
\newcommand{\Bl}{\boldsymbol{\mathrm{Bl}}}

\newcommand{\Hgensl}[1]{\left\llbracket #1 \right\rrbracket_{\sltwo}}
\newcommand{\Hgenslo}[1]{\overline{\left\llbracket #1 \right\rrbracket}_{\sltwo}}
\newcommand{\Hgengl}[1]{\left\llbracket #1 \right\rrbracket_{\gltwo}}

\newcommand{\Hgenslpm}[1]{\left\llbracket #1 \right\rrbracket^{\pm}_{\sltwo}}
\newcommand{\Hgenglpm}[1]{\left\llbracket #1 \right\rrbracket^{\pm}_{\gltwo}}
\newcommand{\Hgenpm}[1]{\left\llbracket #1 \right\rrbracket^{\pm}_{\parameter}}

\newcommand{\HKBNpm}[1]{\left\llbracket #1 \right\rrbracket^{\pm}_{\KBN}}
\newcommand{\HCMWpm}[1]{\left\llbracket #1 \right\rrbracket^{\pm}_{\CMW}}
\newcommand{\HCapm}[1]{\left\llbracket #1 \right\rrbracket^{\pm}_{\Ca}}
\newcommand{\HBlpm}[1]{\left\llbracket #1 \right\rrbracket^{\pm}_{\Bl}}

\newcommand{\Hgenboth}[1]{\left\llbracket #1 \right\rrbracket^{\pmandnot}_{\parameter}}

\newcommand{\HKBNboth}[1]{\left\llbracket #1 \right\rrbracket^{\pmandnot}_{\KBN}}
\newcommand{\HCMWboth}[1]{\left\llbracket #1 \right\rrbracket^{\pmandnot}_{\CMW}}
\newcommand{\HCaboth}[1]{\left\llbracket #1 \right\rrbracket^{\pmandnot}_{\Ca}}
\newcommand{\HBlboth}[1]{\left\llbracket #1 \right\rrbracket^{\pmandnot}_{\Bl}}
\newcommand{\HOboth}[1]{\left\llbracket #1 \right\rrbracket^{\pmandnot}_{\mathcal{O}}}

\newcommand{\Hgen}[1]{\left\llbracket #1 \right\rrbracket_{\parameter}}
\newcommand{\HGen}[1]{\left\llbracket #1 \right\rrbracket_{\parameterS}}

\newcommand{\HKBN}[1]{\left\llbracket #1 \right\rrbracket_{\KBN}}

\newcommand{\Kom}{\boldsymbol{\mathcal{C}}^{b}}
\newcommand{\Komh}{\boldsymbol{\mathcal{K}}^{b}}

\newcommand{\twoKomh}{\boldsymbol{\mathfrak{K}}^{b}}
\newcommand{\Tan}{\boldsymbol{\mathcal{T}\mathrm{an}}^\pm}
\newcommand{\twoTan}{\boldsymbol{\mathfrak{T}\mathrm{an}}^\pm}

\newcommand{\twoTanup}{\boldsymbol{\mathfrak{T}\mathrm{an}}}

\newcommand{\ecup}{
\xy
(0,0)*{\includegraphics[scale=1]{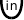}};
\endxy
}

\newcommand{\ecups}{
\xy
(0,0)*{\includegraphics[scale=.75]{figs/fig0-3.pdf}};
\endxy
}

\newcommand{\icup}{
\xy
(0,0)*{\includegraphics[scale=1]{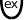}};
\endxy
}

\newcommand{\icups}{
\xy
(0,0)*{\includegraphics[scale=.75]{figs/fig0-1.pdf}};
\endxy
}

\newcommand{\ecap}{
\xy
(0,0)*{\includegraphics[scale=1]{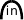}};
\endxy
}

\newcommand{\ecaps}{
\xy
(0,0)*{\includegraphics[scale=.75]{figs/fig0-4.pdf}};
\endxy
}

\newcommand{\icap}{
\xy
(0,0)*{\includegraphics[scale=1]{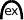}};
\endxy
}

\newcommand{\icaps}{
\xy
(0,0)*{\includegraphics[scale=.75]{figs/fig0-2.pdf}};
\endxy
}

\newcommand{\eleft}{
\xy
(0,0)*{\includegraphics[scale=1]{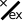}};
\endxy
}

\newcommand{\elefts}{
\xy
(0,0)*{\includegraphics[scale=.75]{figs/fig0-6.pdf}};
\endxy
}

\newcommand{\ileft}{
\xy
(0,0)*{\includegraphics[scale=1]{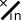}};
\endxy
}

\newcommand{\ilefts}{
\xy
(0,0)*{\includegraphics[scale=.75]{figs/fig0-8.pdf}};
\endxy
}

\newcommand{\eright}{
\xy
(0,0)*{\includegraphics[scale=1]{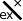}};
\endxy
}

\newcommand{\erights}{
\xy
(0,0)*{\includegraphics[scale=.75]{figs/fig0-5.pdf}};
\endxy
}

\newcommand{\iright}{
\xy
(0,0)*{\includegraphics[scale=1]{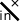}};
\endxy
}

\newcommand{\irights}{
\xy
(0,0)*{\includegraphics[scale=.75]{figs/fig0-7.pdf}};
\endxy
}

\newtheorem{theoremm}{Theorem}[section]

\declaretheorem[style=plain,name=Theorem,numbered=no]{theoremintro}
\declaretheorem[style=plain,name=Corollary,qed=$\square$,numbered=no]{corollaryintro}

\declaretheorem[style=plain,name=Theorem,numberlike=theoremm]{theorem}
\declaretheorem[style=plain,name=Lemma,numberlike=theoremm]{lemma}
\declaretheorem[style=plain,name=Proposition,numberlike=theoremm]{proposition}

\declaretheorem[style=plain,name=Lemma,qed=$\square$,numberlike=theoremm]{lemmaqed}
\declaretheorem[style=plain,name=Proposition,qed=$\square$,numberlike=theoremm]{propositionqed}
\declaretheorem[style=plain,name=Corollary,qed=$\square$,numberlike=theoremm]{corollary}

\declaretheorem[style=definition,name=Definition,numberlike=theorem]{definition}

\declaretheorem[style=remark,name=Example,numberlike=theorem]{example}
\declaretheorem[style=remark,name=Remark,numberlike=theorem]{remark}

\allowdisplaybreaks

\numberwithin{equation}{section}
\usepackage[final]{hyperref,bookmark}
\usepackage{cleveref}
\hypersetup{
pdftoolbar=true,        
pdfmenubar=true,        
pdffitwindow=false,     
pdfstartview={FitH},    
pdftitle={Generic \texorpdfstring{$\mathfrak{gl}_2$}{gl2}-foams, web and arc algebras},    
pdfauthor={Michael Ehrig, Catharina Stroppel and Daniel Tubbenhauer},     
pdfsubject={},   
pdfcreator={Michael Ehrig, Catharina Stroppel and Daniel Tubbenhauer},   
pdfproducer={Michael Ehrig, Catharina Stroppel and Daniel Tubbenhauer}, 
pdfkeywords={}, 
pdfnewwindow=true,      
colorlinks=true,       
linkcolor=mydarkblue,          
citecolor=teal,        
filecolor=magenta,      
urlcolor=orchid,          
linkbordercolor=lava,
citebordercolor=teal,
urlbordercolor=orchid,  
linktocpage=true
}
\let\fullref\autoref
\def\makeautorefname#1#2{\expandafter\def\csname#1autorefname\endcsname{#2}}
\makeautorefname{equation}{Equation}%
\makeautorefname{footnote}{footnote}%
\makeautorefname{item}{item}%
\makeautorefname{figure}{Figure}%
\makeautorefname{table}{Table}%
\makeautorefname{part}{Part}%
\makeautorefname{appendix}{Appendix}%
\makeautorefname{chapter}{Chapter}%
\makeautorefname{section}{Section}%
\makeautorefname{subsection}{Section}%
\makeautorefname{subsubsection}{Section}%
\makeautorefname{paragraph}{Paragraph}%
\makeautorefname{subparagraph}{Paragraph}%
\makeautorefname{theorem}{Theorem}%
\makeautorefname{theo}{Theorem}%
\makeautorefname{thm}{Theorem}%
\makeautorefname{addendum}{Addendum}%
\makeautorefname{addend}{Addendum}%
\makeautorefname{add}{Addendum}%
\makeautorefname{maintheorem}{Main theorem}%
\makeautorefname{mainthm}{Main theorem}%
\makeautorefname{corollary}{Corollary}%
\makeautorefname{corol}{Corollary}%
\makeautorefname{coro}{Corollary}%
\makeautorefname{cor}{Corollary}%
\makeautorefname{lemma}{Lemma}%
\makeautorefname{lemm}{Lemma}%
\makeautorefname{lem}{Lemma}%
\makeautorefname{sublemma}{Sublemma}%
\makeautorefname{sublem}{Sublemma}%
\makeautorefname{subl}{Sublemma}%
\makeautorefname{proposition}{Proposition}%
\makeautorefname{proposit}{Proposition}%
\makeautorefname{propos}{Proposition}%
\makeautorefname{propo}{Proposition}%
\makeautorefname{prop}{Proposition}%
\makeautorefname{property}{Property}
\makeautorefname{proper}{Property}
\makeautorefname{scholium}{Scholium}%
\makeautorefname{step}{Step}%
\makeautorefname{conjecture}{Conjecture}%
\makeautorefname{conject}{Conjecture}%
\makeautorefname{conj}{Conjecture}%
\makeautorefname{question}{Question}
\makeautorefname{questn}{Question}
\makeautorefname{quest}{Question}
\makeautorefname{ques}{Question}
\makeautorefname{qn}{Question}
\makeautorefname{definition}{Definition}%
\makeautorefname{defin}{Definition}%
\makeautorefname{defi}{Definition}%
\makeautorefname{def}{Definition}%
\makeautorefname{dfn}{Definition}%
\makeautorefname{notation}{Notation}
\makeautorefname{nota}{Notation}
\makeautorefname{notn}{Notation}
\makeautorefname{remark}{Remark}%
\makeautorefname{rema}{Remark}%
\makeautorefname{rem}{Remark}%
\makeautorefname{rmk}{Remark}%
\makeautorefname{rk}{Remark}%
\makeautorefname{remarks}{Remarks}%
\makeautorefname{rems}{Remarks}%
\makeautorefname{rmks}{Remarks}%
\makeautorefname{rks}{Remarks}%
\makeautorefname{example}{Example}%
\makeautorefname{examp}{Example}%
\makeautorefname{exmp}{Example}%
\makeautorefname{exam}{Example}%
\makeautorefname{exa}{Example}%
\makeautorefname{algorithm}{Algorith}%
\makeautorefname{algo}{Algorith}%
\makeautorefname{alg}{Algorith}%
\makeautorefname{axiom}{Axiom}%
\makeautorefname{axi}{Axiom}%
\makeautorefname{ax}{Axiom}%
\makeautorefname{case}{Case}%
\makeautorefname{claim}{Claim}%
\makeautorefname{clm}{Claim}%
\makeautorefname{assumption}{Assumption}%
\makeautorefname{assumpt}{Assumption}%
\makeautorefname{conclusion}{Conclusion}%
\makeautorefname{concl}{Conclusion}%
\makeautorefname{conc}{Conclusion}%
\makeautorefname{condition}{Condition}%
\makeautorefname{condit}{Condition}%
\makeautorefname{cond}{Condition}%
\makeautorefname{construction}{Construction}%
\makeautorefname{construct}{Construction}%
\makeautorefname{const}{Construction}%
\makeautorefname{cons}{Construction}%
\makeautorefname{criterion}{Criterion}%
\makeautorefname{criter}{Criterion}%
\makeautorefname{crit}{Criterion}%
\makeautorefname{exercise}{Exercise}%
\makeautorefname{exer}{Exercise}%
\makeautorefname{exe}{Exercise}%
\makeautorefname{problem}{Problem}%
\makeautorefname{problm}{Problem}%
\makeautorefname{probm}{Problem}%
\makeautorefname{prob}{Problem}%
\makeautorefname{solution}{Solution}%
\makeautorefname{soln}{Solution}%
\makeautorefname{sol}{Solution}%
\makeautorefname{summary}{Summary}%
\makeautorefname{summ}{Summary}%
\makeautorefname{sum}{Summary}%
\makeautorefname{operation}{Operation}%
\makeautorefname{oper}{Operation}%
\makeautorefname{observation}{Observation}%
\makeautorefname{observn}{Observation}%
\makeautorefname{obser}{Observation}%
\makeautorefname{obs}{Observation}%
\makeautorefname{ob}{Observation}%
\makeautorefname{convention}{Convention}%
\makeautorefname{convent}{Convention}%
\makeautorefname{conv}{Convention}%
\makeautorefname{cvn}{Convention}%
\makeautorefname{warning}{Warning}%
\makeautorefname{warn}{Warning}%
\makeautorefname{note}{Note}%
\makeautorefname{fact}{Fact}%
\begin{document}
\vbadness=10001
\hbadness=10001
\overfullrule=1mm
\title[Generic $\gltwo$-foams, web and arc algebras]{Generic $\gltwo$-foams, web and arc algebras}
\author[M. Ehrig, C. Stroppel and D. Tubbenhauer]{Michael Ehrig, Catharina Stroppel and Daniel Tubbenhauer}

\address{M.E.: Beijing Institute of Technology, School of Mathematics and Statistics, Liangxiang Campus of Beijing Institute of Technology, Fangshan District, 100488 Beijing, China}
\email{michael.ehrig@outlook.com}

\address{C.S.: Mathematisches Institut, Universit\"at Bonn, Endenicher Allee 60, Room 4.007, 53115 Bonn, Germany}
\email{stroppel@math.uni-bonn.de}

\address{D.T.: Institut f{\"u}r Mathematik, Universit{\"a}t Z{\"u}rich, Winterthurerstrasse 190, Campus Irchel, Office Y27J32, CH-8057 Z{\"u}rich, Switzerland, \href{www.dtubbenhauer.com}{www.dtubbenhauer.com}}
\email{daniel.tubbenhauer@math.uzh.ch}

\begin{abstract}
We define parameter dependent 
$\gltwo$-foams and their 
associated web and arc algebras, and verify 
that they specialize to several known 
$\sltwo$ or $\gltwo$ 
constructions related to higher link and tangle invariants. 
Moreover, we show that all these specializations are equivalent, and 
we deduce several applications, e.g. we discuss the consequences 
for the associated link and tangle invariants, and their functoriality.
\end{abstract}

\maketitle
\tableofcontents
\section{Introduction}\label{sec:intro}
Let $\parameter=(\parto,\parha,\parepi,\parep,\parom,\parome)$ 
be a tuple of $6$ generic parameters with the last 
four $\parepi,\parep,\parom,\parome$ being invertible.
In this paper we introduce 
a $\parameter$-version of 
a two-dimensional \textit{singular topological quantum field theory (TQFT)} 
with the purpose to define a $6$\textit{-parameter} 
\textit{foam} $2$\textit{-category} 
$\F^\pm[\parameter]$ that unifies several topological variants of Khovanov homology.
By specialization of $\parameter$ we obtain for instance foam $2$-categories 
known from the context 
of higher link and tangle invariants associated 
with the quantum groups for $\sltwo$ or $\gltwo$:
\begin{gather}\label{eq:main-specs}
\renewcommand{\arraystretch}{1.3}
\begin{tabular}{L{11.5cm}}
$\KBN\!\!:\,$ Khovanov/Bar-Natan's 
cobordisms \cite{Kho}, \cite{BN1} by specializing 
$\parameter$ 
to $(0,0,1,1,1,1)$,
\tabularnewline
$\Ca\!:\,$ Caprau's foams \cite{Cap},\cite{Cap2} 
by specializing $\parameter$ 
to $(\parto,\parha,1,1,i,-i)$,
\tabularnewline
$\CMW\!\!:\,$ Clark--Morrison--Walker's disoriented 
cobordisms \cite{CMW}
by specializing $\parameter$ 
to $(0,0,1,1,i,-i)$,
\tabularnewline
$\Bl\!:\,$ Blanchet's foams \cite{Bla} 
by specializing $\parameter$ 
to $(0,0,1,-1,1,-1)$.
\tabularnewline
\end{tabular} 
\end{gather}

Our $6$-parameter version can in particular be used to 
commonly treat and compare all the theories in \eqref{eq:main-specs}.  
The $6$ parameters play quite different roles in the theory: 
\begin{enumerate}[label=(\roman*)]

\item \label{enum:intro-a} $\parto$ and $\parha$ deform the Frobenius 
algebra $\C[X]/(X^2)$ involved in the TQFT to a possibly semisimple algebra.

\item \label{enum:intro-b} $\parepi$ encodes different 
choices of traces of this Frobenius algebra.

\item \label{enum:intro-c} The remaining three parameters interpolate between the 
$\sltwo$ and the $\gltwo$ theory.
\end{enumerate}

As a consequence of our considerations we 
obtain that for fixed deformation parameters 
$\parto$ (arbitrary) and $\parha=0$, the 
specializations $(1,1,1,1)$, $(1,1,i,-i)$, $(1,-1,1,-1)$ 
for the remaining parameters are equivalent 
(see \fullref{corollary:all-equivalent}); in particular, the 
theories from \eqref{eq:main-specs} agree. 
That means not only that the tangle invariants are isomorphic, 
but stronger that the categories used for these 
theories are equivalent. 

We focus in this paper on the most interesting case $\parha=0$ avoiding 
the straightforward lengthy calculations in general. 
In fact, by working with a more suitable basis, the calculations 
can be significantly reduced covering the case 
$\parha\neq 0$ as well, as was shown in \cite{BHPW}.

One main player in our construction is the topologically defined 
\textit{web algebra} $\webalg^\pm[\parameter]$ 
corresponding to $\F^\pm[\parameter]$. The web algebra 
comes with a $2$-category of 
certain bimodules providing a fully faithful $2$-representation 
of $\F^\pm[\parameter]$. The second main player is 
the algebraically defined \textit{arc algebra}, 
which is a parameter version of Khovanov's 
arc algebra, \cite{Khov}, \cite{BS1}.  We consider 
this algebra for the specialization 
$\parameterS=(\parto,0,1,\parsign\paro^2,\paro,\parsign\paro)$ 
of $\parameter$ with $\paro$ generic and $\parsign\in\{\pm 1\}$. 
For this specialized algebra $\Arcalg$, we 
relate directly the theories from \eqref{eq:main-specs} 
and at the same time encode algebraically/combinatorially 
the topological data coming from foams. Our main result can be formulated (slightly informally) as follows: 

\makeautorefname{theorem}{Theorems}

\begin{theoremintro}(See  \fullref{theorem:matchalgebras}, 
\ref{theorem:crazy-isomorphism2} and \ref{theorem:matchalgebrasold}.)
The following holds: 
\begin{enumerate}[label=(\alph*)]

\item The web algebra specialized at $\parameterS$ is isomorphic to the arc algebra $\Arcalg$. The associated 
bimodule $2$-categories are equivalent.

\item $\Arcalg$ is a free 
$\Z[\parto,\parsign,\paro^\pm]$-algebra 
whose specializations $\Arcalg\otimes_\Z R$ 
to $(\parto,\parsign,\paro)\in R^3$ for some commutative ring $R$ 
with $\paro$ invertible and $\parsign\in\{\pm 1\}$ are all isomorphic 
as graded algebras. Similarly for the web algebra.

\end{enumerate}
\end{theoremintro}

\makeautorefname{theorem}{Theorem}

We want to emphasize that our approach 
provides \textit{explicit} isomorphisms which 
are moreover compatible with the simultaneous
specialization of $\parto$ on both side of the 
isomorphism to any element in $R$. (One can check that 
the corresponding statements also hold for general 
fixed $\parha$. We will not consider this 
situation here, but refer to \cite{BHPW} instead.) 
These isomorphisms induce explicit 
isomorphisms of the corresponding categories of 
graded bimodules, in particular of those which are the values of the  
higher tangle invariants. As a byproduct, we derive equivalences 
of graded, $\Z[\parto,\parsign,\paro^\pm]$-linear $2$-categories
\begin{gather}\label{eq:main-equi}
\Modpgr{\Arcalg}\cong\F[\parameterS]
\text{ and } 
\F[\ast]\cong\Modpgr{\ArcalgS[\ast]}.
\end{gather}
This also holds for any further simultaneous specialization of $\parto$. It implies in particular

\begin{corollaryintro}(See \fullref{corollary:all-equivalent}.)
The $\KBN$, $\Ca$, $\CMW$ and $\Bl$ 
foam $2$-categories, obtained via specialization as in \eqref{eq:main-specs},
are all equivalent when one works over $\Z[i]$.
\end{corollaryintro}

As an application we construct in \fullref{sec:further} 
invariants of tangles with values in the homotopy category 
of graded bimodules over the corresponding arc algebras 
for the $\gltwo$- and  $\sltwo$-specializations 
$(\parto,\parha,1,-1,\parsign\parom^2,\parom,-\parom^{-1})$
respectively 
$(\parto,\parha,1,1,\parsign\parom^2,\parom,\parom^{-1})$. 
An interesting question hereby is whether 
these constructions are \textit{functorial} in 
the sense that it extends to an invariant of tangles and 
cobordisms between tangles. We show in \fullref{theorem:functor-gl2} 
that the $\gltwo$-specializations always 
give functorial invariants, while the $\sltwo$-specializations 
fail to do so in the naive construction 
(as was already shown in \cite{Jac}, see also \cite{Str-TQFTs}).

The pure existence of our isomorphisms and equivalences might also be deduced from more abstract arguments, which however are not suitable to establish an explicit isomorphism.
Using arguments as in \cite[Section 5.3, Proposition 5.18]{MPT} or \cite[Section 4]{Tub1}
one can show that $\webalg_{\C}[\KBN]$ is Morita equivalent 
to a certain KLR algebra of level $2$ 
(using $\C$ as a ground field). 
Hereby, the two main ingredients in the corresponding proof in \cite{MPT}
were a categorification of an instance of $q$-Howe duality
as well as Rouquier's universality 
theorem \cite[Proposition 5.6 and Corollary 5.7]{Rou} 
(``such categorifications are unique'').

On the other hand, it was shown in \cite[Theorem 9.2]{BS3} 
that another instance of $q$-Howe duality can be categorified using 
$\ArcalgS_{\C}[\KBN]$. Moreover, one can deduce 
from \cite[Propositions 3.5 and 3.3]{LQR}---in the light 
of \fullref{proposition:cats-are-equal}---the same for 
$\webalg_{\Z[i]}[\CMW]$ and $\webalg_{\Z}[\Bl]$.
Since $\ArcalgS_{\C}[\KBN]$ is constructed 
from $\webalg_{\C}[\KBN]$, ``uniqueness of 
categorification in type $\mathrm{A}$''
would yield 
our isomorphism theorems for this very special case.

However, we should stress that there is still work 
to be done for the remaining theories, since it is only known 
for $\ArcalgS_{\C}[\KBN]$ that it categorifies the corresponding 
highest weight module of the ``Howe dual'' quantum group, 
and one would need to check the same for 
the $\Ca$, $\CMW$, $\Bl$ setups. In the general $6$-parameter 
setting the categorification results have still 
to be established. The abstract argumentation also 
misses the important different roles played by the deformation parameters.

Finally, it would be interesting to compare our results to \cite{Vog}, or to analyze the case of $\gltwo[n]$ and $\sltwo[n]$ web algebras as studied for example in \cite{Ma-sln-web-algebras} or \cite{Tub}. We expect that there are also sign adjustments for the odd arc algebras from \cite{NaVa} or for skein lasagna modules, \cite[Theorem 1.1]{MaNe} for $N=2$. One might like to reformulate the results of our paper in terms of the language developed in \cite{KhKi}, \cite{Kh20}.
\subsection{Conventions}\label{subsec:conventions}
By a \textit{ring} $R$ we mean
a commutative, unital ring without zero divisors. 
Given $R$-algebras $A$ and $B$ (not necessarily unital or free over $R$) we denote by $\Mod{\somealg}$ and by
by $\biMod{\somealg\text{-}\someotheralg}$ the category of $\somealg$-modules
$\somealg$-$\someotheralg$-bimodules which are free over $R$, respectively. An $\somealg$-$\someotheralg$-bimodule 
$M$ is \textit{biprojective}, if 
it is projective both as a left $\somealg$-module and as a right $\someotheralg$-module. \textit{Graded} 
means always $\Z$\textit{-graded}. Given graded algebras $\somealg$ and $\someotheralg$, 
we will consider the category of \textit{graded} 
$\somealg$-$\someotheralg$-bimodules
$M=\bigoplus_{i\in\Z}M_i$ with degree preserving 
morphisms denoted by $\Hom_{0}$. It comes with 
the grading shifting functor $\langle s\rangle$ such 
that $M\langle s\rangle_i=M_{i-s}$ for $s\in\Z$, and 
grading forgetting functor $M\mapsto\overline M$ to 
$\biMod{\somealg\text{-}\someotheralg}$. In particular,
\begin{gather}\label{eq:degreehom}
\Hom_{\biMod{\somealg\text{-}\someotheralg}}(\overline M,\overline N)=
{\textstyle\bigoplus_{s\in\Z}}\;
\Hom_{0}(M,N\langle s\rangle).
\end{gather}
Hence instead of working with graded modules, grading shifts and degree zero morphisms one could instead not allow grading shifts, but work with morphisms of arbitrary degrees, hence in the setting of graded categories. We use then analog conventions for graded categories (that is, categories enriched in graded modules), and graded $2$-categories 
are categories enriched in graded categories. We also use various sub- and superscripts on these, as already done in the introduction. Following the above points we view graded categories are often identified with categories with a free $\mathbb{Z}$-action (i.e. instead of enriching the morphisms one extends the set of objects), see e.g \cite[Section 2.1]{MaOvSt-koszul-functors} and references therein.

We will consider three diagrammatic calculi: 
webs, foams and arc diagrams. 
Our reading convention for all of these 
is from bottom to top and left to right. 
Diagrammatic right 
respectively left actions will be given by acting on the 
bottom respectively on the top.
We tend to simplify our illustrations by e.g. omitting 
labels and using colors instead to distinguish the type of edges and facets in webs and foams.
For the readers with a gray-scale version we illustrate colored web edges using 
dashed lines and colored foam facets as shaded.

We treat the proofs which require a longer line of argument separately in \fullref{sec:mainproofs}.

\medskip
\noindent\textbf{Acknowledgments:} 
We like to thank Jonathan Comes, Jonathan Grant, 
Martina Lanini, David Rose, Pedro Vaz, Paul Wedrich 
and Arik Wilbert for helpful comments and discussions, 
and the referees for useful suggestions. 
Special thanks to the bars of Cologne for helping to 
write down one of the main isomorphisms of this paper.
\section{Foams, web and arc algebras with parameters}\label{sec:foams}
The construction in this section is a $6$-parameter version of 
\cite{EST} (which is inspired by earlier work e.g.
\cite{Khova} and \cite{Bla}). We follow closely the setup therein; 
in particular, all conventions not involving parameters are the same as in \cite{EST}.
\subsection{Webs and prefoams}\label{subsec:foams}
Let $\parameter=(\parto,\parha,\parepi,\parep,\parom,\parome)$ 
be as in the introduction. We set 
$\field=\Z[\parto,\parha,\parepi^{\pm 1},\parep^{\pm 1},\parom^{\pm 1},\parome^{\pm 1}]$ as a graded ring, by putting all $6$ parameters in degree 
zero except of $\mathrm{deg}_{\field}(\parto)=4$ and $\mathrm{deg}_{\field}(\parha)=2$.

We fix the following commutative Frobenius algebras with trace map and comultiplication:
\begin{gather}\label{eq:frob-algebras}
\begin{array}{llcc}
\mathrm{F}_{\osymbol}=\bigslant{\field[X]}{I},
&\mathrm{tr}_{\osymbol}(1)=0,&\Delta_{\osymbol}(1)
=\parepi^{-1}(1\otimes X + X\otimes 1) - \parha 
1\otimes 1,
\\
&{\mathrm{tr}_{\osymbol}(X)=\parepi,}
&\Delta_{\osymbol}(X)=\parepi^{-1} X\otimes X +\parto 1\otimes 1.
\\
\mathrm{F}_{\psymbol}=\field&\mathrm{tr}_{\psymbol}(1)=\parep
&\Delta_{\psymbol}(1)=\parep^{-1} 1\otimes 1.\\
\end{array}
\end{gather}
Here $I$ is the ideal $(\parepi^{-2}X^2-\parto-\parha\parepi^{-1}X)$.

Each of these algebras defines an ordinary $2$-dimensional TQFT 
in the sense of \cite[Section 1.3]{Kock}. 
As in \cite[Section 2]{EST} we want to glue now these different TQFTs to 
obtain a singular TQFT. For that recall that a 
\textit{(topological) web} is a piecewise linear, plane, one-dimensional CW complex, equipped with a $\{1,2\}$-labeling and an orientation 
such that each interior point of it has a local neighborhood 
of the following form: (The outer frame is for illustration only.) 
\begin{gather}\label{eq:local-webs}
\raisebox{0.075cm}{\xy
(0,0)*{\includegraphics[scale=.65]{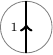}};
\endxy}
\quad,\quad
\xy
(0,0)*{\includegraphics[scale=.65]{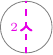}};
\endxy
\quad,\quad
\xy
(0,0)*{\includegraphics[scale=.65]{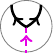}};
\endxy
\quad,\quad
\xy
(0,0)*{\includegraphics[scale=.65]{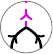}};
\endxy
\end{gather} 
By convention, this includes the 
\textit{empty web} $\emptyset$. A web is \textit{closed}, if its boundary is empty. 
Edges labeled $1$ are called \textit{ordinary edges}, 
whereas edges labeled $2$ are \textit{phantom edges} 
and drawn dashed (and colored). We use standard 
topological terms for the ordinary parts, for instance 
\textit{circles in webs}, are circles after we ignored 
all phantom edges. 
To define morphisms between webs recall the notion of a 
\textit{prefoam} from \cite[Definition~2.3]{EST}. These are certain 
singular surfaces (with dots on facets) such that each 
interior point has a local neighborhood 
of one of the four forms 
\begin{gather}\label{eq:orientation}
\raisebox{0.01cm}{\xy
(0,0)*{\includegraphics[scale=.65]{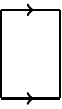}};
\endxy}
\quad,\quad
\raisebox{0.01cm}{\xy
(0,0)*{\includegraphics[scale=.65]{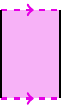}};
\endxy}
\quad,\quad
\reflectbox{\raisebox{0.00cm}{\xy
(0,0)*{\includegraphics[scale=.65]{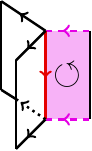}};
\endxy}}
\quad,\quad
\raisebox{0.00cm}{\xy
(0,0)*{\includegraphics[scale=.65]{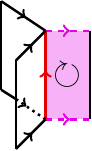}};
\endxy}
\colon
\raisebox{0.00cm}{\xy
(0,0)*{\includegraphics[scale=.65]{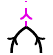}};
\endxy}
\to
\raisebox{0.00cm}{\xy
(0,0)*{\includegraphics[scale=.65]{figs/fig2-1e.pdf}};
\endxy}
\end{gather}
where we denoted as an example on the right the boundary webs at the bottom and top.
The first picture on the left displays an ordinary facet, the second a phantom facet.
The closed webs and prefoams form a monoidal category which we denote by $\Ff$.

In order to define a monoidal 
functor $\sTQFT[\parameter]\colon\Ff\to\newmod$ with 
values in $\field$-modules we use the so-called 
\textit{universal construction} from \cite[Section 1.A]{BHMV} which we 
briefly summarize. Given closed webs $v,w$ let $\Hom_{\field}(v,w)$ 
be the free $\field$-module with basis given by the set of prefoams from $v$ 
to $w$. This defines a functor $\Hom_{\field}(\emptyset,{}_-)\colon\Ff\to\newmod$. 
For any fixed closed web $w$, composition in $\Ff$ extends to a 
$\field$-linear map $$\op{m}_w\colon\Hom_{\field}(\emptyset,w)\otimes_P\Hom_{\field}(w,\emptyset)
\to\mathrm{End}_{\field}(\emptyset):=\Hom_{\field}(\emptyset,\emptyset),\quad f\otimes g\mapsto g\circ f.$$ 
Given now a $\field$-linear \textit{evaluation map}
$\mathsf{ev}\colon\mathrm{End}_{\field}(\emptyset)\to\field$ we let the \textit{radical} 
associated with $w$ be
$\mathrm{rad}_w=\{f\in\Hom_{\field}(\emptyset,W)\mid\mathsf{ev}\circ\op{m}_w(g\otimes f)=0
\text{ for any }g\in\Hom_{\field}(w,\emptyset)\}$. Then we obtain a functor 
$\sTQFT[\parameter]\colon\Ff\to\newmod$ defined on objects 
by $w\mapsto\Hom_{\field}(\emptyset,w)/\mathrm{rad}_w$ and on morphisms 
by $f\mapsto f\circ{}_-$. 
(Note that this is indeed well-defined, 
because if $f\in\mathrm{rad}_v$ and $h\in\Hom_{\field}(v,w)$, 
then $h\circ f\in\mathrm{rad}_w$, since $g\circ h\in\Hom_{\field}(v,\emptyset)$ 
for any $g\in\Hom_{\field}(w,\emptyset)$.)

We apply now this universal construction to a 
special choice of $\mathsf{ev}$. Namely, we take the 
Frobenius algebras $\mathrm{F}_{\osymbol}$ and 
$\mathrm{F}_{\psymbol}$ from \eqref{eq:frob-algebras} 
together with the gluing maps defined by 
\begin{gather}\label{eq:TQFT2}
\begin{aligned}
\mathrm{gl}_{\mathrm{F}_{\osymbol}}\colon
\mathrm{F}_{\osymbol}\otimes\mathrm{F}_{\osymbol}\to \mathrm{F}_{\osymbol},\, 
(a+& bX)\otimes (c+dX)\mapsto ac+(\parom bc + \parome ad) X,
\\
\mathrm{gl}_{\mathrm{F}_{\psymbol}}\colon 
&\mathrm{F}_{\psymbol}\to\mathrm{F}_{\osymbol},\; 1\mapsto 1.
\end{aligned}
\end{gather}
and then apply the method from \cite[Section 2.2]{EST} to get an 
evaluation map $\mathsf{ev}$. (By convention, $\mathsf{ev}(\mathrm{id}_{\emptyset})=1$.)
Our goal is to show 
that with this choices, the functor 
$\sTQFT[\parameter]$ has values in the 
category, $\fmod$, of free $\field$-modules of finite rank:

\begin{theorem}\label{theorem:blanchet}
The universal construction
with the above choice of 
evaluation $\mathsf{ev}$ gives rise to a monoidal 
functor $\sTQFT[\parameter]\colon\Ff\to\fmod$.
\end{theorem}

The proof of this occupies the next section. 
We call $\sTQFT[\parameter]$ the $6$\textit{-parameter, singular TQFT}. 
\subsection{Foam relations}\label{subsec:foam-relations}
To make sense of the following, we enrich 
$\Ff$ to a  $\field$-linear category $\Ffno$ (by allowing formal finite, $P$-linear combinations of morphisms) and consider the extension $\sTQFT[\parameter]\colon\Ffno\to\newmod$ of the functor $\sTQFT[\parameter]$. 
Given $\field$-linear combinations, say $f$ and $g$, of 
prefoams we say $f$ \textit{equals} $g$ \textit{modulo the kernel} of $\sTQFT[\parameter]$, if 
$\sTQFT[\parameter](f)=\sTQFT[\parameter](g)$ as $\field$-linear maps. 

\begin{remark}\label{remark:later-use}
Later we often use the 
specialization of $\parameter$ to $\parameterS$ 
from the introduction. 
For convenience, we indicate sometimes in small print, with brackets and in gray, the corresponding values 
under the specialization to $\parameterS$.
\end{remark}

The next lemma can be proven as in \cite[Section 2.2]{EST}.

\begin{lemma}\label{lemma:moreandmore}
The following equalities hold in $\Ffno$ modulo the kernel of $\sTQFT[\parameter]$.
The \textit{ordinary and phantom sphere relations} 
and the \textit{dot removing relations}:
\begin{gather}
\raisebox{0.075cm}{\xy
(0,0)*{\includegraphics[scale=.65]{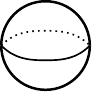}};
\endxy}
=0,\quad\quad
\raisebox{0.075cm}{\xy
(0,0)*{\includegraphics[scale=.65]{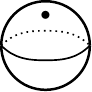}};
\endxy}
=\neatspec{1}{\parepi},\quad\quad
\raisebox{0.075cm}{\xy
(0,0)*{\includegraphics[scale=.65]{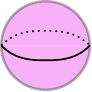}};
\endxy}
=\neatspec{\parsign\paro^{2}}{\parep},\label{eq:theusualrelations1}
\\
\neatspec{1}{\parepi^{-2}}\,
\raisebox{0.075cm}{\xy
(0,0)*{\includegraphics[scale=.65]{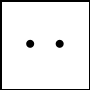}};
\endxy}
\;=
\neatspec{\parto}{\parto}
\;
\raisebox{0.075cm}{\xy
(0,0)*{\includegraphics[scale=.65]{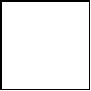}};
\endxy}
+
\neatspec{0}{\parha\parepi^{-1}}
\;
\raisebox{0.075cm}{\xy
(0,0)*{\includegraphics[scale=.65]{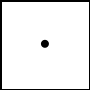}};
\endxy}
\quad,\quad
\raisebox{0.075cm}{\xy
(0,0)*{\includegraphics[scale=.65]{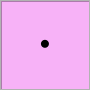}};
\endxy}
=
\neatspec{\parsign\paro^{-2}}{\parep^{-1}}
\raisebox{0.075cm}{\xy
(0,0)*{\includegraphics[scale=.65]{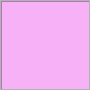}};
\endxy}
\label{eq:theusualrelations1b}
\end{gather}
The \textit{ordinary and phantom neck cutting relations}:
\begin{gather}\label{eq:theusualrelations2}
\raisebox{0.075cm}{\xy
(0,0)*{\includegraphics[scale=.65]{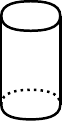}};
\endxy}
\;=\;
\neatspec{1}{\parepi^{-1}}
\raisebox{0.075cm}{\xy
(0,0)*{\includegraphics[scale=.65]{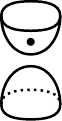}};
\endxy}
\;+\;
\neatspec{1}{\parepi^{-1}}
\raisebox{0.075cm}{\xy
(0,0)*{\includegraphics[scale=.65]{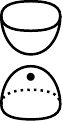}};
\endxy}
\;-\;
\neatspec{0}{\parha}
\raisebox{0.075cm}{\xy
(0,0)*{\includegraphics[scale=.65]{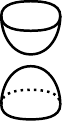}};
\endxy}
\quad,\quad
\raisebox{0.075cm}{\xy
(0,0)*{\includegraphics[scale=.65]{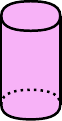}};
\endxy}
\;=\;
\neatspec{\parsign\paro^{-2}}{\parep^{-1}}
\raisebox{0.075cm}{\xy
(0,0)*{\includegraphics[scale=.65]{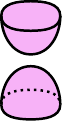}};
\endxy}
\end{gather}
The \textit{sphere} or \textit{theta foam} 
relations (where $a,b\in\{0,1\}$ denote the number of dots):
\begin{gather}\label{eq:sphere}
\raisebox{0.075cm}{\xy
(0,0)*{\includegraphics[scale=.65]{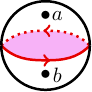}};
\endxy}
=
\setlength{\arraycolsep}{5pt}
\renewcommand{\arraystretch}{1.6}
\left\{\begin{array}{c c c c c}
0,            & \text{if }a=0,b=0, & \phantom{a} & \neatspec{\paro}{\parepi^{-1}\parom}, & \text{if }a=1,b=0, \\
\neatspec{\parsign\paro}{\parepi^{-1}\parome},   & \text{if }a=0,b=1, & \phantom{a} & 0, & \text{if }a=1,b=1. \\
\end{array}\right.
\end{gather}
The \textit{bubble removals} (where we have a sphere in a phantom plane, with the top 
dots on the front facets and the bottom dots on the back facets):
\begin{gather}\label{eq:bubble1}
\begin{gathered}
\phantom{.}
\xy
(0,0)*{\includegraphics[scale=.65]{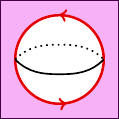}};
\endxy
=
0
=
\xy
(0,0)*{\includegraphics[scale=.65]{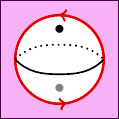}};
\endxy
\;,\;
\neatspec{\paro}{\parepi^{-1}\parep\parome^{-1}}
\xy
(0,0)*{\includegraphics[scale=.65]{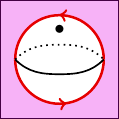}};
\endxy
=
\xy
(0,0)*{\includegraphics[scale=.65]{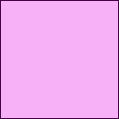}};
\endxy
=
\neatspec{\parsign\paro}{\parepi^{-1}\parep\parom^{-1}}
\xy
(0,0)*{\includegraphics[scale=.65]{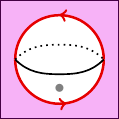}};
\endxy
\end{gathered}
\end{gather}
The \textit{(singular) neck cutting relation} (with top dot in the front and 
bottom dot in the back):
\begin{gather}\label{eq:neckcut}
\begin{aligned}
\phantom{.}
\xy
(0,0)*{\includegraphics[scale=.65]{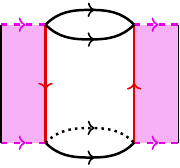}};
\endxy
=&
\neatspec{\paro}{\parepi^{-1}\parep\parome^{-1}}
\xy
(0,0)*{\includegraphics[scale=.65]{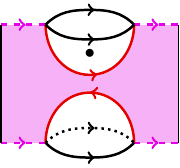}};
\endxy
+
\neatspec{\parsign\paro}{\parepi^{-1}\parep\parom^{-1}}
\xy
(0,0)*{\includegraphics[scale=.65]{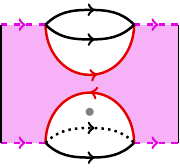}};
\endxy
\end{aligned}
\end{gather}
The \textit{ordinary-to-phantom neck cutting relation} 
(in the leftmost picture the upper closed circle is an ordinary facet) 
and the \textit{squeezing relation}:
\\
\noindent\begin{tabularx}{.95\textwidth}{@{}XX@{}}
\begin{equation}\hspace{-7cm}\label{eq:neckcutphantom}
\xy
(0,0)*{\includegraphics[scale=.65]{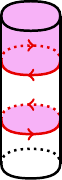}};
\endxy
\;=\;
\neatspec{\parsign}{\parep^{-1}\parome^2}
\xy
(0,0)*{\includegraphics[scale=.65]{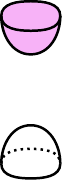}};
(0,16)*{\phantom{a}};
\endxy
\end{equation} &
\begin{equation}\hspace{-4.5cm}\label{eq:squeezing}
\xy
(0,0)*{\includegraphics[scale=.65]{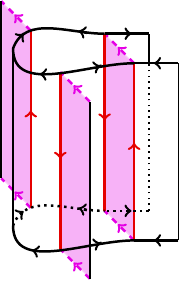}};
\endxy
=\;
\neatspec{\parsign}{\parep\parome^{-2}}
\xy
(0,0)*{\includegraphics[scale=.65]{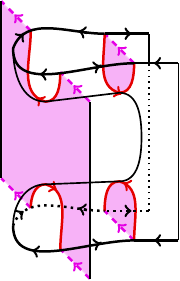}};
\endxy
\end{equation}
\end{tabularx}\\
The \textit{dot migration relations}:
\begin{gather}\label{eq:dotmigration}
\scalebox{.8}{$
\neatspec{1}{\parepi^{-1}}
\reflectbox{
\raisebox{0.075cm}{\xy
(0,0)*{\includegraphics[scale=.65]{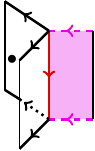}};
\endxy}}
\;=\;
\neatspec{\parsign}{\parepi^{-1}\parom\parome^{-1}}
\reflectbox{
\raisebox{0.075cm}{\xy
(0,0)*{\includegraphics[scale=.65]{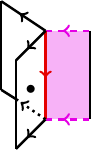}};
\endxy}}
\;-\;
\neatspec{0}{\parha\parom\parome^{-1}}
\reflectbox{
\raisebox{0.075cm}{\xy
(0,0)*{\includegraphics[scale=.65]{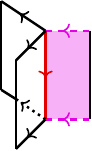}};
\endxy}}
\;,\;
\neatspec{1}{\parepi^{-1}}
\raisebox{0.075cm}{\xy
(0,0)*{\includegraphics[scale=.65]{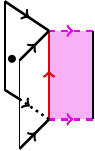}};
\endxy}
\;=\;
\neatspec{\parsign}{\parepi^{-1}\parom\parome^{-1}}
\raisebox{0.075cm}{\xy
(0,0)*{\includegraphics[scale=.65]{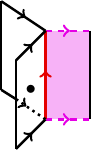}};
\endxy}
\;-\;
\neatspec{0}{\parha\parom\parome^{-1}}
\raisebox{0.075cm}{\xy
(0,0)*{\includegraphics[scale=.65]{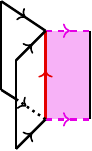}};
\endxy}$}
\end{gather}
The \textit{closed seam removal relation}:
\begin{gather}\label{eq:closedseam}
\raisebox{0.075cm}{\xy
(0,0)*{\includegraphics[scale=.65]{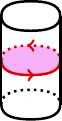}};
\endxy}
\;=\;
\neatspec{\parsign\paro}{\parepi^{-1}\parome}
\raisebox{0.075cm}{\xy
(0,0)*{\includegraphics[scale=.65]{figs/fig2-26.pdf}};
\endxy}
\;+
\neatspec{\paro}{\parepi^{-1}\parom}
\raisebox{0.075cm}{\xy
(0,0)*{\includegraphics[scale=.65]{figs/fig2-27.pdf}};
\endxy}
\end{gather}
The \textit{phantom cup removal} and 
the \textit{phantom squeeze relation}:
\\
\noindent\begin{tabularx}{.95\textwidth}{@{}XX@{}}
\begin{equation}\hspace{-6.7cm}\label{eq:phantom-removal}
\xy
(0,0)*{\includegraphics[scale=.65]{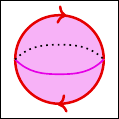}};
\endxy
\!=\!
\neatspec{\parsign\paro}{\parepi^{-1}\parome}
\xy
(0,0)*{\includegraphics[scale=.65]{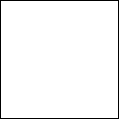}};
(0,8.75)*{\phantom{a}};
\endxy
\end{equation} &
\begin{equation}\hspace{-5.7cm}\label{eq:phantom-squeezing}
\xy
(0,0)*{\includegraphics[scale=.65]{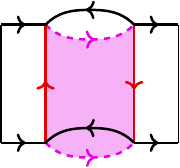}};
\endxy
\!=\!
\neatspec{\parsign^{-1}\paro^{-1}}{\parepi\parome^{-1}}
\xy
(0,0)*{\includegraphics[scale=.65]{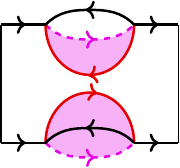}};
\endxy
\end{equation}
\end{tabularx}\\
and their counterparts with the phantom facets in the back, and  
$\parom$ instead of $\parome$.
\end{lemma}

\begin{remark}\label{remark:generators}
By repeated application of neck cutting \eqref{eq:theusualrelations2},
one can in fact check that the 
relations \eqref{eq:theusualrelations1}-\eqref{eq:sphere} already 
imply the remaining relations.
\end{remark}

Note that the substitution $X\mapsto\parepi X$, which rescales the dot on ordinary facets, gets rid of the parameter $\parepi$;
we will set $\parepi=1$ from now on without loss of generality.

Given a prefoam $f$, let $\hat f$ be 
the (dotted) CW complex obtained from it 
by removing the phantom edges and facets. We define the
\textit{degree} of $f$ as
\begin{gather}\label{eq:foam-degree}
\mathrm{deg}(f)=-\chi(\hat f)+2\#\text{dots}+\neatfrac{1}{2}\#\text{vbound},
\end{gather}
where $\chi$ is the topological Euler characteristic, 
and the numbers $\#\text{dots}$ and $\#\text{vbound}$, denote the total number 
of dots, respectively of vertical boundary components of $\hat f$.

\begin{definition}
Let $\Fker$ be the additive closure of the monoidal category 
obtained by taking the quotient of $\Ffno$ by the relations modulo 
the kernel of $\sTQFT[\parameter]$ listed in \fullref{lemma:moreandmore}.
\end{definition}

(The \textit{additive closure} means that we 
allow instead finite formal direct sums of objects 
and matrices of morphisms, see e.g. \cite[Definition 3.2]{BN1}. 
The additive closure of a $2$-category means we take the additive 
closure on the level of $1$- and $2$-morphisms, keeping the objects.)

Note that $\Fker$ is a graded category. 
We have in $\Fker$ isomorphisms of objects
\begin{gather}\label{eq:web-eval}
\underbrace{\xy
(0,0)*{\includegraphics[scale=.65]{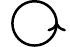}};
\endxy
\cong
\emptyset\langle+1\rangle
\oplus
\emptyset\langle-1\rangle}_{\text{By }\eqref{eq:theusualrelations1},\eqref{eq:theusualrelations2}},
\quad
\underbrace{\xy
(0,0)*{\includegraphics[scale=.65]{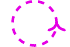}};
\endxy
\cong
\emptyset}_{\text{By }\eqref{eq:theusualrelations1},\eqref{eq:theusualrelations2}},
\quad
\underbrace{\xy
(0,0)*{\includegraphics[scale=.65]{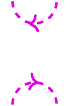}};
\endxy
\cong
\xy
(0,0)*{\includegraphics[scale=.65]{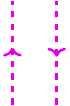}};
\endxy}_{\text{By }\eqref{eq:theusualrelations2}}
\quad,\quad
\underbrace{\xy
(0,0)*{\includegraphics[scale=.65]{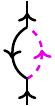}};
\endxy
\cong
\xy
(0,0)*{\includegraphics[scale=.65]{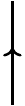}};
\endxy}_{\text{By }\eqref{eq:phantom-removal},\eqref{eq:phantom-squeezing}}
\end{gather}
as well as their horizontal mirrors and isomorphisms between isotopic webs. 

A consequence of the topological definitions is the existence of an \textit{adjunction} isomorphism 
of $\field$-modules $\Hom_{\field}(v,u)\cong\Hom_{\field}(\emptyset,uv^{\ast})$ 
for any webs $v$, $w$ via \textit{clapping}; hereby $uv^{\ast}$ 
denotes the closed web obtained by reversing the orientation of 
$v$ and gluing it on top of $u$:
\begin{gather}\label{eq:foam-clapping}
\xy
\xymatrix{
\xy
(0,0)*{\includegraphics[scale=.65]{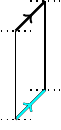}};
\endxy
\quad
\ar@{~>}[r]&
\quad
\raisebox{-.045cm}{\xy
(0,0)*{\includegraphics[scale=.65]{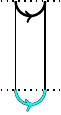}};
\endxy}
\quad
\ar@{~>}[r]&
\quad
\xy
(0,0)*{\includegraphics[scale=.65]{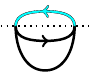}};
\endxy
}
\endxy
\end{gather}
In this example, the identity prefoam displayed on the left corresponds to the counit 
prefoam displayed in the right, see e.g. \cite[Lemma 3.7]{MPT} for a more detailed treatment. 
We denote $\Hom_\Fker(v,w):=\Hom_\Fker(\emptyset,wv^{\ast})$. Elements in these sets are 
called \textit{foams} from $v$ to $w$.
Composition of $f\in\Hom_\Fker(v,w)$ with $g\in\Hom_\Fker(w,z)$ is defined 
as $(\op{id}_{z}\otimes c\otimes\op{id}_{v^\ast})\circ (g\otimes \op{id}_{wv^\ast})\circ f $, 
where $c\colon w^\ast w\rightarrow\emptyset$ corresponds to the identity on $w^\ast$ by applying adjunction twice. One checks that this is well-defined and defines a 
category, the \textit{foam category}.

\begin{lemma}\label{lemma:web-eval}
In $\Fker$ and in the foam category, hom-spaces are free $\field$-modules of 
finite rank with endomorphism ring of $\emptyset$ isomorphic to $\field$.
\end{lemma}

\begin{proof}
Because of the relations in \fullref{lemma:moreandmore}, 
we know that $\End_{\Fker}(\emptyset)$ is 
a quotient of $\field$, since any closed foam equals an element of $\field$. 
Since $\mathsf{ev}(\mathrm{id}_{\emptyset})=1$, the last claim follows.
By clapping it suffices to
consider $\Hom_\Fker(\emptyset,w)$ for closed webs $w$ only.
Then \eqref{eq:web-eval} gives an isomorphism
from $\Hom_\Fker(\emptyset,w)$ to 
finitely many (possibly shifted) copies of $\End_\Fker(\emptyset)$. (This follows 
because of the arguments given in \cite[Section 6.A]{ETWi}.)
\end{proof}

Hence, the above constructs the functor $\sTQFT[\parameter]$, which maps to $\fmod$, and so \fullref{theorem:blanchet} is proved. 
\subsection{Foam \texorpdfstring{$2$}{2}-categories}\label{sub:genericfoam}
In this section we will define, generalizing \cite{EST}, a $2$-category which we call 
the $6$\textit{-parameter foam} $2$\textit{-category}

For the definition denote by $\Y^\pm$ the set of 
all $\vec{k}=(k_i)_{i\in\Z}\in\{0,{+}1,{+}2,{-}1,{-}2\}^{\Z}$ with $k_i=0$ 
for $|i|\gg 0$.  We view a web as a morphism between two such 
sequences, namely the 
\textit{orientation sequences} $\vec{k}$, $\vec{k}^{\prime}$ induced 
by the orientation of the web at the boundary 
points at the bottom and top extended to an infinite sequence by setting $k_i=0$ for $i<0$ and $i\gg 0$.

The absolute values of the $k_i$
are given by the labels of the web. The sign 
is $+$ in the source (=bottom) respectively target 
(=top) sequence if the web points out of the bottom 
respectively towards the top; and it is $-$ otherwise.

\begin{definition}\label{definition:foamcat}
Let $\F^\pm[\parameter]$ be the graded, additive, $\field$-linear $2$-category given by:
\begin{enumerate}[label=(\roman*)]

\item The objects are all $\vec{k}\in\Y^\pm$.

\item The $1$-morphisms and $2$-morphisms spaces are given 
by the foam category. More precisely, the $1$-morphisms 
between $\vec{k}$, $\vec{k}^{\prime}$ are the free $\field$-module with basis all webs whose 
pair of orientation sequences agrees with $(\vec{k}$, $\vec{k}^{\prime})$ up to infinitely many zeros. 

\item Composition of webs $v\circ u=uv$ is  
stacking $v$ on top of $u$, vertical 
composition of prefoams $g\circ f$ 
is stacking $g$ on top of $f$, horizontal 
composition $f\otimes g$ is putting $g$ 
to the right of $f$, whenever those 
operations make sense.
\end{enumerate}
\end{definition}

For $\vec{k}\in\Y^\pm$ with $\sum_{i\in\Z}|k_i|=2\ell$ 
consider the space of $1$-morphisms of cup webs
$\CUP(\vec{k})=\twoHom_{\F^\pm[\parameter]}(\emptyset,\vec{k})$ 
and then for $u,v\in\CUP(\vec{k})$ the 
space of $2$-morphisms $\twoHom_{\F^\pm[\parameter]}(\mathrm{id}_{\emptyset},uv^\ast)$. 
We recall a certain basis $\CUPbasis{u}{v}{\vec{k}}$, 
called the \textit{cup foam basis} for this $2$-hom space. 
Its formal definition is a bit technical
(but not hard) and can be found in \cite[Definition 4.12]{EST} 
(with the evident adaptions required for our setting). For 
our purposes it suffices to work with a more intuitive informal description: 
The isomorphisms from \eqref{eq:web-eval} induce 
isomorphisms given by 
foams as follows (including a right-handed version of the second isomorphism)
\begin{gather}\label{eq:cup-foam-basis-isos}
\xymatrix@C26mm{
\xy
(0,0)*{\includegraphics[scale=.65]{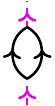}};
\endxy
\ar@<2pt>[r]^/-.5cm/{\left(\xy
(0,0)*{\includegraphics[scale=0.325]{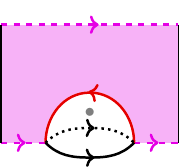}};
\endxy
\;,\;
\xy
(0,0)*{\includegraphics[scale=0.325]{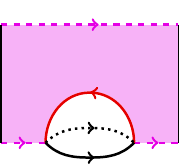}};
\endxy
\right)^{\mathrm{T}}}
&
\xy
(0,0)*{\includegraphics[scale=.65]{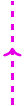}};
\endxy
\langle+1\rangle
\oplus
\raisebox{-.075cm}{$\xy
(0,0)*{\includegraphics[scale=.65]{figs/fig2-web-eval2.pdf}};
\endxy$}
\langle-1\rangle
\ar@<2pt>[l]^/.5cm/{\left(
\xy
(0,0)*{\includegraphics[scale=0.325]{figs/fig2-62}};
\endxy
\;,\;
\xy
(0,0)*{\includegraphics[scale=0.325]{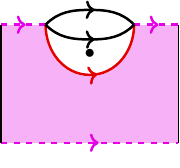}};
\endxy\right)^{\phantom{\mathrm{T}}}}
}
\quad,\quad
\xymatrix@C12mm{
\xy
(0,0)*{\includegraphics[scale=.65]{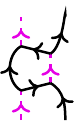}};
\endxy
\ar@<2pt>[r]^{
\xy
(0,0)*{\includegraphics[scale=0.325]{figs/fig5-squeeze1}};
\endxy}
&
\xy
(0,0)*{\includegraphics[scale=.65]{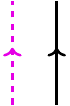}};
\endxy
\ar@<2pt>[l]^{
\xy
(0,0)*{\includegraphics[scale=0.325]{figs/fig5-squeeze2}};
\endxy}
}
\end{gather}
We fix these isomorphisms and additionally a 
\textit{base point} $\blacklozenge$ on each circle 
in $uv^{\ast}$ with maximal $x$ coordinate and then the 
lowest $y$ coordinate, where we see $uv^{\ast}$ as 
being embedded in the $xy$-plane. We moreover fix an 
\textit{evaluation of} $uv^{\ast}$, i.e. 
a sequence of isomorphisms as in \eqref{eq:cup-foam-basis-isos} 
which recursively 
reduce $uv^{\ast}$ to copies of $\emptyset$. This 
evaluation has to be chosen such that the segments 
of $uv^{\ast}$ containing the 
base points are only removed using the leftmost 
isomorphism in \eqref{eq:cup-foam-basis-isos}. 
We can then define the cup foam basis 
inductively by reading the corresponding 
isomorphisms from 
$\emptyset$ to $uv^{\ast}$.
(The reader might think about cup foam basis elements as ``cups'' in the 
naive sense of the word.)

We summarize the main properties in the following lemma 
which we will use throughout, see \cite[Section~2 and Definition~4.12]{EST} for a proof.
For $\vec{k}\in\Y^\pm$ with $\sum_{i\in\Z}|k_i|=2\ell$, 
define the \textit{shift} $d(\vec{k})=\ell-\sum_{i\in\Z}|k_i|(|k_i|-1)$.
Note that for $\vec{k}\in\Y^\pm$ with $k_i\neq{\pm}2$ for all $i$ we have $d(\vec{k})=\ell$.

\begin{lemmaqed}(Cup foam basis)
\label{lemma:it-is-a-basis}

\begin{enumerate}[label=(\alph*)]
\item Given two webs $u,v$, let $\clapping({}_-)$ denote their clappings, cf. the middle picture in \eqref{eq:clapping}.
Then there are isomorphisms of graded $\field$-modules
\begin{align*}
&\,\twoHom_{\F^\pm[\parameter]}(v,u)
\cong
\twoHom_{\F^\pm[\parameter]}(\clapping(v),\clapping(u))
\cong\,
\twoHom_{\F^\pm[\parameter]}
(\emptyset,\clapping(u)\clapping(v)^{\ast})\langle d(\vec{k})\rangle.
\end{align*} 

\item The set $\CUPbasis{u}{v}{\vec{k}}$ is a homogeneous, $\field$-linear 
basis of $\twoHom_{\F^\pm[\parameter]}(\emptyset,uv^{\ast})$ 
for any $u,v\in\CUP(\vec{k})$. It descends also to a basis for any specialization 
of the parameters.\qedhere
\end{enumerate}
\end{lemmaqed}

Note that (b) ensures that (almost) all statements which 
we are going to make hold verbatim for any specialization of parameters. 
We will stop making this comment.
\subsection{Known specializations}\label{sub:special}
Let $R=\Z[\parto,\parha]$ in case $\Ca$, $R=\Z[i]$ in case 
$\CMW$, and $R=\Z$ in cases $\KBN$ or $\Bl$.
The following $2$-categories recover the 
setups from \eqref{eq:main-specs}.

\begin{definition}\label{definition-several-2-cats}
We define four graded, additive, $R$-linear $2$-categories  
$\F^\pm[\KBN]$, $\F^\pm[\Ca]$, $\F^\pm[\CMW]$ and $\F^\pm[\Bl]$
as in \fullref{definition:foamcat} 
except for the following differences.
\begin{enumerate}[label=(\roman*)]

\item The parameters are as in \eqref{eq:main-specs}.

\item As $1$-morphisms one has ``webs'' locally given as
\begin{gather}\label{eq:as-webs}
\xy
(0,0)*{\includegraphics[scale=.65]{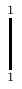}
\;\raisebox{.55cm}{,}\;
\raisebox{.5cm}{$\emptyset$}
\;\raisebox{.55cm}{,}\;
\includegraphics[scale=.65]{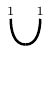}
\;\raisebox{.55cm}{,}\;
\includegraphics[scale=.65]{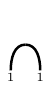}};
(0,-6.5)*{\text{{\tiny in case $\KBN$}}};
\endxy
\;;\;
\xy
(0,0)*{\includegraphics[scale=.65]{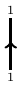}
\;\raisebox{.55cm}{,}\;
\raisebox{.5cm}{$\emptyset$}
\;\raisebox{.55cm}{,}\;
\includegraphics[scale=.65]{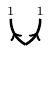}
\;\raisebox{.55cm}{,}\;
\includegraphics[scale=.65]{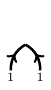}};
(0,-6.5)*{\text{{\tiny in case $\Ca$}}};
\endxy
\;;\;
\xy
(0,0)*{\includegraphics[scale=.65]{figs/fig2-1.pdf}
\;\raisebox{.55cm}{,}\;
\raisebox{.5cm}{$\emptyset$}
\;\raisebox{.55cm}{,}\;
\includegraphics[scale=.65]{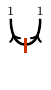}
\;\raisebox{.55cm}{,}\;
\includegraphics[scale=.65]{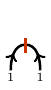}};
(0,-6.5)*{\text{{\tiny in case $\CMW$}}};
\endxy
\;;\;
\xy
(0,0)*{\includegraphics[scale=.65]{figs/fig2-1.pdf}
\;\raisebox{.55cm}{,}\;
\includegraphics[scale=.65]{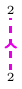}
\;\raisebox{.55cm}{,}\;
\includegraphics[scale=.65]{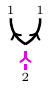}
\;\raisebox{.55cm}{,}\;
\includegraphics[scale=.65]{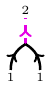}};
(0,-6.5)*{\text{{\tiny in case $\Bl$}}};
\endxy
\end{gather}

\item The $2$-morphisms   
are those defined as in the proof of \fullref{proposition:special} below.
\end{enumerate}
\end{definition}

\begin{proposition}\label{proposition:special}
There are functors of graded, additive, $R$-linear $2$-categories
\begin{gather}
\begin{gathered}
\F^\pm[0,0,1,1,1,1]{\longrightarrow}\F^\pm[\KBN],\quad\quad
\F^\pm[\parto,\parha,1,1,i,-i]{\longrightarrow}\F^\pm[\Ca],
\\
\F^\pm[0,0,1,1,i,-i]\stackrel{\cong}{\longrightarrow}\F^\pm[\CMW],\quad\quad
\F^\pm[0,0,1,-1,1,-1]\stackrel{\cong}{\longrightarrow}\F^\pm[\Bl],
\end{gathered}
\end{gather}
which are surjective on $0$, $1$ and $2$-morphisms; and even equivalences in the last two cases.
\end{proposition}

\begin{proof}
The case $\Bl$ follows directly, since the specialization is the $2$-category in \cite{Bla}. 
For the other cases we first need to define the $2$-categories 
in question on the level of $2$-morphisms 
and then set up the $2$-functors 
which provide the quotient functors. 

The $2$-morphisms of $\F^\pm[\KBN]$ are 
$R$-linear combinations of prefoams with ordinary parts only, modulo the relations
from \eqref{eq:theusualrelations1}-\eqref{eq:theusualrelations2} involving ordinary cobordisms with the 
specialization of the parameters as in \eqref{eq:main-specs}. We obtain a 
functor $\F^\pm[0,0,1,1,1,1]{\longrightarrow}\F^\pm[\KBN]$. 
that annihilates all $0$, $1$ and $2$-morphisms which involve a phantom piece. 
It is clearly full on $2$-morphisms. It induces in fact an equivalence 
of $2$-categories from the full $2$-subcategory involving no phantom parts to 
$\F^\pm[\KBN]$ thanks to  \fullref{lemma:it-is-a-basis} and the corresponding 
basis theorem in $\F^\pm[\KBN]$, see e.g. \cite[Section 9.1]{BN1}.

The $2$-morphisms of $\F^\pm[\Ca]$ are 
$R$-linear combinations of the 
topological CW complexes obtained from prefoams by removing the phantom edges 
and phantom facets, but keeping the singular seams 
and the topological relations induced from prefoams, see \cite{Cap2}, e.g.
\begin{gather}\label{eq:some-2-functors-more}
\xy
(0,0)*{\includegraphics[scale=.65]{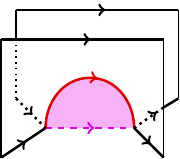}};
\endxy
\;\longmapsto\;
\xy
(0,0)*{\includegraphics[scale=.65]{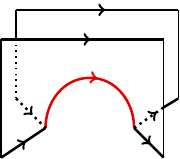}};
\endxy
\end{gather}
We obtain a well-defined forgetful functor $\F^\pm[\parto,\parha,1,1,i,-i]{\longrightarrow}\F^\pm[\Ca]$ by 
identifying our $\alpha$ and $\beta$ with the variables $a$ and $h$ in \cite{Cap2} 
(the relations in \cite{Cap2} should in particular be 
compared with our dot removal and migration and sphere relations). 
The functor just forgets all phantom parts and only remembers the singular seams. It is clearly 
surjective on each level. 
It is also faithful by \fullref{lemma:it-is-a-basis} and by 
the corresponding basis theorem in $\F^\pm[\Ca]$ which can be proven analogously 
to \fullref{lemma:it-is-a-basis}. In particular, $0$ and $1$-morphisms might be identified via this functor, but not $2$-morphisms.

The $2$-morphisms of $\F^\pm[\CMW]$ can be defined as
the quotient of $\F^\pm[0,0,1,1,i,-i]$ 
modulo an extra relation, which in the notation of \cite{CMW} is
\begin{gather}\label{eq:disor}
\xy
(0,0)*{\includegraphics[scale=.65]{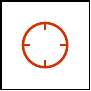}};
\endxy
=
i\cdot
\xy
(0,0)*{\includegraphics[scale=.65]{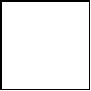}};
\endxy
\quad,\quad
\xy
(0,0)*{\includegraphics[scale=.65]{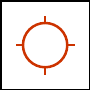}};
\endxy
=
-i\cdot
\xy
(0,0)*{\includegraphics[scale=.65]{figs/fig2-73.pdf}};
\endxy
\end{gather}
The identification of this definition with the one from 
\cite{CMW} is given by removing the phantom edges 
and phantom facets, and then replacing the singular seams by disorientation 
lines. The orientation of the seam induces the direction of the 
disorientation line as follows.
\begin{gather}\label{eq:example-cmw}
\xy
(0,0)*{\includegraphics[scale=.65]{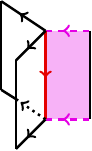}};
\endxy
\mapsto
\xy
(0,0)*{\includegraphics[scale=.65]{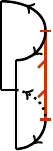}};
\endxy
\quad,\quad
\xy
(0,0)*{\includegraphics[scale=.65]{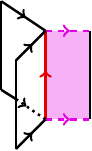}};
\endxy
\mapsto
\xy
(0,0)*{\includegraphics[scale=.65]{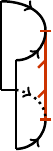}};
\endxy
\quad,\quad
\xy
(0,0)*{\includegraphics[scale=.65]{figs/fig2-60.pdf}};
\endxy
\;\mapsto\;
\xy
(0,0)*{\includegraphics[scale=.65]{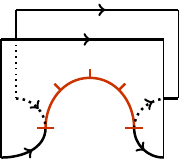}};
\endxy
\end{gather}
Here we also displayed the special case of the saddle on the right. 
According to the above description it is obtained from a 
prefoam displayed next to it by removing the phantom face and replacing the singular seams 
by a disorientation line.
We point out that, for the choice of parameters $\parom=-\parome=i$, the relation in \eqref{eq:disor} is actually \eqref{eq:phantom-removal}. In particular, we can just define the functor $\F^\pm[0,0,1,1,i,-i]{\longrightarrow}\F^\pm[\CMW]$ in the evident way. It is then an equivalence by its very definition.

The grading is always given by \eqref{eq:foam-degree}, i.e. induced by 
the topological Euler characteristic. In particular, 
disorientation lines do not change the degree. By construction 
all the functors are grading preserving, with $a$ and $h$ in \cite{Cap2} 
in the same degrees as $\alpha$ and $\beta$.
\end{proof}
\subsection{Web algebras}\label{subsection:web-algebras}
We define, following \cite{EST}, a $6$-parameter web algebra which is an algebraic version 
$\webalg^\pm[\parameter]$ of $\F^\pm[\parameter]$. The only difference to \cite[Sections 2.4 and 2.5]{EST}  
is that we work over $\field$ and use the foams defined in \fullref{definition:foamcat}. 

Denote by $\bY^\pm\subset\Y^\pm$ the set of all \textit{balanced} tuples, i.e. 
their components sum up to an even number. 
The $6$-\textit{parameter web algebra} is the graded $\field$-module
\begin{gather}\label{eq:webalg}
\webalg^\pm[\parameter]_{\vec{k}}=
{\textstyle\bigoplus_{u,v\in\CUP(\vec{k})}}
{}_u(\webalg^\pm[\parameter]_{\vec{k}})_v,\quad\;\;
\webalg^\pm[\parameter]=
{\textstyle\bigoplus_{\vec{k}\in\bY^\pm}}
\webalg^\pm[\parameter]_{\vec{k}},
\end{gather}
where ${}_u(\webalg^\pm[\parameter]_{\vec{k}})_v=
\twoHom_{\F^\pm[\parameter]}
(\emptyset,uv^{\ast})\langle d(\vec{k})\rangle$
with the grading given by the foam degree from \eqref{eq:foam-degree}, and the multiplication 
\begin{gather}\label{eq:multfoams}
\boldsymbol{\mathrm{Mult}}\colon
\webalg^\pm[\parameter]_{\vec{k}}\otimes\webalg^\pm[\parameter]_{\vec{k}}\rightarrow\webalg^\pm[\parameter]_{\vec{k}},\;f\otimes g\mapsto \boldsymbol{\mathrm{Mult}}(f,g)=fg
\end{gather} 
defined as in \cite{EST} using multiplication foams.
It is a nice immediate fact that web algebras are graded, associative, unital $\field$-algebras, which are 
free as $\field$-modules.

Given any web $u\in\Hom_{\F^\pm[\parameter]}(\vec{k},\vec{l})$ 
such that the sum $\vec{k}+\vec{l}$ of the boundaries is contained in $\bY$, 
we consider the $\webalg^\pm[\parameter]$-bimodule
\begin{gather}\label{eq:the-web-bimodules}
\M^\pm[\parameter](u)=
{\textstyle\bigoplus_{\substack{v\in\CUP(\vec{k}),\\w\in\CUP(\vec{l})}\phantom{,}}}
\twoHom_{\F^\pm[\parameter]}(\emptyset,vuw^{\ast})
\end{gather}
with left (top) and 
right (bottom) action of $\webalg^\pm[\parameter]$ given by the multiplication foams. 
We call all such $\webalg^\pm[\parameter]$-bimodules 
$\M^\pm[\parameter](u)$ \textit{web bimodules}.
These have a \textit{cup foam basis} $\CUPBasis{u}$ 
by considering 
all webs $vuw^{\ast}$ for $v\in\CUP(\vec{k}),w\in\CUP(\vec{l})$. The following summarizes the main properties of these bimodules:

\begin{propositionqed}(Cf. \cite[Sections 2.4 and 2.5]{EST}.)\label{proposition:webbimodules2}
Let $u\in\Hom_{\F^\pm[\parameter]}(\vec{k},\vec{l})$.

\begin{enumerate}[label=(\alph*)]

\item $\M^\pm[\parameter](u)$ are graded, biprojective  
$\webalg^\pm[\parameter]$-bimodules which are free as $\field$-modules
such that the summands in \eqref{eq:the-web-bimodules} are of finite rank for all pairs $v,w$.

\item The set $\CUPBasis{u}$ is a homogeneous, $\field$-basis of $\M^\pm[\parameter](u)$.

\item $\CUPBasis{u}$ is also a basis for any specialization 
of the parameters.\qedhere

\end{enumerate} 
\end{propositionqed}

\begin{definition}\label{definition:catbimodulesweb}
Let $\webcatpm$ be the additive closure of the following graded, $\field$-linear $2$-category:
\begin{enumerate}[label=(\roman*)]

\item Objects are the various $\vec{k}\in\bY^\pm$.

\item $1$-morphisms are tensor products 
(taken over the algebra $\webalg^\pm[\parameter]$) of the
$\webalg^\pm[\parameter]$-bimodules $\M^\pm[\parameter](u)$.

\item $2$-morphisms are $\webalg^\pm[\parameter]$-bimodule homomorphisms.

\item The composition of web bimodules is 
the tensor product ${}_-\otimes_{\webalg^\pm[\parameter]}{}_-$.
The vertical composition of $\webalg^\pm[\parameter]$-bimodule homomorphisms is 
the usual composition. The horizontal composition 
is given by tensoring over $\webalg^\pm[\parameter]$.
\end{enumerate}
\end{definition}

This $2$-category provides a faithful $2$-representation 
of the $2$-category $\F^\pm[\parameter]$:

\begin{proposition}\label{proposition:cats-are-equal}
There is an embedding of graded, additive, $\field$-linear $2$-categories
\begin{gather}\label{eq:upsilon}
\Upsilon\colon\!\F^\pm[\parameter]\hookrightarrow\webcatpm,
\end{gather}
which is bijective on objects and essential surjective and faithful on $1$-morphisms.
\end{proposition}

\begin{proof}
Define $\Upsilon$ by additively extending the following assignments:
\begin{enumerate}[label=(\roman*)]

\item On objects $\vec{k}$ we set $\Upsilon(\vec{k})=\vec{k}$.

\item On $1$-morphisms 
$u\in\Hom_{\F^\pm[\parameter]}(\vec{k},\vec{l})$ we set $\Upsilon(u)=\M^\pm[\parameter](u)$.

\item On $2$-morphisms 
$f\in\twoHom_{\F^\pm[\parameter]}(u,v)$ we set 
$\Upsilon(f)\colon\M^\pm[\parameter](u)\to\M^\pm[\parameter](v)$ given 
by stacking $f$ on top of the elements of $\M^\pm[\parameter](u)$.
\end{enumerate}
Note that $\Upsilon(f)$ is a $\webalg^\pm[\parameter]$-bimodule homomorphism, by the arguments 
from \cite[Section 2.7]{Khov}. Since $\Upsilon$ is bijective on objects,
it remains to show that $\Upsilon$ is 
essential surjective and faithful on $1$-morphisms 
(which is however just a parameter adapted version of \cite[Theorem 1]{Khov}) 
and faithful (which follows directly from the existence of the cup foam basis) on $2$-morphisms.
Clearly, $\Upsilon$ is degree preserving.
The statement follows.
\end{proof}
\subsection{Upwards oriented webs}\label{subsec:upwards}
Let $\Y\subset\Y^\pm$ be the subset of $\vec{k}$'s without 
entries of the form ${-}1,{-}2$. We call a web \textit{upward oriented}
if it is embedded in $\R\times[-1,1]$, 
with its boundary in $\R\times\{-1,1\}$
such that for each generic horizontal slice 
its orientations point upwards through this slice.

\begin{definition}\label{definition:foamcat-next}
Let $\F[\parameter]$ be the full $\field$-linear $2$-subcategory of 
$\F^{\pm}[\parameter]$ given by the objects in $\Y$, all
upward oriented webs and the full space of $2$-morphisms.
\end{definition}

\begin{remark}\label{remark:not-serious}
We sometimes restrict ourselves to upwards oriented webs and their foams. Via clapping
\begin{gather}\label{eq:clapping}
\xy
(0,0)*{\includegraphics[scale=.65]{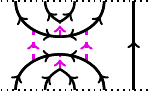}};
(0,-8)*{\text{\tiny$u$}};
\endxy
\quad,\quad
\xy
(0,0)*{\includegraphics[scale=.65]{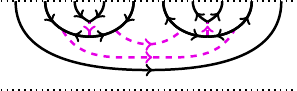}
\quad\raisebox{.5cm}{or}\quad
\includegraphics[scale=.65]{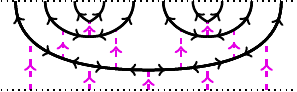}};
(0,-8)*{\text{\tiny$\clapping(u)$}};
\endxy
\end{gather}
this is not a serious restriction. In fact, a bit more work 
establishes all of our results for general webs and their foams, 
see e.g. \cite{BHPW} or \cite{ETWi}. 
Clapping justifies to think of foams between webs which have only 
phantom edges as being ``closed''.
\end{remark}

When working in the upwards oriented setup we use the same 
conventions as before, but omit the superscript ${}^{\pm}$. 
We leave it to the reader to transfer the results established so far in the evident way to this setting.
\subsection{Arc algebras}\label{subsec:arc-algebras}
Following \cite[Section 3]{EST}, we now define a parameter version of arc algebras, using the 
parameters $\parameterS=(\parto,0,1,\parsign\paro^2,\paro,\parsign\paro)$ 
over the ground ring $\ring=\Z[\parto,\parsign,\paro^\pm]$. The only difference 
to \cite{EST} is the multiplication, which we will therefore explain.
The crucial difference compared to web algebras is that arc algebras 
come with a nice graded basis and an explicit combinatorial multiplication, but e.g. 
some basic natural properties like the associativity is a priori not clear from the algebraic definition of the multiplication, and is often proved by invoking a TQFT as e.g. in \cite[Section 3]{BS1}. 

Recall from \cite[Section 3]{EST} \textit{weights}, i.e.
$\lambda=(\lambda_i)_{i\in\mathbb{Z}}$ 
with $\lambda_i\in\{\circ,\times,\down,\up\}$, 
such that $\lambda_i=\circ$ for $|i|\gg0$, 
\textit{(balanced) blocks} $\bblock\subset\block$, 
and \textit{block sequences} 
$\mathrm{seq}(\Lambda) = ({\rm seq}(\Lambda)_i)_{i \in \mathbb{Z}}$ 
by taking any $\lambda\in\Lambda$ and replacing the symbols 
$\up,\down$ by $\dummy$.

For each $\lambda\in\block$ 
there is associated a \textit{cup}
$\underline{\lambda}$ and \textit{cap diagram} $\overline{\lambda}$. These cup $c$ and cap $d^{\ast}$ diagrams can be oriented using 
the entries from a weight. We denote these oriented cup and cap diagrams by $c\lambda$ and $\lambda d^{\ast}$.
We also have the notion of (oriented) circles $C$ in $cd^{\ast}$.

For fixed $\Lambda\in\block$, let
$\CupBasis{\Lambda}=\{\underline{\lambda}
\nu\overline{\mu}\mid\underline{\lambda}\nu\overline{\mu}\text{ is oriented and } 
\lambda,\mu,\nu\in\Lambda\}$ 
and its subset $\Cupbasis{\lambda}{\mu}{\Lambda}$, where we fix $\lambda$ and $\mu$, containing $\CUPbasis{\lambda}{\mu}{\Lambda}$, where we additionally require that $\lambda$ and $\mu$ are balanced. We also have a degree given by
\begin{gather}\label{eq:degreegenKh}
\begin{aligned}
\mathrm{deg}(c\lambda)=\#\{\text{clockwise cups in } c\lambda\}&,\;
\mathrm{deg}(\lambda d^{\ast})=\#\{\text{clockwise caps in } \lambda d^{\ast}\},\\
\mathrm{deg}(c\lambda d^{\ast})=&\,\mathrm{deg}(c\lambda)+\mathrm{deg}(\lambda d^{\ast}).
\end{aligned}
\end{gather}
We call elements of these sets \textit{basis cup diagrams}.

Denote by ${}_\lambda(\Arcalg_\Lambda)_\mu$ the 
$\ring$-linear span of the basis vectors 
from $\CUPbasis{\lambda}{\mu}{\Lambda}$.
The \textit{signed} $2$-\textit{parameter arc algebras} are defined as the 
graded (using \eqref{eq:degreegenKh}), free $\ring$-modules
\begin{gather}\label{definition:genericarcalgebra}
\Arcalg_\Lambda =
{\textstyle\bigoplus_{\underline{\lambda}
\nu \overline{\mu} \in \CUPBasis{\Lambda}}}
{}_\lambda(\Arcalg_\Lambda)_\mu,
\quad\quad
\Arcalg=
{\textstyle\bigoplus_{\Lambda\in\bblock}}
\Arcalg_\Lambda,
\end{gather}
with multiplication given as follows.
Up to coefficients, the multiplication is the same as in \cite[Section 3.3]{EST} 
and we only focus on explaining these coefficients. In fact, \cite[Section 3.3]{EST} is the specialization 
$(0,0,1,-1,1,-1)$ in the notation of this paper.

Recall the \textit{position of} $i$ denoted 
by $\pos_\Lambda(i)$, the \textit{distance} $\length_\Lambda(\gamma)$, 
the \textit{saddle width} $s_\Lambda(\gamma)$ an the \textit{length} 
$\length_\Lambda(\gamma_1\cdots\gamma_r)$ of a 
sequence of arcs.
The signs can 
be divided into the \textit{dot moving signs}, 
the \textit{topological sign} 
and the \textit{saddle sign}.
These signs are as follows (explained for each case in detail below).
\begin{gather}\label{eq:thesigns}
\begin{aligned}
\textit{Dot moving }& \textit{signs: }\parsign^{\length_\Lambda(\gamma_i^{\rm dot})}\quad\text{and}\quad
\parsign^{\length_\Lambda(\gamma_i^{\rm ndot})}.\\
\textit{Topological sign: }& \parsign^{\neatfrac{1}{4}(\length_\Lambda(C_{\rm in})-2)}.\quad\quad
\textit{Saddle sign: } \parsign^{s_\Lambda(\gamma)}.
\end{aligned}
\end{gather}
The dot moving signs can appear in any situation, the topological 
sign will appear for nested merges and splits, and the saddle sign for 
nested merges and non-nested splits.
Each case can pick up some 
extra factors $\parto$, $\parsign$ 
or $\paro$ as we are going to describe now.

\noindent\textbf{Non-nested Merge.} The non-nested 
circles $C_i$ and $C_j$---containing vertices $i$ 
respectively $j$---are merged into $C$. 
The cases from above are modified as follows.

\noindent\textit{$\succ$ Both circles oriented anticlockwise.} 
As in \cite[Section 3.3]{EST} 
(no extra coefficients).

\noindent\textit{$\succ$ One circle $C_i$
oriented clockwise, one oriented anticlockwise.} 
As in \cite[Section 3.3]{EST}, and coefficient 
\begin{gather}\label{eq:dotsign}
\parsign^{\length_\Lambda(\gamma_i^{\rm dot})}. 
\end{gather}

\noindent\textit{$\succ$ Both circles $C_i,C_j$ oriented clockwise.}  
This case is not mentioned in \cite[Section 3.3]{EST}, since 
the result is $0$ therein. 
However, the same rules apply replacing this 
$0$ with the diagram being oriented anticlockwise and 
coefficient
\begin{gather}\label{eq:doubledotsign}
\parto\cdot\parsign^{\length_\Lambda(\gamma_i^{\rm dot})}\cdot 
\parsign^{\length_\Lambda(\gamma_j^{\rm dot})}. 
\end{gather}

\noindent\textbf{Nested Merge.} The nested 
circles $C_i$ and $C_j$ (with notation as before) are merged into $C$. Denote 
by $C_{\rm in}$ the inner of the two original circles. 
Then:

\noindent\textit{$\succ$ Both circles oriented anticlockwise.} 
As in \cite[Section 3.3]{EST} with coefficient
\begin{gather}\label{eq:do-not-forget-the-saddle}
\parsign \cdot \parsign^{\neatfrac{1}{4}(\length_\Lambda(C_{\rm in})-2)}\cdot \parsign^{s_\Lambda(\gamma)}.
\end{gather}

\noindent\textit{$\succ$ One circle 
oriented clockwise, one oriented anticlockwise.} 
The coefficient is
\begin{gather}\label{eq:do-not-forget-the-saddle-2}
\parsign\cdot\parsign^{\length_\Lambda(\gamma_k^{\rm dot})}\cdot\parsign^{\neatfrac{1}{4}(\length_\Lambda(C_{\rm in})-2)}\cdot \parsign^{s_\Lambda(\gamma)}.
\end{gather}

\noindent\textit{$\succ$ Both circles oriented clockwise.} 
As above, this case does not appear in \cite[Section 3.3]{EST} 
since the result is $0$. However, 
we again use the diagram with the anticlockwise orientation and coefficient
\begin{gather}\label{eq:do-not-forget-the-saddle-3}
\parto\cdot\parsign \cdot\parsign^{\length_\Lambda(\gamma_i^{\rm dot})}\cdot \parsign^{\length_\Lambda(\gamma_j^{\rm dot})} \cdot \parsign^{\neatfrac{1}{4}(\length_\Lambda(C_{\rm in})-2)}\cdot \parsign^{s_\Lambda(\gamma)}.
\end{gather}

\noindent\textbf{Non-nested Split.} The circle $C$ 
splits into the non-nested circles $C_i$ and 
$C_j$---containing vertices $i$ respectively $j$.

\noindent\textit{$\succ$ $C$ oriented anticlockwise.}
The coefficients are now
\begin{gather}\label{eq:new-dot1}
\paro\cdot\parsign^{\length_\Lambda(\gamma_i^{\rm ndot})}\cdot\parsign^{s_\Lambda(\gamma)},
\end{gather}
while the one where $C_j$ is oriented clockwise is multiplied with 
\begin{gather}\label{eq:new-dot2}
\parsign\cdot\paro\cdot
\parsign^{\length_\Lambda(\gamma_j^{\rm ndot})}\cdot\parsign^{s_\Lambda(\gamma)}.
\end{gather}

\noindent\textit{$\succ$ $C$ oriented clockwise.}
Multiply the copy with both circles oriented clockwise by 
\begin{gather}\label{eq:new-dot3}
\paro\cdot 
\parsign^{\length_\Lambda(\gamma_j^{\rm dot})}\cdot
\parsign^{\length_\Lambda(\gamma_i^{\rm ndot})}\cdot\parsign^{s_\Lambda(\gamma)}
\end{gather}
and the copy with both circles oriented anticlockwise by
\begin{gather}\label{eq:new-dot4}
\parto\cdot\parsign\cdot\paro\cdot
\parsign^{\length_\Lambda(\gamma_j^{\rm dot})}
\cdot\parsign^{\length_\Lambda(\gamma_j^{\rm ndot})}\cdot\parsign^{s_\Lambda(\gamma)}.
\end{gather}

\noindent\textbf{Nested Split.} We use here the same 
notations as in the non-nested split case, and 
we denote by $C_{\rm in}$ and $C_{\rm out}$ the inner and outer 
of the two circles $C_i$ and $C_j$.

\noindent\textit{$\succ$ $C$ oriented anticlockwise.}
As for the non-nested split, but 
the copy where $C_{\rm in}$ is oriented clockwise is multiplied with
\begin{gather}\label{eq:new-dot5}
\paro \cdot \parsign^{\neatfrac{1}{4}(\length_\Lambda(C_{\rm in})-2)},
\end{gather}
while the copy where $C_{\rm out}$ is oriented clockwise is multiplied with 
\begin{gather}\label{eq:new-dot6}
\parsign \cdot \paro \cdot \parsign^{\neatfrac{1}{4}(\length_\Lambda(C_{\rm in})-2)}.
\end{gather}

\noindent\textit{$\succ$ $C$ oriented clockwise.}
Multiply the copy with both circles oriented clockwise by
\begin{gather}\label{eq:new-dot7}
\paro \cdot \parsign^{\neatfrac{1}{4}(\length_\Lambda(C_{\rm in})-2)},
\end{gather}
and the one with both circles oriented anticlockwise by
\begin{gather}\label{eq:new-dot8}
\parto \cdot \parsign \cdot \paro \cdot \parsign^{\neatfrac{1}{4}(\length_\Lambda(C_{\rm in})-2)}.
\end{gather}

\begin{example}\label{ex:multiplication}
In a simple, non-nested merge we have no coefficients at all:
\begin{gather}\label{eq:mult1}
\xy
\xymatrix{
\raisebox{.08cm}{\xy
(0,0)*{\includegraphics[scale=.65]{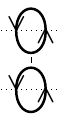}};
\endxy}
\ar@{|->}[r]
&
\raisebox{.08cm}{\xy
(0,0)*{\includegraphics[scale=.65]{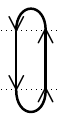}};
\endxy}
\bigg(
\ar@{|->}[r]
&
\raisebox{.08cm}{\xy
(0,0)*{\includegraphics[scale=.65]{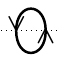}};
\endxy}
\bigg)
}
\endxy
\end{gather}
The rightmost step above, 
called \textit{collapsing}, is always performed 
at the end of a multiplication procedure and is 
not displayed in what follows.
\end{example}

\begin{example}\label{ex:multiplication2}
Here is an example where coefficients occur.
Consider a merge of two 
anticlockwise, nested circles:
\begin{gather}\label{eq:mult2}
\begin{aligned}
\xy
\xymatrix@C-10pt{
\raisebox{.08cm}{\xy
(0,0)*{\includegraphics[scale=.65]{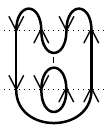}};
\endxy}
\ar@{|->}[r]
&
\raisebox{.08cm}{\xy
(0,0)*{\includegraphics[scale=.65]{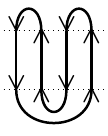}};
\endxy}
}
\endxy
\quad\text{and}\quad
\xy
\xymatrix@C-10pt{
\raisebox{-0.1cm}{
\raisebox{.08cm}{\xy
(0,0)*{\includegraphics[scale=.65]{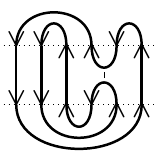}};
\endxy}
}
\ar@{|->}[r]
&
\parsign
\raisebox{-0.1cm}{
\raisebox{.08cm}{\xy
(0,0)*{\includegraphics[scale=.65]{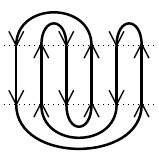}};
\endxy}
}
}
\endxy
\end{aligned}
\end{gather}
Here we have $s_\Lambda(\gamma)=1$, but 
$\neatfrac{1}{4}(\length_\Lambda(C_{\rm in})-2)=0$ for the left multiplication step and 
$\neatfrac{1}{4}(\length_\Lambda(C_{\rm in})-2)=1$ for the right multiplication step. 
\end{example}

\begin{example}
In both examples given now a non-nested merge is performed, followed by a 
split into two non-nested 
respectively nested circles. First, the H-shape:
\begin{gather}\label{eq:mult3}
\begin{aligned}
&\xy
\xymatrix{
\raisebox{.08cm}{\xy
(0,0)*{\includegraphics[scale=.65]{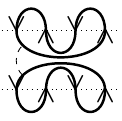}};
\endxy}
\ar@{|->}[r]
&
\raisebox{.08cm}{\xy
(0,0)*{\includegraphics[scale=.65]{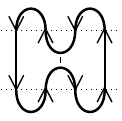}};
\endxy}
\ar@{|->}[r]
&
\parsign\paro
\raisebox{.08cm}{\xy
(0,0)*{\includegraphics[scale=.65]{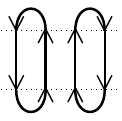}};
\endxy}
+\parsign^3\paro
\raisebox{.08cm}{\xy
(0,0)*{\includegraphics[scale=.65]{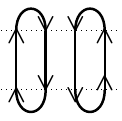}};
\endxy}
}
\endxy \\
&\xy
\xymatrix{
\raisebox{.08cm}{\xy
(0,0)*{\includegraphics[scale=.65]{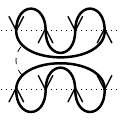}};
\endxy}
\ar@{|->}[r]
&
\raisebox{.08cm}{\xy
(0,0)*{\includegraphics[scale=.65]{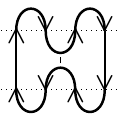}};
\endxy}
\ar@{|->}[r]
&
\parsign\paro 
\raisebox{.08cm}{\xy
(0,0)*{\includegraphics[scale=.65]{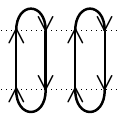}};
\endxy}
+ \parto\parsign^3\paro
\raisebox{.08cm}{\xy
(0,0)*{\includegraphics[scale=.65]{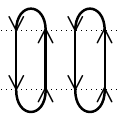}};
\endxy}
}
\endxy
\end{aligned}
\end{gather}
Here we have $s_\Lambda(\gamma)=1$, but $\length_\Lambda(\gamma_1^{\rm ndot})=0$ 
and $\length_\Lambda(\gamma_2^{\rm ndot})=1$. 
Moreover, $\length_\Lambda(\gamma_2^{\rm dot})=0$ in the bottom case.
Next, the C shape.
\begin{gather}\label{eq:mult4}
\begin{aligned}
&\xy
\xymatrix{
\raisebox{.08cm}{\xy
(0,0)*{\includegraphics[scale=.65]{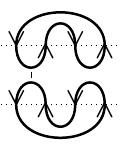}};
\endxy}
\ar@{|->}[r]
&
\raisebox{.08cm}{\xy
(0,0)*{\includegraphics[scale=.65]{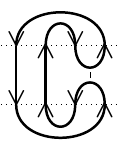}};
\endxy}
\ar@{|->}[r]
&
\paro
\raisebox{.08cm}{\xy
(0,0)*{\includegraphics[scale=.65]{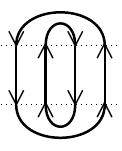}};
\endxy}
+ 
\parsign\paro
\raisebox{.08cm}{\xy
(0,0)*{\includegraphics[scale=.65]{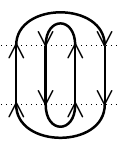}};
\endxy}
}
\endxy \\
&\xy
\xymatrix{
\raisebox{.08cm}{\xy
(0,0)*{\includegraphics[scale=.65]{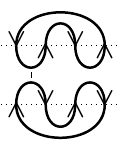}};
\endxy}
\ar@{|->}[r]
&
\raisebox{.08cm}{\xy
(0,0)*{\includegraphics[scale=.65]{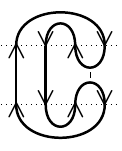}};
\endxy}
\ar@{|->}[r]
&
\paro
\raisebox{.08cm}{\xy
(0,0)*{\includegraphics[scale=.65]{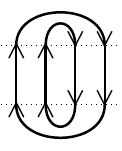}};
\endxy}
+
\parto\parsign\paro
\raisebox{.08cm}{\xy
(0,0)*{\includegraphics[scale=.65]{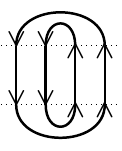}};
\endxy}
}
\endxy
\end{aligned}
\end{gather}
Here $\neatfrac{1}{4}(\length_\Lambda(C_{\rm in})-2)=0$, 
and again $\length_\Lambda(\gamma_2^{\rm dot})=0$ for the bottom.
\end{example}

\begin{remark}\label{remark:noCs}
The $\reflectbox{\text{C}}$ shape cannot appear 
as long as we impose the choice of the 
order of cup-cap pairs from left to right in the surgery 
procedure.
\end{remark}

Finally, recall the notion of an arc bimodule \cite[Section 3.4]{EST} which 
is a a graded, free $\ring$-module supported on the set 
$\CUPBasis{\compmatch,\compmatcht}$ of 
$\compmatch=(\Lambda_{0},\dots,\Lambda_{r})$-composite matchings, i.e. 
arc diagrams with bottom sequence $\compmatch$ and top sequence 
$\compmatcht$. Further, we need to shift by 
counting the number of ups $\mathrm{up}(\Lambda_0)$ 
and downs $\mathrm{down}(\Lambda_0)$ in $\Lambda_{0}$.
In formulas,
\begin{gather}\label{eq:arc-bimodule}
\Mo[\parameterS](\compmatch,\compmatcht) 
=\ring
\CUPBasis{\compmatch,\compmatcht} 
\langle-(\mathrm{up}(\Lambda_0)+\mathrm{down}(\Lambda_0))\rangle,
\end{gather}
which shares the very same properties as 
web bimodules, but these are always free of finite rank.
(In some sense they are reduced versions of web bimodules.)

\begin{definition}\label{definition:catbimodules}
Let $\Modpgr{\Arcalg}$ be the additive closure of the following graded, $\ring$-linear $2$-category:
\begin{enumerate}[label=(\roman*)]

\item Objects are the various $\Lambda\in\bblock$.

\item $1$-morphisms are finite direct sums and 
tensor products (taken over the algebra $\Arcalg$) 
of the $\Arcalg$-bimodules 
$\Mo[\parameterS](\compmatch,\compmatcht)$.

\item $2$-morphisms are $\Arcalg$-bimodule 
homomorphisms.

\item The composition of arc bimodules is 
the tensor product ${}_-\otimes_{\Arcalg}{}_-$. 
The vertical composition of $\Arcalg$-bimodule 
homomorphisms is 
the usual composition. The horizontal composition is 
given by tensoring (over $\Arcalg$).

\end{enumerate}
\end{definition}
\section{Isomorphisms, equivalences and their consequences}\label{sec:iso}
This section has two main goals. First, 
we will construct an isomorphism 
of graded algebras
$\Iso{}{}\colon\webalg[\parameterS]^{\based}\stackrel{\cong}{\longrightarrow}\Arcalg$, 
where $\webalg[\parameterS]^{\based}$ 
is a certain subalgebra of $\webalg[\parameterS]$ 
defined in \eqref{eq:newwebalgs}. (In fact, this isomorphism is our only reason to restrict to $\parameterS$ instead of working with $\parameter$.)

From this we obtain (with $\web({}_-)$ as in \eqref{eq:webmap}):

\begin{theorem}\label{theorem:matchalgebras}
There is an equivalence of graded, $\ring$-linear
$2$-categories 
\begin{gather}\label{eq:equi-2-cat}
\boldsymbol{\Iso{}{}}\colon\webcatS\stackrel{\cong}{\longrightarrow}\Modpgr{\Arcalg}
\end{gather} 
induced by $\Iso{}{}$
under which the web bimodules 
$\M[\parameterS](\web(\compmatch,\compmatcht))$ and the 
arc bimodules $\Mo[\parameterS](\compmatch,\compmatcht)$ 
are identified.
\end{theorem}

Second, let $R[\parto]$ be a ring 
with a grading so that all $r\in R$ are of degree $0$ 
and $\parto$ is of degree $4$.
Let 
$\specS\colon\ring\to R[\parto]$ be any 
ring homomorphism with $\specS(\parto)=\parto$. 
Using the embedding $R[\parto]\to\ring$, set 
\begin{gather}\label{eq:scalars}
\ArcalgS[\parto,\specS(\parsign),\specS(\paro)]=
\ArcalgS_{R[\parto]}[\parto,\specS(\parsign),\specS(\paro)]
\otimes_{R[\parto]}\ring.
\end{gather}
(We need these scalar extensions for technical reasons, e.g. to 
make statements as ``isomorphisms of $\ring$-algebras''. We omit the 
subscript for these.)
We show in \fullref{subsec:iso-alg}, 
where we explicitly construct the isomorphism from \eqref{eq:iso-quite-new}, 
the following.

\begin{theorem}\label{theorem:crazy-isomorphism2}
There is an isomorphism of graded $\ring$-algebras
\begin{gather}\label{eq:iso-quite-new}
\Isonew{}{}\colon\Arcalg\overset{\cong}{\longrightarrow}
\ArcalgS[\parto,\specS(\parsign),\specS(\paro)].
\end{gather}
\end{theorem}

Since we now have an isomorphism of graded $\ring$-algebras we immediately get equivalences of the associated bimodule $2$-categories:

\begin{theorem}\label{theorem:matchalgebrasnew}
Let $R[\parto],\specS$ and 
$\ArcalgS[\parto,\specS(\parsign),\specS(\paro)]$ be as above. 
There is an equivalence 
(which is, in fact, even an isomorphism) of graded, $\ring$-linear
$2$-categories 
\begin{gather}\label{eq:equi-2-cat-arc}
\boldsymbol{\Isonew{}{}}\colon\Modpgr{\Arcalg}
\stackrel{\cong}{\longrightarrow}\Modpgr{\ArcalgS[\parto,\specS(\parsign),\specS(\paro)]}
\end{gather}
induced by $\Isonew{}{}$
under which $\Mo[\parameterS](\compmatch,\compmatcht)$ 
and 
$\Mo[\parto,\specS(\parsign),\specS(\paro)](\compmatch,\compmatcht)$ 
are identified.
\end{theorem}

Taking \fullref{proposition:special}, the 
equivalences \eqref{eq:equi-2-cat} and \fullref{theorem:matchalgebrasnew} 
together (and working over $\Z[\parto,i]$), we obtain that $\F_{\Z[\parto,i]}[\KBN]$, 
$\F_{\Z[\parto,i]}[\Ca]$, $\F_{\Z[\parto,i]}[\CMW]$ and 
$\F_{\Z[\parto,i]}[\Bl]$ are all equivalent,
see \fullref{corollary:all-equivalent}.
\subsection{Web and arc algebras}\label{subsec:algmodel}
To prove  \fullref{theorem:matchalgebras} we first construct a graded algebra isomorphism
$\Iso{}{}\colon\webalg[\parameterS]^{\based}\to\Arcalg$. 
Recall from \cite[Section 4]{EST} that there is a bijection
\begin{gather}\label{eq:identification}
\bY\to\bblock,\quad\vec{k}\mapsto\Lambda,\quad\text{given by  }\;0\mapsto\circ,
\quad 1\mapsto\dummy,\quad 2\mapsto\times.
\end{gather}
Here $\circ,\dummy,\times$ are entries of ${\rm seq}(\Lambda)$ 
and $\Lambda$ is determined demanding that $\Lambda$ is balanced. 
We identify, using \eqref{eq:identification}, 
such $\vec{k}$'s and $\Lambda$'s in what follows.
Moreover, for 
$\Lambda\in\bblock$ and 
$\lambda\in\Lambda$, there is a unique web $\web(\lambda)$ 
associated to the cup diagram $\underline{\lambda}$, see \cite[Lemma 4.8]{EST}. 
That is, there is a map
\begin{gather}\label{eq:webmap}
\web({}_-)\colon\Lambda\to\CUP(\vec{k}),\quad\lambda\mapsto\web(\lambda)
\end{gather}
assigning to $\lambda$ its cup diagram $\underline{\lambda}$ and then an associated web, which depends on a choice. In fact, see \fullref{example:from-webs-to-arc-diagrams}, $\web({}_-)$ is a split of the evident map which assigns to each web an arc diagram.
Similarly, 
for each
$\compmatch$-composite matching $\compmatcht$ 
there is a unique associated web 
$\web(\compmatch,\compmatcht)$ (given by an analogous map). 
The images of these maps are called \textit{basis webs}. All the reader needs 
to know about these basis webs is summarized in 
\fullref{example:from-webs-to-arc-diagrams}
below. Details can be found in \cite[Section 4.1]{EST}.

\begin{example}\label{example:from-webs-to-arc-diagrams}
Given a web $u$, then we can associate to it 
an arc diagram $\cc{u}$ via
\begin{gather}\label{eq:webmap-2}
\xy
(0,0)*{\includegraphics[scale=.65]{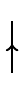}};
\endxy
\mapsto
\xy
(0,0)*{\includegraphics[scale=.65]{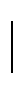}};
\endxy
\quad,\quad
\xy
(0,0)*{\includegraphics[scale=.65]{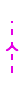}};
\endxy
\mapsto
\emptyset
\quad,\quad
\xy
(0,0)*{\includegraphics[scale=.65]{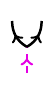}};
\endxy
\mapsto
\xy
(0,0)*{\includegraphics[scale=.65]{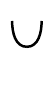}};
\endxy
\quad,\quad
\xy
(0,0)*{\includegraphics[scale=.65]{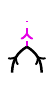}};
\endxy
\mapsto
\xy
(0,0)*{\includegraphics[scale=.65]{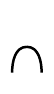}};
\endxy
\end{gather}
Since we do not consider any relations on the set of webs, isotopic webs are not equal and there are 
plenty of webs giving the same arc diagram, but there is a preferred choice 
of a preimage which defines our split $\web({}_-)$. 
An example is
\begin{gather}\label{eq:diagramexample}
\xy
(0,0)*{\includegraphics[scale=.65]{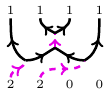}};
\endxy
\rightsquigarrow
\xy
(0,0)*{\includegraphics[scale=.65]{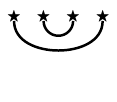}};
\endxy
\quad,\quad
\xy
(0,0)*{\includegraphics[scale=.65]{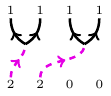}};
\endxy
\rightsquigarrow
\xy
(0,0)*{\includegraphics[scale=.65]{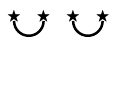}};
\endxy
\end{gather}
The only thing we need is the existence of this split, and that it is 
parameter independent.
\end{example}

Moreover, as indicated in \fullref{example:from-webs-to-arc-diagrams}, 
the set of basis webs
\begin{gather}\label{eq:fin-dim-webs}
\CUPB(\vec{k})=\{u\in\CUP(\vec{k})\mid u=\web(\lambda)\text{ for some }\lambda\in\Lambda\}
\end{gather}
is always a strict and finite subset of $\CUP(\vec{k})$.
Now, given $\lambda,\mu\in\Lambda$,
let us denote
\begin{gather}\label{eq:newwebalgs}
\webalg[\parameterS]^{\based}_{\vec{k}}=
{\textstyle\bigoplus_{u,v\in\CUPB(\vec{k})}}
{}_u(\webalg[\parameterS]_{\vec{k}})_v,
\quad\quad\webalg[\parameterS]^{\based}=
{\textstyle\bigoplus_{\vec{k}\in\bY}}
\webalg[\parameterS]^{\based}_{\vec{k}}.
\end{gather}
Clearly, $\webalg[\parameterS]^{\based}$ 
is a graded subalgebra of $\webalg[\parameterS]$ 
with \textit{based web bimodules} 
$\M[\parameterS]^{\based}(u)$ given as 
in \eqref{eq:the-web-bimodules}, but using 
$\CUPB(\vec{k})$ instead of $\CUP(\vec{k})$. 
We can view these as $\webalg[\parameterS]$-bimodules 
as well, and we thus, get a 
$2$-category $\webcatScirc$ consisting of 
based $\webalg[\parameterS]$- or $\webalg[\parameterS]^{\based}$-bimodules.

Recalling the cup foam bases, 
we have the following lemmas.

\begin{lemma}\label{lemma:matchvs}
Let $u,v\in\CUPB(\vec{k})$ be webs such that $u=\web(\lambda)$ and $v=\web(\mu)$.
There is an isomorphism of graded, free $\field$-modules
\begin{gather}\label{eq:isoalgebras1}
\Iso{uv}{\lambda\mu}\colon{}_u(\webalg[\parameterS]_{\vec{k}})_v\to 
{}_{\lambda}(\Arcalg_{\Lambda})_{\mu}
\end{gather}
which sends $\CUPbasis{u}{v}{\vec{k}}$ to $\CUPbasis{\lambda}{\mu}{\Lambda}$ 
by identifying 
the basis cup foams without dots with anticlockwise circles and 
the basis cup foams with dots with clockwise circles.

Let $u\in\Hom_{\F[\parameterS]}(\vec{k},\vec{l})$ 
be a web such that $u=\web(\compmatch,\compmatcht)$.
There is an isomorphism of graded, free $\ring$-modules
\begin{gather}\label{eq:isoalgebrasnext}
\Iso{u}{(\compmatch,\compmatcht)}\colon \M[\parameterS]^{\based}(u)\to
\Mo[\parameterS](\compmatch,\compmatcht)
\end{gather}
which sends $\CUPBasis{u}$ to $\CUPBasis{\compmatch,\compmatcht}$ 
by identifying 
the basis cup foams without dots with anticlockwise circles and 
the basis cup foams with dots with clockwise circles.
\end{lemma}

\begin{proof}
The arguments used in \cite[Lemmas 4.15 and 4.16]{EST} 
as well as the construction of the two bases in question 
are parameter independent. Thus, 
\cite[Lemmas 4.15 and 4.16]{EST} work verbatim and the claim follows. 
\end{proof}

\begin{lemma}\label{lemma:matchvs2}
For any $\lambda,\mu\in\Lambda$ and $u=\web(\lambda), v=\web(\mu)$: the isomorphisms
$\Iso{uv}{\lambda\mu}$ 
from \eqref{eq:isoalgebras1} 
extend to isomorphisms of graded, free 
$\ring$-modules
\begin{gather}\label{eq:isoalgebras2}
\Iso{\vec{k}}{\Lambda}\colon\webalg[\parameterS]^{\based}_{\vec{k}}\to\Arcalg_{\Lambda},\quad\quad
\Iso{}{}\colon\webalg[\parameterS]^{\based}\to\Arcalg.
\end{gather}
\end{lemma}

\begin{proof}
Clear by \fullref{lemma:matchvs}.
\end{proof}

\begin{theorem}\label{theorem:matchalgebrasold}
The maps from \eqref{eq:isoalgebras2} are isomorphisms of graded algebras.
\end{theorem}

The non-trivial proof 
of \fullref{theorem:matchalgebrasold} is given in 
\fullref{subsec:proof-one}
\begin{remark}\label{remark:why-parameterS}
We use the specialization of the parameters $\parameter$ to 
$\parameterS$ to not 
having to worry about the difference between the 
``directions'' in which we squeeze, migrate dots 
or perform ordinary-to-phantom neck cutting. 
Being more careful with the performed steps in the topological 
rewriting process  
leads to an analog of \fullref{theorem:matchalgebrasold} 
for $\parameter$ as well. Since this would require the introduction 
of some involved (but straightforward) notions for the 
diagram combinatorics keeping track of 
directions, we have decided, for brevity and clearness, 
to only do the $\parameterS$ case here---which includes 
our list of examples from \eqref{eq:main-specs} anyway.
Moreover, we could also relax the condition of webs being 
upwards oriented. This makes the involved scalars 
more cumbersome, and we decided not to pursue this direction further.
\end{remark}

A direct consequence of \fullref{theorem:matchalgebrasold} is that the algebraic multiplication rule is actually topologically in nature:

\begin{corollary}\label{corollary:matchalgebras1}
The multiplication rule of $\Arcalg_{\Lambda}$ 
is independent of the order in 
which the surgeries are performed. In particular, $\Arcalg_{\Lambda}$ is 
a graded, associative, unital algebra.
\end{corollary}

We are now ready to prove our first main result, i.e. the equivalence 
from \eqref{eq:equi-2-cat}.

\begin{proof}[Proof of \fullref{theorem:matchalgebras}]
The algebras $\webalg[\parameterS]_{\vec{k}}$ and 
$\webalg[\parameterS]^{\based}_{\vec{k}}$ 
are graded Morita equivalent (this 
can be seen as in \cite[Proof of Theorem 4.1]{EST}).
To deduce \fullref{theorem:matchalgebras}, we note that
the identification of the bimodules as graded, free $\ring$-modules is clear 
by \fullref{lemma:matchvs}, while the actions 
agree by \fullref{theorem:matchalgebrasold} 
and the construction of the actions. 
Everything in these arguments is independent of the parameters 
and thus, the theorem follows.
\end{proof}
\subsection{Arc algebras: isomorphisms}\label{subsec:iso-alg}
In this section we show that the signed $2$-parameter arc 
algebra $\Arcalg$ and the 
(scalar extended) $\KBN$ 
specialization
\begin{gather}\label{eq:some-label}
\ArcalgS[\KBN]= 
\ArcalgS_{\Z[\parto]}[\KBN] 
\otimes_{\Z[\parto]}\Z[\parto,\paro^{\pm 1}]
\end{gather} 
are isomorphic as graded 
algebras. Here, as usual, 
$\parto$ is of degree $4$. 
We now like to show \fullref{theorem:crazy-isomorphism2}, which implies \fullref{theorem:matchalgebrasnew}.

Both,  
$\ArcalgS[\KBN]_\Lambda$ and $\Arcalg_\Lambda$, are 
isomorphic as 
graded, free $\ring$-modules to 
$\ring\CUPBasis{\Lambda}$, 
with $\CUPBasis{\Lambda}$ being the cup foam basis. 
By definition, the multiplication differs only by the appearing coefficients 
in the result. Hence, we will give the isomorphism from 
$\ArcalgS[\KBN]_\Lambda$ to $\Arcalg_\Lambda$
by defining 
a coefficient for each of the diagrams appearing in the 
multiplication and show that the maps intertwine the two 
multiplication rules.

\begin{definition}\label{definition:stacked1}
We call any diagram appearing in an 
intermediate step in the multiplication procedure 
of the arc algebra a 
\textit{(stacked) diagram}. We denote such diagrams 
throughout this section by $\stacked$ (possibly with decorations and indices), and 
choices of orientations of it by $\stacked^{\mathrm{or}}$ (possibly with decorations and indices).
\end{definition}

\begin{definition}\label{definition:stacked}
We define sets of arcs inside 
a circle $C$ in a fixed diagram stacked $\stacked$:
\begin{enumerate}[label=(\roman*)]

\item $\icup(C)$ denotes the set of all cups in $C$ such 
that the exterior of $C$ is above the cup.

\item $\icap(C)$ denotes the set of all caps in $C$ such 
that the exterior of $C$ is below the cap.

\item $\ecup(C)$ denotes the set of all cups in $C$ such 
that the interior of $C$ is above the cup.

\item $\ecap(C)$ denotes the set of all caps in $C$ such 
that the interior of $C$ is below the cap.

\end{enumerate}
The exterior and 
interior is meant here with respect 
to the circle $C$ only (ignoring 
all other components of $\stacked$). 
\end{definition}

\begin{example}\label{example:ext-int}
The outer circle $C_{\mathrm{out}}$ in the third diagram in \eqref{eq:mult2} 
has $\icup(C_{\mathrm{out}})=1$, $\icap(C_{\mathrm{out}})=0$, $\ecup(C_{\mathrm{out}})=1$ 
and $\ecap(C_{\mathrm{out}})=2$. The inner circle $C_{\mathrm{in}}$ 
in the same stacked diagram 
has exactly the same numbers. The circle $C$ 
in the rightmost diagram in \eqref{eq:mult2} has 
$\icup(C)=1$, $\icap(C)=1$, $\ecup(C)=2$ 
and $\ecap(C)=2$. Moreover,
\begin{gather}\label{eq:some-label-2}
\left(\icup(C_{\mathrm{out}})\,{\scriptstyle\cup}\,\icap(C_{\mathrm{out}})
\,{\scriptstyle\cup}\,\ecup(C_{\mathrm{in}})
\,{\scriptstyle\cup}\,\ecap(C_{\mathrm{in}})\right)\setminus\text{surg}
=\icup(C)\,{\scriptstyle\cup}\,\icap(C).
\end{gather}
Here ``surg'' means the set containing the cup-cap 
involved in the surgery.
\end{example}

We denote by $\mathbb{B}(\stacked)$ 
the set of all possible orientations of 
a given diagram $\stacked$, and by $\op{t}(D)$ the rightmost point of $D$.

\begin{definition}\label{definition:the-coefficient-map}
For a fixed $\stacked$, we define its $\ring$-linear 
\textit{coefficient map} via:
\begin{gather}\label{eq:some-label-3}
\begin{aligned}
\coeff_{\stacked} \colon \ring\mathbb{B}(\stacked) &\longrightarrow
\ring\mathbb{B}(\stacked), 
\\
\stacked^{\mathrm{or}} & \longmapsto \left( 
{\textstyle\prod_{\text{circles}}}\;
\coeff_{\parsign}(C,\stacked^{\mathrm{or}})\cdot\coeff_{\paro}(C,\stacked^{\mathrm{or}})\right)\stacked^{\mathrm{or}},
\end{aligned}
\end{gather}
and we let $\coeff(C,D)=
\coeff_{\parsign}(C,\stacked^{\mathrm{or}})\cdot\coeff_{\paro}(C,\stacked^{\mathrm{or}})$.
Here the product runs over all circles in $\stacked$, 
and the involved terms (i.e. for each such circle $C$) are defined as follows.

\noindent $\succ$ If $C$ is 
oriented anticlockwise, then set
\begin{gather}\label{eq:some-label-4}
\begin{aligned}
\coeff_{\parsign}(C,\stacked^{\mathrm{or}}) \;=\;& 
{\textstyle\prod_{\gamma \in \icups(C)}}\;
\parsign^{(s_{\Lambda}(\gamma)+1)\pos_{\Lambda}(\gamma)} 
\cdot
{\textstyle\prod_{\gamma \in \icaps(C)}} \;
\parsign^{s_{\Lambda}(\gamma)(\pos_{\Lambda}(\gamma)+1)},
\\
\coeff_{\paro}(C,\stacked^{\mathrm{or}}) \;=\;&
{\textstyle\prod_{\gamma \in \icups(C)}}\;
\paro^{-s_{\Lambda}(\gamma)}
\cdot
{\textstyle\prod_{\gamma \in \icaps(C)}}\;
\paro^{s_{\Lambda}(\gamma)-1},
\end{aligned}
\end{gather}
where, as usual, 
$\pos_{\Lambda}(\gamma)$ denotes the position 
of the leftmost points of the $\gamma$'s
and $s_{\Lambda}(\gamma)$ is 
the saddle width.

\noindent $\succ$ If $C$ is 
oriented clockwise, then we use the 
same coefficient and additionally
multiply by $\parsign^{\op{t}(C)}$. (The 
reader might think of $\parsign^{\op{t}(C)}$ as keeping 
track of ``dot moving''.)
\end{definition}

Since $\parsign=\pm 1$, its powers
matter only mod $2$.

\begin{example}\label{example:needed!}
The circle $C_{\mathrm{out}}$
in the third diagram $\stacked_3$ in \eqref{eq:mult2} has only one 
cup $\gamma\in\icups(C_{\mathrm{out}})$, and $s_{\Lambda}(\gamma)=1$ 
and $\pos_{\Lambda}(\gamma)=3$. Thus, if 
$\stacked^{\rm or}_3$ denotes the orientation from \eqref{eq:mult2}, then
$\coeff(C_{\mathrm{out}},\stacked^{\rm or}_3)=\parsign^6\paro^{-1}=\paro^{-1}$. Similarly 
one obtains 
$\coeff(C_{\mathrm{in}},\stacked^{\rm or}_3)=\parsign^4\paro^{-1}=\paro^{-1}$. 
Moreover, the circle $C$ 
in the rightmost diagram $\stacked_4$ in \eqref{eq:mult2} has 
one cup $\gamma$ and one cap $\gamma^{\prime}$ ``pushing inwards'',
and $s_{\Lambda}(\gamma)=2$, $\pos_{\Lambda}(\gamma)=1$, $s_{\Lambda}(\gamma^{\prime})=1$ 
and $\pos_{\Lambda}(\gamma^{\prime})=1$. 
Thus, 
$\coeff(C,\stacked^{\rm or}_4)=
\parsign^3\paro^{-2}\cdot\parsign^2=\parsign\paro^{-2}$.
\end{example}

We will usually write 
$\coeff(C^{\rm anti})=\coeff(C,\stacked^{\mathrm{or}})$ 
etc.
to denote the coefficient for the circle $C$ when 
the orientation is chosen such that $C$ is oriented 
anticlockwise, and similarly 
$\coeff(C^{\rm cl})=\coeff(C,\stacked^{\mathrm{or}})$ 
when it is chosen such that $C$ is 
oriented clockwise. For example, we have by definition
\begin{gather}\label{eq:from-anti-to-clock}
\coeff(C^{\rm cl}) = \coeff(C^{\rm anti}) \cdot 
\parsign^{\op{t}(C)}.
\end{gather}

\fullref{definition:the-coefficient-map} 
restricts to a homogeneous, $\ring$-linear map 
\begin{gather}\label{eq:some-label-5}
\coeff_{\lambda,\mu} \colon \ring\CUPbasis{\lambda}{\mu}{\Lambda} 
\to\ring\CUPbasis{\lambda}{\mu}{\Lambda},
\end{gather}
for $\lambda,\mu\in\Lambda$. By summing all of these up we obtain a 
homogeneous, $\ring$-linear map
\begin{gather}\label{eq:the-crazy-map}
\coeff_\Lambda \colon  
\ArcalgS[\KBN]_{\Lambda}
\to \Arcalg_{\Lambda}
\end{gather}
since 
$\ArcalgS[\KBN]_{\Lambda}\cong
\ring\CUPBasis{\Lambda}
\cong\Arcalg_{\Lambda}$, 
as graded, free $\ring$-modules, by definition of the arc algebras.

In fact, the $\ring$-linear map from \eqref{eq:the-crazy-map} 
is actually an isomorphism of graded algebras: 

\begin{proposition}\label{proposition:crazy-isomorphism}
The maps $\coeff_\Lambda$
from \eqref{eq:the-crazy-map} are isomorphisms 
of graded $\ring$-algebras for all $\Lambda\in\bblock$. These can be extended to an 
isomorphism of graded $\ring$-algebras
\begin{gather}\label{eq:some-label-7}
\coeff \colon  
\ArcalgS[\KBN]
\overset{\cong}{\longrightarrow} \Arcalg.
\end{gather}
\end{proposition}

The non-trivial proof of 
\fullref{proposition:crazy-isomorphism}, explicitly matches the multiplication rules and is
given in \fullref{subsec:proof-two}.

Given the setup as in the beginning 
of this section, we define the map 
$\Isonew{}{}$ from \eqref{eq:iso-quite-new} 
as follows. Let 
$\coeff^{\parto,\specS(\parsign),\specS(\paro)}\colon\ArcalgS[\KBN]
\to\ArcalgS[\parto,\specS(\parsign),\specS(\paro)]$ 
be the homogeneous, $\ring$-linear map obtained in the same way 
as $\coeff\colon\ArcalgS[\KBN]
\to\Arcalg$, but using the specialized parameters $\specS(\parsign)$ 
and $\specS(\paro)$. 
Then, by \fullref{proposition:crazy-isomorphism}, 
set
\begin{gather}\label{eq:to-avoid-as-before}
\Isonew{}{}\colon\Arcalg\to\ArcalgS[\parto,\specS(\parsign),\specS(\paro)],\;
\Isonew{}{}=\coeff^{\parto,\specS(\parsign),\specS(\paro)}\circ(\coeff)^{-1}.
\end{gather}

We are now ready to prove \fullref{theorem:crazy-isomorphism2}, 
assuming \fullref{proposition:crazy-isomorphism}.

\begin{proof}[Proof of \fullref{theorem:crazy-isomorphism2}]
The proof of \fullref{proposition:crazy-isomorphism} 
only uses that $\parsign=\pm 1$ and that $\paro$ is invertible. 
Thus, the same arguments work for any 
$\specS(\parsign)$ and $\specS(\paro)$ 
providing a homogeneous isomorphism $\coeff^{\parto,\specS(\parsign),\specS(\paro)}$
between $\ArcalgS[\KBN]$ and $\ArcalgS[\parto,\specS(\parsign),\specS(\paro)]$.
\end{proof}
\subsection{Arc bimodules: bimodule homomorphisms}\label{subsec:iso-bimod}
In \fullref{proposition:crazy-isomorphism} we have identified $\ArcalgS[\KBN]$ 
with $\Arcalg$ 
using the coefficient map. Thus, there is also 
an identification of their bimodules. The aim 
of this section is to make this explicit.
For the identification of the 
bimodules $\Mo[\KBN](\compmatch,\compmatcht)$ and 
$\Mo[\parameterS](\compmatch,\compmatcht)$ for a 
fixed admissible matching 
$(\compmatch,\compmatcht)$ we need to introduce some 
additional notations and slightly modify the 
coefficient map. But 
otherwise the identification works as for the algebras.

\begin{definition}\label{definition:stacked2}
As in \fullref{definition:stacked1}, we call any diagram 
appearing in the intermediate 
step of the multiplication procedure a 
\textit{stacked diagram}.
Furthermore, fixing a circle $C$ in a stacked diagram,
we define subsets of arcs
containing the arcs in the basic 
moves of 
the second type for $\alpha_i$ 
and $-\alpha_i$ from \cite[(31)]{EST}, i.e. local moves from 
\makebox[1em]{\raisebox{0.03cm}{$\dummy$}} \makebox[1em]{$\times$} to 
\makebox[1em]{$\times$} \makebox[1em]{\raisebox{0.03cm}{$\dummy$}}, 
or from \makebox[1em]{$\times$} \makebox[1em]{\raisebox{0.03cm}{$\dummy$}} to 
\makebox[1em]{\raisebox{0.03cm}{$\dummy$}} \makebox[1em]{$\times$}. 
We divide these depending on the exterior or interior of 
$C$:

\begin{enumerate}[label=(\roman*)]

\item $\eright(C)$ denotes the set of arcs in a local move from \makebox[1em]{$\times$} \makebox[1em]{\raisebox{0.03cm}{$\dummy$}} to 
\makebox[1em]{\raisebox{0.03cm}{$\dummy$}} \makebox[1em]{$\times$}, where 
the exterior of $C$ is to the lower left of the arc.

\item $\eleft(C)$ denotes the set of arcs in a local move
from \makebox[1em]{\raisebox{0.03cm}{$\dummy$}} \makebox[1em]{$\times$} to 
\makebox[1em]{$\times$} \makebox[1em]{\raisebox{0.03cm}{$\dummy$}}, where the 
exterior of $C$ is to the lower right of the arc.

\item $\iright(C)$ denotes the set of arcs in a local move 
from \makebox[1em]{$\times$} \makebox[1em]{\raisebox{0.03cm}{$\dummy$}} to 
\makebox[1em]{\raisebox{0.03cm}{$\dummy$}} \makebox[1em]{$\times$}, where the interior 
of $C$ is to the lower left of the arc.

\item $\ileft(C)$ denotes the set of arcs in a local move from 
\makebox[1em]{\raisebox{0.03cm}{$\dummy$}} \makebox[1em]{$\times$} to 
\makebox[1em]{$\times$} \makebox[1em]{\raisebox{0.03cm}{$\dummy$}}, where the interior of 
$C$ is to the lower right of the arc.
\end{enumerate}

Again, the 
exterior and 
interior is meant here with respect 
to the circle $C$ only. 
\end{definition}

\begin{definition}\label{definition:the-coefficient-map2}
As in \fullref{definition:the-coefficient-map} we define a \textit{coefficient map} $\coeff_{\stacked}$ and $\coeff(C,\stacked)$ with the following differences:

\noindent $\succ$ If $C$ is 
oriented anticlockwise
when looking at the orientation 
$\stacked^{\mathrm{or}}$, then set
\begin{gather}\label{eq:the-crazy-map-def}
\begin{aligned}
\coeff_{\parsign}(C,\stacked^{\mathrm{or}}) =\;& 
{\textstyle\prod_{\gamma \in \icups(C)}}\;
\parsign^{(s_{\Lambda}(\gamma)+1)\pos_{\Lambda}(\gamma)} 
\cdot 
{\textstyle\prod_{\gamma \in \icaps(C)}}\;
\parsign^{s_{\Lambda}(\gamma)(\pos_{\Lambda}(\gamma)+1)}\\
\;\;&\cdot 
{\textstyle\prod_{\gamma \in\,\erights(C)}}\; 
\parsign^{\pos_{\Lambda}(\gamma)}
\cdot
{\textstyle\prod_{\gamma \in\elefts(C)}}\;
\parsign^{\pos_{\Lambda}(\gamma)+1},
\\
\coeff_{\paro}(C,\stacked^{\mathrm{or}}) =\;&
{\textstyle\prod_{\gamma \in \icups(C)}}\;
\paro^{-s_{\Lambda}(\gamma)} 
{\textstyle\prod_{\gamma \in \icaps(C)}}\;
\paro^{s_{\Lambda}(\gamma)-1} 
\cdot
\paro^{\#\left( \eright(C)\,\cup\,\eleft(C)\right)},\\
\end{aligned}
\end{gather}
where we use the same notations as in 
\fullref{definition:the-coefficient-map}.

\noindent $\succ$ If $C$ is 
oriented clockwise
when looking at the orientation 
$\stacked^{\mathrm{or}}$, then we use the 
same coefficient and additionally
multiply by $\parsign^{\op{t}(C)}$.
\end{definition}

Similar to \eqref{eq:the-crazy-map}, we use these maps to define a 
homogeneous, $\ring$-linear map
$\coeff_{\compmatch,\compmatcht}$ which has the following properties.

\begin{proposition}\label{proposition:crazy-module-maps}
The map 
\begin{gather}\label{eq:the-crazy-map-modules}
\coeff_{\compmatch,\compmatcht} \colon  
\Mo[\KBN](\compmatch,\compmatcht)
\to \Mo[\parameterS](\compmatch,\compmatcht)
\end{gather}
is an isomorphism of graded, free $\ring$-modules that
intertwines the bimodule actions of $\ArcalgS[\KBN]$ and $\Arcalg$, i.e. for any 
$x \in \ArcalgS[\KBN]$ and any
$m \in \Mo[\KBN](\compmatch,\compmatcht)$ 
it holds that
$\coeff_{\compmatch,\compmatcht}(x\cdot m)
=\coeff(x)\cdot\coeff_{\compmatch,\compmatcht}(m)$ and $\coeff_{\compmatch,\compmatcht}(m\cdot x)
=\coeff_{\compmatch,\compmatcht}(m)\cdot\coeff(x)$.
\end{proposition}

The proof of this proposition appears in \fullref{subsec:proof-three}.
As before in \eqref{eq:to-avoid-as-before}, 
we use \fullref{proposition:crazy-module-maps} 
to define
$\coeff^{\parto,\specS(\parsign),\specS(\paro)}_{\compmatch,\compmatcht} \colon  
\Mo[\parameterS](\compmatch,\compmatcht)
\to \Mo[\parto,\specS(\parsign),\specS(\paro)](\compmatch,\compmatcht)$ 
to be the homogeneous, $\ring$-linear map obtained in the same way 
as $\coeff_{\compmatch,\compmatcht}$ from \eqref{eq:the-crazy-map-modules}, 
but using the specialized parameters $\specS(\parsign)$ 
and $\specS(\paro)$ instead of $\parsign$ and $\paro$. 
Then set
\begin{gather}\label{eq:coeff-in-action}
\coeff_{\Isonew{}{}}=\coeff^{\parto,\specS(\parsign),\specS(\paro)}_{\compmatch,\compmatcht}
\circ\!(\coeff_{\compmatch,\compmatcht})^{-1}
\colon\Mo[\parameterS](\compmatch,\compmatcht)\to\Mo[\parto,\specS(\parsign),\specS(\paro)](\compmatch,\compmatcht).
\end{gather}

As in the proof \fullref{theorem:crazy-isomorphism2}, 
but using \fullref{proposition:crazy-module-maps} we get:

\begin{corollary}\label{corollary:crazy-module-maps}
The map $\coeff_{\Isonew{}{}}$
is an isomorphism of graded, free $\ring$-modules that
intertwines the actions of $\Arcalg$ and 
$\ArcalgS[\parto,\specS(\parsign),\specS(\paro)]$.
\end{corollary}

\begin{proof}[Proof of \fullref{theorem:matchalgebrasnew}]
This follows from \fullref{theorem:crazy-isomorphism2} 
and \fullref{corollary:crazy-module-maps}.
\end{proof}
\subsection{Arc bimodules: co-structure}\label{subsec:comult}
The aim of this section is to describe a \textit{co-structure} 
topologically on web bimodules $\M[\parameterS](v)$ 
and algebraically on arc bimodules $\Mo[\parameterS](\compmatch,\compmatcht)$. 
Then we match theses structures---which again comes with sophisticated scalars---for 
different specializations of $\parameterS$
using an isomorphism similar, but not equal, to the coefficient map 
from \eqref{eq:the-crazy-map-modules}.
We start on the side of $\webalg[\parameterS]$. 
We again
only consider balanced $\vec{k},\vec{l}\in\bY$.

\begin{remark}\label{remark:again-more-general}
The construction of the co-structure 
clearly works for $\webalg[\parameter]$ and $\webalg^\pm[\parameter]$ as well.
\end{remark}

\begin{definition}\label{definition:reverse-mult}
Let $v\in\Hom_{\F[\parameterS]}(\vec{k},\vec{l})$. Recalling that 
we consider in 
$\F[\parameterS]$ webs without relations, we can pick any pair of neighboring 
vertical usual edges, ignoring possible phantom edges,  
and perform a \textit{reverse surgery} on $\M[\parameterS](v)$:
\begin{equation}\label{eq:next-label}
\xy
\xymatrix{
\raisebox{0.08cm}{\reflectbox{
	\xy
	(0,0)*{\includegraphics[scale=.65]{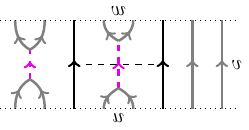}};
	\endxy
}}\ar@{|->}[rr]^{\text{saddle foam}}
&&
\raisebox{0.08cm}{\reflectbox{
	\xy
	(0,0)*{\includegraphics[scale=.65]{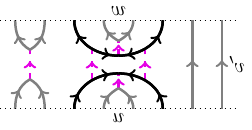}};
	\endxy
}}
}
\endxy
\end{equation}
ending up with a new web 
$v^{\prime}\in\Hom_{\F[\parameterS]}(\vec{k},\vec{l})$, and involving the saddle foam
\begin{gather}\label{eq:next-label-2}
\xy
(0,0)*{\includegraphics[scale=.65]{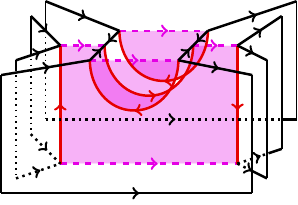}};
\endxy
\end{gather} 
Here where $u,w$ are (not-illustrated) webs glued to the bottom respectively top.
This should be read as follows: start with 
$f\in\twoHom_{\F[\parameter]}(\oneinsert{2\omega_{\ell}},uv^{\ast}vw)$ 
and stack on top of it a foam which is the 
identity at the bottom ($u$ part)
and top ($w$ part) of the web, and the saddle 
in between. Repeat this for all $u\in\CUP(\vec{k}),w\in\CUP(\vec{l})$.
\end{definition}

Note that we made a certain choice where to perform 
the reverse surgery, which is however uniquely determined by $v^{\prime}$. Thus, we 
can write $\boldsymbol{\mathrm{rMult}}_v^{v^{\prime}}$ etc. 
without ambiguity.

The following is clear.

\begin{lemma}\label{lemma:intertwiner}
The procedure from \fullref{definition:reverse-mult} 
defines a $\webalg[\parameterS]$-bimodule 
homomorphism 
$\boldsymbol{\mathrm{rMult}}_v^{v^{\prime}}
\colon\M[\parameterS](v)\to\M[\parameterS](v^{\prime})$.\qedhere
\end{lemma}

We define the same on the side of $\Arcalg$. As usual, 
all blocks are balanced.

\begin{definition}\label{definition:reverse-mult-arcalg}
Let $\compmatcht$ be a $\compmatch$-composite 
matching. 
Recalling that 
we construct these 
using the basic moves from \cite[(31)]{EST}, we can pick any pair of neighboring 
vertical arcs (ignoring possible cups and caps, and symbols 
\makebox[1em]{$\circ$} or \makebox[1em]{$\dummy$} in between) 
and perform a \textit{reverse surgery} on $\Mo[\parameterS](\compmatch,\compmatcht)$ 
giving us a new composite matching $\compmatcht^{\prime}$ 
for $\compmatch^{\prime}$:
\begin{gather}\label{eq:next-label-3}
\xy
\xymatrix{
\raisebox{0.075cm}{\reflectbox{
	\xy
	(0,0)*{\includegraphics[scale=.65]{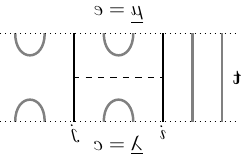}};
	\endxy
}}
\ar@{|->}[r]
&
\raisebox{0.075cm}{\reflectbox{
	\xy
	(0,0)*{\includegraphics[scale=.65]{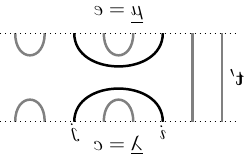}};
	\endxy
}}
}
\endxy
\end{gather}
Using the reversed surgery, we define a $\ring$-linear 
map $\boldsymbol{\mathrm{rmult}}_{\compmatcht}
^{\compmatcht^{\prime}}
\colon
\Mo[\parameterS](\compmatch,\compmatcht)
\to\Mo[\parameterS](\compmatch^{\prime},\compmatcht^{\prime})$ 
on basis elements of $\Mo[\parameterS](\compmatch,\compmatcht)$, which is in some sense
dual to the multiplication of the arc algebra,
by applying the following rules to the components involved in the reversed surgery:

\noindent \textbf{Non-nested Merge.} The non-nested 
circles $C_i$ and $C_j$ are merged into $C$. We use the same scalars as for the multiplication.

\noindent \textbf{Nested Merge.} The nested 
circles $C_i$ and $C_j$ are merged into $C$. 
We use the scalars as 
for the multiplication, but additionally multiply with $\parsign$.

\noindent \textbf{Non-nested Split.} The circle $C$ 
splits into the non-nested circles $C_{\mathrm{bot}}$ and $C_{\mathrm{top}}$ 
(being at the bottom or top of the picture). 
We use almost the same scalars as for 
multiplication, but in case $C$ is oriented anticlockwise, we use 
(for bottom, respectively top, circle oriented clockwise)
\begin{gather}\label{eq:again-something-1}
\paro \cdot \parsign^{\length_\Lambda(\gamma_{\mathrm{bot}}^{\rm ndot})}
\cdot\parsign^{s_\Lambda(\gamma)}
\quad\text{respectively}\quad
\parsign \cdot \paro \cdot \parsign^{\length_\Lambda(\gamma_{\mathrm{top}}^{\rm ndot})}
\cdot\parsign^{s_\Lambda(\gamma)}.
\end{gather}
Here 
$\gamma_{\mathrm{bot}}^{\rm ndot}$ respectively $\gamma_{\mathrm{top}}^{\rm ndot}$ 
are to be understood similar to \eqref{eq:new-dot1} and \eqref{eq:new-dot2}. 
In case $C$ is oriented clockwise, we use
\begin{gather}\label{eq:again-something-2}
\paro \cdot \parsign^{\length_\Lambda(\gamma_{\mathrm{top}}^{\rm dot})}\cdot\parsign^{\length_\Lambda(\gamma_{\mathrm{bot}}^{\rm ndot})}\cdot\parsign^{s_\Lambda(\gamma)}
\;\text{respectively}\;
\parto \cdot \parsign \cdot \paro \cdot \parsign^{\length_\Lambda(\gamma_{\mathrm{top}}^{\rm dot})}\cdot\parsign^{\length_\Lambda(\gamma_{\mathrm{top}}^{\rm ndot})}\cdot\parsign^{s_\Lambda(\gamma)}
\end{gather}
for both circles oriented clockwise respectively anticlockwise.

\noindent \textbf{Nested Split.} The circle $C$ 
splits into the non-nested circles $C_{\mathrm{in}}$ and $C_{\mathrm{out}}$. 
We use the same scalars as for the multiplication.
\end{definition}

\begin{example}\label{ex:yup-needed-yet-again}
An illustration of the reverse multiplication 
is
\begin{gather}\label{eq:again-something-3}
\xy
\xymatrix@C-10pt{
\raisebox{0cm}{
\xy
(0,0)*{\includegraphics[scale=.65]{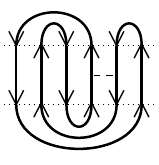}};
\endxy
}
\ar@{|->}[r]
&
\paro\parsign^1
\raisebox{0cm}{
\xy
(0,0)*{\includegraphics[scale=.65]{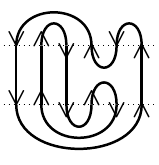}};
\endxy
}
\;+\;\parsign\paro\parsign^1
\raisebox{0cm}{
\xy
(0,0)*{\includegraphics[scale=.65]{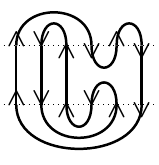}};
\endxy
}
}
\endxy
\end{gather}
\end{example}

\begin{remark}\label{remark:costructure}
We note that the web algebra is in fact a 
symmetric Frobenius algebra, by abstract reasons. 
(This can be seen by copying \cite[Proposition 30]{Khov} or \cite[Theorem 3.9]{MPT}.) 
The same holds for the arc algebra. 
(This can be seen by copying \cite[Theorem 6.3]{BS1}.)
The existence of a comultiplication for the web and the arc algebra can thus be deduced because they are Frobenius algebras. For the specialization $(0,1,1)$ of the arc algebra or the web algebra in general the above comultiplication are the ones coming from the Frobenius structure, but we have not verified this for the general case for the arc algebra.
\end{remark}

\begin{lemma}\label{lemma:intertwiner-2}
The procedure from \fullref{definition:reverse-mult-arcalg} 
defines an $\Arcalg$-bimodule homomorphism 
$\boldsymbol{\mathrm{rmult}}\colon
\Mo[\parameterS](\compmatch,\compmatcht)
\to\Mo[\parameterS](\compmatch^{\prime},\compmatcht^{\prime})$.
\end{lemma}

\begin{proof}
This follows 
from \fullref{proposition:match-ts} below.
\end{proof}

As before, we have the following, 
recalling the equivalence $\boldsymbol{\Iso{}{}}$ 
from \fullref{theorem:matchalgebras}
induced by the isomorphism $\Iso{}{}$ from \eqref{eq:isoalgebras2}
under which the web bimodules 
$\M[\parameterS](\web(\compmatch,\compmatcht))$ and the 
arc bimodules $\Arcalg(\compmatch,\compmatcht)$ 
are identified.

\begin{proposition}\label{proposition:match-ts}
Fixing the 
$\ring$-linear isomorphisms 
$\Iso{\cdot}{\cdot}$ from \eqref{eq:isoalgebrasnext}, it holds 
\begin{gather}\label{eq:again-something-4}
\Iso{\web(\compmatch^{\prime},\compmatcht^{\prime})}{(\compmatch^{\prime},\compmatcht^{\prime})}\circ\boldsymbol{\mathrm{rMult}}_{\web(\compmatch,\compmatcht)}^{\web(\compmatch^{\prime},\compmatcht^{\prime})}=
\boldsymbol{\mathrm{rmult}}_{\compmatcht}
^{\compmatcht^{\prime}}
\circ\Iso{\web(\compmatch,\compmatcht)}{(\compmatch,\compmatcht)}.
\end{gather}
\end{proposition}

\begin{proof}
Similar to the proof of 
\fullref{theorem:matchalgebrasold}. Indeed, we can use the same 
argumentation (noting that 
the shifting basic moves 
as in the first two columns of \cite[(31)]{EST} can be incorporated), but we turn the corresponding 
pictures by $\neatfrac{\pi}{2}$ (which gives the slight changes for the 
scalars in the algebraic setting). We skip the calculations for brevity.
\end{proof}

We now aim to match the bimodule maps for different 
specializations of $\parameterS$ as in 
\fullref{subsec:iso-alg} and \fullref{subsec:iso-bimod}. For this purpose, 
we define a coefficient map which is again slightly 
modified.
In particular, we use the same 
notations as in \fullref{definition:stacked2}.

\begin{definition}\label{definition:the-coefficient-map3}
For fixed $\stacked$, we define its 
\textit{reverse coefficient map} $\coeffb$ as
\begin{gather}\label{eq:label?}
\begin{aligned}
\coeffb_{\parsign}(C,\stacked^{\mathrm{or}}) \;=\;& \phantom{..}
{\textstyle\prod_{\gamma \in \ecups(C)}}\; 
\parsign^{s_{\Lambda}(\gamma)(\pos_{\Lambda}(\gamma)+1)} \cdot
{\textstyle\prod_{\gamma \in \ecaps(C)}}\; 
\parsign^{(s_{\Lambda}(\gamma)+1)\pos_{\Lambda}(\gamma)}
\\
&\cdot
{\textstyle\prod_{\gamma \in\irights(C)}}\; 
\parsign^{\pos_{\Lambda}(\gamma)+1}
\cdot 
{\textstyle\prod_{\gamma \in\ilefts(C)}}\;
\parsign^{\pos_{\Lambda}(\gamma)},
\\
\coeffb_{\paro}(C,\stacked^{\mathrm{or}}) \;=\;&
{\textstyle\prod_{\gamma \in \ecups(C)}}\;
\paro^{-s_{\Lambda}(\gamma)+1}
\cdot 
{\textstyle\prod_{\gamma \in \ecaps(C)}}\;
\paro^{s_{\Lambda}(\gamma)} \cdot 
\paro^{\# \left( \iright(C)\, {\scriptstyle \cup}\, \ileft(C) \right) }
\end{aligned}
\end{gather}
instead of \eqref{eq:the-crazy-map-def}, 
and a further factor of $\parsign^{\mathrm{t}(C)}$ for the clockwise circle.
\end{definition}

For the following proposition we use the evident 
notation to distinguish the reverse multiplication maps from 
\fullref{definition:reverse-mult} for different 
choices of specializations.

\begin{proposition}\label{proposition:crazy-module-maps-yet-again}
The homogeneous, $\ring$-linear map 
(defined as in \eqref{eq:the-crazy-map-modules}, but using $\coeffb$ instead 
of $\coeff$)
$\coeffb_{\compmatch,\compmatcht} \colon  
\Mo[\KBN](\compmatch,\compmatcht)
\to \Mo[\parameterS](\compmatch,\compmatcht)$
is an isomorphism of graded, free $\ring$-modules 
such that the following commutes:
\begin{gather}\label{eq:diagram-which-commutes}
\begin{aligned}
\xy
\xymatrix{
\Mo[\KBN](\compmatch,\compmatcht)
\ar[rr]^{\boldsymbol{\mathrm{rmult}}[\KBN]_{\compmatcht}^{\compmatcht^{\prime}}}
\ar[d]_{\coeffb_{\compmatch,\compmatcht}} && \Mo[\KBN](\compmatch^{\prime},\compmatcht^{\prime})
\ar[d]^{\coeffb_{\compmatch^{\prime},\compmatcht^{\prime}}}\\
\Mo[\parameterS](\compmatch,\compmatcht)
\ar[rr]_{\boldsymbol{\mathrm{rmult}}[\parameterS]_{\compmatch}^{\compmatcht^{\prime}}}
&& \Mo[\parameterS](\compmatch^{\prime},\compmatcht^{\prime}).
}
\endxy
\end{aligned}
\end{gather}
\end{proposition}

\begin{proof}
This is a dual version of the 
proof that 
$\coeff_{\compmatch,\compmatcht} \colon  
\Mo[\KBN](\compmatch,\compmatcht)
\to \Mo[\parameterS](\compmatch,\compmatcht)$ 
from \fullref{proposition:crazy-module-maps} 
intertwines the actions of $\ArcalgS[\KBN]$ and $\Arcalg$ 
(again checking the cases \ref{enum:nnm}-\ref{enum:ns} as in 
the proof of \fullref{theorem:matchalgebrasold}) with the 
following differences: 
the non-nested cases work analogously, while in the nested cases one needs 
to successively apply \fullref{lemma:comparesaddles-bimodules} 
as in the presented nested merge case in the proof 
of \fullref{proposition:crazy-module-maps}.
\end{proof}

Similar to \eqref{eq:coeff-in-action}, but
using \fullref{proposition:crazy-module-maps-yet-again} 
and the corresponding maps, 
we define
\begin{gather}\label{eq:again-something-10}
\coeffb_{\Isonew{}{}}\colon\Mo[\parameterS](\compmatch,\compmatcht)
\to\Mo[\KBN](\compmatch,\compmatcht)
\to\Mo[\parto,\specS(\parsign),\specS(\paro)](\compmatch,\compmatcht).
\end{gather}
The following is now clear because the proof 
of \fullref{proposition:crazy-module-maps-yet-again} 
does not use the specific form of the parameters in question.

\begin{corollary}\label{corollary:crazy-module-maps-yet-again}
The map $\coeffb_{\Isonew{}{}}$
is an isomorphism of graded, free $\ring$-modules such that 
the corresponding diagram in \eqref{eq:diagram-which-commutes} 
commutes.
\end{corollary}

\begin{example}\label{example:yup-needed!}
Denote the diagrams in \fullref{ex:yup-needed-yet-again} 
from left to right by $\stacked_1$, 
$\stacked_2$ and $\stacked_3$. Then
$\coeffb_{\stacked_1}(\stacked_1^{\mathrm{or}})=
\parsign\cdot\paro\cdot
\stacked_1^{\mathrm{or}}$,
$\coeffb_{\stacked_2}(\stacked_2^{\mathrm{or}})=
1\cdot
\stacked_2^{\mathrm{or}}$ 
and
$\coeffb_{\stacked_3}(\stacked_3^{\mathrm{or}})=
\parsign\cdot
\stacked_3^{\mathrm{or}}$.
Thus, we see that \eqref{eq:diagram-which-commutes} commutes in this example.
\end{example}

\begin{remark}\label{remark:new-coeff-intertwines}
We will see in \fullref{lemma:signs-are-constant} that 
$\coeffb_{\stacked}(\stacked^{\mathrm{or}})$ can be expressed 
in terms of $\coeff_{\stacked}(\stacked^{\mathrm{or}})$ times a 
constant that can either be determined by counting 
cups or by counting caps as well as shifts. Hence, it is 
evident that for $x \in \ArcalgS[\KBN]$ and any
$m \in \Mo[\KBN](\compmatch,\compmatcht)$ 
it holds that
$\coeffb_{\compmatch,\compmatcht}(x\cdot m)
=\coeff(x)\cdot\coeffb_{\compmatch,\compmatcht}(m)$. (Similarly for the right action.)
Thus, 
$\coeffb$ is a graded, free $\ring$-modules isomorphism
intertwining the two actions.
\end{remark}

\begin{lemma}\label{lemma:signs-are-constant-next}
The compositions 
\begin{gather}\label{eq:again-something-11}
\coeffb_{\Isonew{}{}}^{-1}\circ\,\coeff_{\Isonew{}{}},
\coeff_{\Isonew{}{}}^{-1}\circ\,\coeffb_{\Isonew{}{}}
\colon\Mo[\parameterS](\compmatch,\compmatcht)\to\Mo[\parameterS](\compmatch,\compmatcht)
\end{gather}
are $\Arcalg$-bimodule maps.
\end{lemma}

\begin{proof}
This follows from \fullref{lemma:signs-are-constant} via \fullref{remark:new-coeff-intertwines},
and \fullref{proposition:crazy-module-maps} 
(and, as before, that our arguments do not use the specific form of the 
parameters in question).
\end{proof}
\subsection{Consequences}\label{subsec:consequences}
\makeautorefname{theorem}{Theorems}

We now formulate some consequences of 
\fullref{theorem:crazy-isomorphism2}, \ref{theorem:matchalgebrasnew} 
and \ref{theorem:matchalgebrasold}. 
By construction, 
$\Upsilon$ from \fullref{proposition:cats-are-equal} 
gives rise to a based version $\Upsilon^{\based}$.

\makeautorefname{theorem}{Theorem}

\begin{proposition}\label{proposition:cats-are-equal-yes}
There is an equivalence of graded, $\ring$-linear $2$-categories
\begin{gather}\label{eq:again-something-12}
\Upsilon^{\based}\colon\F^\pm[\parameterS]\overset{\cong}{\longrightarrow}\webcatScircpm,
\end{gather}
which is bijective on objects.
Similarly for $\specS\colon\ring\to R$ 
such that either:
\begin{enumerate}[label=(\alph*)]

\item \label{enum:generic} $\specS(\parto)=\parto$ is generic (i.e. an indeterminate) or $\specS(\parto)=0$.

\item \label{enum:semisimple} $\specS(\parto)$ is invertible, 
$\sqrt{\specS(\parto)}\in R$ and $\neatfrac{1}{2}\in R$.

\end{enumerate}
The algebras in question 
are semisimple under the circumstances of \ref{enum:semisimple}.
\end{proposition}

The proof 
is given in \fullref{subsec:proof-four}. 
The main step is to calculate the ranks of hom-spaces 
between bimodules. By web evaluation \eqref{eq:web-eval} 
and braid closing, this does not depend on whether we are working 
in the upwards oriented setting or not.

\begin{theorem}\label{theorem:all-equivalent}
Let $R[\parto]$ and $\specS$ be as 
in \fullref{theorem:matchalgebrasnew}. 
Then there are equivalences 
of graded, $\ring$-linear $2$-categories
\begin{gather}\label{eq:again-something-13}
\F[\parameterS]\cong\F[\parto,\specS(\parsign),\specS(\paro)].
\end{gather}
(Similarly for any 
simultaneous specialization of $\parto$ 
satisfying the conditions \ref{enum:generic} or \ref{enum:semisimple} 
from \fullref{proposition:cats-are-equal-yes}.)
\end{theorem}

\begin{proof}
This is just assembling all the pieces. First we use 
\fullref{proposition:cats-are-equal-yes} to see that 
both sides are equivalent 
to the corresponding module categories of (specialized) web algebras. 
Then we use \fullref{theorem:matchalgebras} to 
translate it to the corresponding 
arc algebras. Finally, using
\fullref{theorem:matchalgebrasnew} 
provides the statement.
\end{proof}

If one works over $\Z[\parto]$, then 
\fullref{theorem:all-equivalent} 
shows that the $2$-categories 
coming from the $\KBN$ and $\Bl$ setups  
are equivalent.
Together with \fullref{proposition:special}, we get the following stronger result in case there exists a
square root of $-1$.

\begin{corollary}\label{corollary:all-equivalent}
There are equivalences 
of graded, $\Z[\parto,i]$-linear $2$-categories
\begin{gather}\label{eq:four2cats}
\F_{\Z[\parto,i]}[\KBN]\cong
\F_{\Z[\parto,i]}[\Ca]\cong
\F_{\Z[\parto,i]}[\CMW]\cong
\F_{\Z[\parto,i]}[\Bl].
\end{gather}
(Similarly, by using $\Z[\sqrt{\specS(\parto)}^{\pm 1},\neatfrac{1}{2},i]$, 
for any 
simultaneous specialization of $\parto$ 
satisfying the condition \ref{enum:semisimple} 
from \fullref{proposition:cats-are-equal-yes}.)
\end{corollary}
\section{Link and tangle invariants}\label{sec:further}
Given a tangle diagram $\overline{T}$,
we now construct a $6$\textit{-parameter chain complex}
$\Hgenpm{{{}_-}}\!\colon\overline{T}\mapsto\Hgenpm{\,\overline{T}\,}$
with values in $\webcatpm$. 
We show in \fullref{proposition:its-an-invariant!} 
that its homotopy type is a tangle invariant, i.e. well-defined on isotopy class es of tangles.

We will see that the $6$-parameter complex specializes to the 
original $\KBN$ complex $\HKBNpm{{{}_-}}$, 
as well as to the 
versions $\HCapm{{{}_-}}$, $\HCMWpm{{{}_-}}$ 
and $\HBlpm{{{}_-}}$.
Moreover, by the explicit isomorphisms in \fullref{sec:iso}, 
this complex allows comparison of its various specializations.

We will treat $\Hgenpm{{{}_-}}$ and $\Hgen{{{}_-}}$ 
simultaneously and use the symbol $\pmandnot$ if 
a distinction is not necessary, e.g. $\webalgboth$ means either 
$\webalgpm$ or $\webalg$.
\subsection{Tangles and tangled webs}
\label{subsec:tangle-definition}
Recall that
an (algebraic) oriented \textit{tangle diagram} 
can be defined as a morphism in the monoidal category 
monoidally generated by the objects $\{0,+,-\}$ (zero is a ``dummy'') and the $1$-morphisms:
\begin{gather}\label{eq:tangles-def}
\raisebox{0.075cm}{\xy
(0,0)*{\includegraphics[scale=.65]{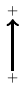}};
\endxy}
\quad,\quad
\xy
(0,0)*{\includegraphics[scale=.65]{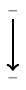}};
\endxy
\quad,\quad
\!\!
\xy
(0,0)*{\includegraphics[scale=.65]{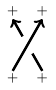}};
(0,6)*{\text{{\tiny$\phantom{a}$}}};
(0,-6)*{\text{{\tiny $+$ crossing}}};
\endxy
\!\!
\quad,\quad
\!\!
\xy
(0,0)*{\includegraphics[scale=.65]{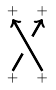}};
(0,6)*{\text{{\tiny$\phantom{a}$}}};
(0,-6)*{\text{{\tiny $-$ crossing}}};
\endxy
\!\!
\quad,\quad
\xy
(0,0)*{
\xy
(0,0)*{\includegraphics[scale=.65]{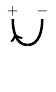}};
\endxy
\quad,\quad
\xy
(0,0)*{\includegraphics[scale=.65]{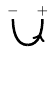}};
\endxy};
(0,6)*{\text{{\tiny$\phantom{a}$}}};
(0,-6)*{\text{{\tiny cups}}};
\endxy
\quad,\quad
\xy
(0,0)*{
\xy
(0,0)*{\includegraphics[scale=.65]{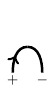}};
\endxy
\quad,\quad
\xy
(0,0)*{\includegraphics[scale=.65]{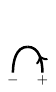}};
\endxy};
(0,6)*{\text{{\tiny$\phantom{a}$}}};
(0,-6)*{\text{{\tiny caps}}};
\endxy
\end{gather}
The third generator is called a \textit{positive crossing} and the fourth 
a \textit{negative crossing}; the cups and caps are $1$-morphisms with source respectively target $0$. 

We only consider tangle diagrams with an 
even number of bottom and top boundary points, called \textit{even tangle diagrams}. 
Local pictures with an odd number of strands can be fit into this framework 
by adding a ``dummy'' strand 
e.g. to the right.

\begin{remark}\label{remark:all-tangles}
The restriction to even tangle diagrams comes from the fact 
that we work with $\webalgboth[{{}_-}]$ and $\ArcalgS[{{}_-}]$. One could treat arbitrary tangles 
by using the generalized algebras (by including rays as e.g. in 
\cite{BS1}): it is clear by our results that one can 
follow \cite[Section 5]{CK} or \cite[Section 5]{Str2} 
to define parameter dependent complexes 
of bimodules for these generalized algebras giving 
rise to an invariant of arbitrary tangles. 
For brevity, we have decided to only consider $\webalgboth[{{}_-}]$ and $\ArcalgS[{{}_-}]$,
and treat these in detail. 
\end{remark} 

We study the following category, cf. \cite[Theorem XII.2.2]{Kas}.

\begin{definition}\label{definition:tanglecat}
The \textit{category of tangles} 
$\Tan$ is defined as follows

\begin{enumerate}[label=(\roman*)]

\item Objects are sequences $\vec{s},\vec{t}\in\{0,+,-\}^{\Z}$ 
with only a finite and even number of non-zero entries 
including $\emptyset=(\dots,0,0,0,\dots)$.

\item $1$-Morphisms 
from $\vec{s}$ to $\vec{t}$ are all tangle 
diagrams with source $\vec{s}$ and target $\vec{t}$ (after ignoring infinitely $0$'s) modulo some relations.

\item The relations are the usual \textit{tangle moves} 
(\cite[Section XII.2]{Kas}).

\item Composition of tangles is given via the evident gluing.
\end{enumerate}
\end{definition}

More precisely, the tangle moves are (see  \eqref{eq:wrong-r2} for illustrations):
\begin{gather}\label{enum:reide1}
\text{``Isotopies'', see \cite[(2.1) to (2.3)]{Kas}.}
\end{gather}
\begin{gather}\label{enum:reide2}
\begin{gathered}
\text{The Reidemeister moves \textbf{R1}, \textbf{R2}, 
and \textbf{R3}, see \cite[(2.4) to (2.6)]{Kas}.}
\\
\text{(\textbf{R2} 
and \textbf{R3} seen as braid moves, i.e. pointing only upwards.)}
\end{gathered}
\end{gather}
\begin{gather}\label{enum:reide3}
\begin{gathered}
\text{Mixed moves \textbf{mR2} \cite[(2.7) and  (2.8)]{Kas}.}\\  \text{(\textbf{R2} with two oppositely oriented strands.)}
\end{gathered}
\end{gather}

We call equivalence classes of tangle diagrams also \textit{tangles}. The elements from $\End_{\Tan}(\emptyset)$ 
are called \textit{links}. We denote by ${}_{\vec{s}}T_{\vec{t}}$ a $1$-morphism in 
$\Hom_{\Tan}(\vec{s},\vec{t})$ and by ${}_{\vec{s}}\overline{T}_{\vec{t}}$ 
a choice of a diagram representing ${}_{\vec{s}}T_{\vec{t}}$. 
Note that one can consistently define \textit{sideways and downwards oriented crossings} using the 
generators from \eqref{eq:tangles-def} in the usual way, cf. \eqref{eq:sideways-crossing}.

We also need an extension of the notion of webs from \fullref{subsec:foams}. 
We let \textit{(upwards oriented) tangled webs}, 
 be defined as (upwards oriented) webs, but 
additionally allowing
\begin{gather}\label{eq:crossings-webs}
\xy
(0,0)*{\includegraphics[scale=.65]{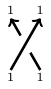}};
(0,6)*{\text{{\tiny$\phantom{a}$}}};
(0,-6)*{\text{{\tiny $+$ crossing}}};
\endxy
\!\!
\quad,\quad
\!\!
\xy
(0,0)*{\includegraphics[scale=.65]{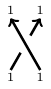}};
(0,6)*{\text{{\tiny$\phantom{a}$}}};
(0,-6)*{\text{{\tiny $-$ crossing}}};
\endxy
\!\!
\quad,\quad
\xy
(0,0)*{
\xy
(0,0)*{\includegraphics[scale=.65]{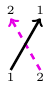}};
\endxy
\quad,\quad
\xy
(0,0)*{\includegraphics[scale=.65]{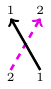}};
\endxy
\quad,\quad
\xy
(0,0)*{\includegraphics[scale=.65]{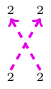}};
\endxy};
(0,6)*{\text{{\tiny$\phantom{a}$}}};
(0,-6)*{\text{{\tiny phantom crossing}}};
\endxy
\end{gather}
called \textit{positive crossing}, \textit{negative crossing} 
and \textit{phantom crossings}. For example,
\begin{gather}\label{eq:sideways-crossing}
\xy
(0,0)*{
\xy
(0,0)*{\includegraphics[scale=.65]{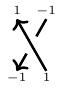}};
\endxy
=
\xy
(0,0)*{\includegraphics[scale=.65]{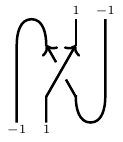}};
\endxy
\quad,\quad
\xy
(0,0)*{\includegraphics[scale=.65]{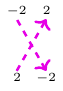}};
\endxy
=
\xy
(0,0)*{\includegraphics[scale=.65]{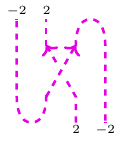}};
\endxy};
(0,10)*{\text{{\tiny$\phantom{a}$}}};
(0,-10)*{\text{{\tiny in the ${}^\pm$ setting}}};
\endxy
\quad,\quad
\xy
(0,0)*{
\xy
(0,0)*{\includegraphics[scale=.65]{figs/fig5-sideways1.pdf}};
\endxy
=
\xy
(0,0)*{\includegraphics[scale=.65]{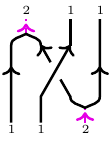}};
\endxy};
(0,10)*{\text{{\tiny$\phantom{a}$}}};
(0,-10)*{\text{{\tiny upwards oriented}}};
\endxy
\end{gather}
We write $\Hom_{\mathrm{tw}}^\pm(\vec{k},\vec{l})$ 
respectively $\Hom_{\mathrm{tw}}(\vec{k},\vec{l})$
for the set of tangled webs respectively 
upwards oriented tangled webs with bottom sequences $\vec{k}$ 
and top sequence $\vec{l}$.

Note that each tangle diagram can be viewed as a tangled web, so we can denote by $\web^{\pm}({}_-)$ the corresponding inclusion.
Moreover, we also define a map 
$\web({}_-)\colon {}_{\vec{s}}\overline{T}_{\vec{t}}\mapsto\web({}_{\vec{s}}\overline{T}_{\vec{t}})$ 
to upwards pointing tangled webs in $\Hom_{\mathrm{tw}}(\vec{k},\vec{l})$ 
with $\vec{k}=\omega_{\ell+s}+\omega_{s}$ and 
$\vec{l}=\omega_{\ell+t}+\omega_{t}$ via
\begin{gather}\label{eq:from-tangles-to-webs}
\begin{gathered}
\phantom{.}
\xy
(0,0)*{\includegraphics[scale=.65]{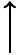}};
\endxy
\rightsquigarrow
\xy
(0,0)*{\includegraphics[scale=.65]{figs/fig5-14.pdf}};
\endxy
\\
\phantom{.}
\xy
(0,0)*{\includegraphics[scale=.65]{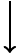}};
\endxy
\rightsquigarrow
\xy
(0,0)*{\includegraphics[scale=.65]{figs/fig5-14.pdf}};
\endxy
\end{gathered}
\quad,\quad
\begin{gathered}
\phantom{.}
\xy
(0,0)*{\includegraphics[scale=.65]{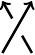}};
\endxy
\rightsquigarrow
\xy
(0,0)*{\includegraphics[scale=.65]{figs/fig5-16.pdf}};
\endxy
\\
\phantom{.}
\xy
(0,0)*{\includegraphics[scale=.65]{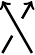}};
\endxy
\rightsquigarrow
\xy
(0,0)*{\includegraphics[scale=.65]{figs/fig5-17.pdf}};
\endxy
\end{gathered}
\quad,\quad
\begin{gathered}
\phantom{.}
\xy
(0,0)*{\includegraphics[scale=.65]{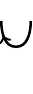}};
\endxy
\rightsquigarrow
\xy
(0,0)*{\includegraphics[scale=.65]{figs/fig4-4.pdf}};
\endxy
\\[-10pt]
\phantom{.}
\xy
(0,0)*{\includegraphics[scale=.65]{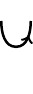}};
\endxy
\rightsquigarrow
\xy
(0,0)*{\includegraphics[scale=.65]{figs/fig4-4.pdf}};
\endxy
\end{gathered}
\quad,\quad
\begin{gathered}
\phantom{.}
\xy
(0,0)*{\includegraphics[scale=.65]{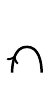}};
\endxy
\rightsquigarrow
\xy
(0,0)*{\includegraphics[scale=.65]{figs/fig4-6.pdf}};
\endxy
\\[-10pt]
\phantom{.}
\xy
(0,0)*{\includegraphics[scale=.65]{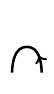}};
\endxy
\rightsquigarrow
\xy
(0,0)*{\includegraphics[scale=.65]{figs/fig4-6.pdf}};
\endxy
\end{gathered}
\end{gather}
By using the phantom crossings from \eqref{eq:crossings-webs}, 
one can always rearrange everything such that one can start in 
$\omega_{\ell+s}+\omega_{s}$ and end in $\omega_{\ell+t}+\omega_{t}$. 
This association is far from being unique, but what we are going to do will not 
depend on the choice of the map $\web({}_-)$. We can make any such choice, e.g. extending the construction from \fullref{subsec:algmodel}.

\begin{example}\label{example:from-tangles-to-webs}
The following illustrates the two maps:
\begin{gather}\label{eq:blueprint?}
\xy
(0,0)*{\includegraphics[scale=.65]{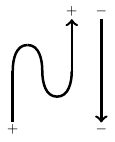}};
\endxy
\quad 
\stackrel{\web^\pm}{\leftsquigarrow} 
\quad
\xy
(0,0)*{\includegraphics[scale=.65]{figs/fig5-22.pdf}};
\endxy
\quad 
\stackrel{\web}{\rightsquigarrow} 
\quad
\xy
(0,0)*{\includegraphics[scale=.65]{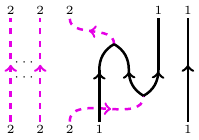}};
\endxy
\end{gather}
In the upwards oriented case, we have to pull the phantom edge from left to right.
\end{example}
\subsection{The \texorpdfstring{$6$}{6}-parameter complex}
\label{subsec:tangle-invariants}
Fix an additive, $\field$-linear, graded $2$-category $\mathfrak{X}$. 
A \textit{chain complex} $(C_i,d_i)_{i\in\Z}$ 
\textit{with values in} $\mathfrak{X}$ 
is a chain complex whose chain groups $C_i$ are the $1$-morphisms from 
$\mathfrak{X}$ and whose 
differentials $d_{i}$ are $2$-morphisms of $\mathfrak{X}$
such that $d_{i+1}\circ d_i=0$ for all $i\in\Z$. 
Such a complex $(C_i,d_i)_{i\in\Z}$ is called \textit{bounded}, if $C_i=0$ for $|i|\gg 0$.
Denote by $\Kom(\mathfrak{X})$ \textit{the category 
of bounded complexes with values in} 
$\mathfrak{X}$ (with chain maps given by $2$-morphisms). Further, we denote by $\Komh(\mathfrak{X})$ 
the bounded homotopy category.

\begin{definition}\label{definition:basic-complexes}
The \textit{basic complexes} are objects in 
$\Kom(\webcatcircboth)$ of the form
\begin{gather}
\Hgenboth{
\xy
(0,0)*{\includegraphics[scale=.65]{figs/fig5-9.pdf}};
\endxy
}
=
\xy
\xymatrix{
\underline{
\xy
(0,0)*{\includegraphics[scale=.65]{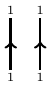}};
\endxy
}\langle+1\rangle \ar[r]^/-.10cm/{
\xy
(0,0)*{\includegraphics[scale=.6]{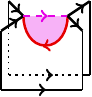}};
\endxy
}
&
\raisebox{.075cm}{\xy
(0,0)*{\includegraphics[scale=.65]{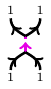}};
\endxy}
\langle+2\rangle,
}
\endxy
\quad
\Hgenboth{
\xy
(0,0)*{\includegraphics[scale=.65]{figs/fig5-10.pdf}};
\endxy
}
=
\xy
\xymatrix{
\raisebox{.075cm}{\xy
(0,0)*{\includegraphics[scale=.65]{figs/fig5-25.pdf}};
\endxy}
\langle-2\rangle \ar[r]^/-.10cm/{
\xy
(0,0)*{\includegraphics[scale=.6]{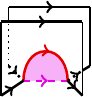}};
\endxy
}
&
\underline{
\xy
(0,0)*{\includegraphics[scale=.65]{figs/fig5-24.pdf}};
\endxy
}\langle-1\rangle,}
\endxy
\label{eq:basic-complexes1}
\\
\Hgenboth{
\xy
(0,0)*{\includegraphics[scale=.65]{figs/fig5-11.pdf}};
\endxy
}
=
\underline{
\xy
(0,0)*{\includegraphics[scale=.65]{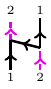}};
\endxy
}
\quad,\quad
\Hgenboth{
\xy
(0,0)*{\includegraphics[scale=.65]{figs/fig5-12.pdf}};
\endxy
}
=
\underline{
\xy
(0,0)*{\includegraphics[scale=.65]{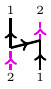}};
\endxy
}
\quad,\quad
\Hgenboth{
\xy
(0,0)*{\includegraphics[scale=.65]{figs/fig5-13.pdf}};
\endxy
}
=
\underline{
\xy
(0,0)*{\includegraphics[scale=.65]{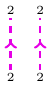}};
\endxy
}
\label{eq:basic-complexes3}
\end{gather}
where the underlined terms are in homological degree zero.
\end{definition}

Using these, we can associate as in \cite[Section 3.4]{Khov}, but now with parameters, to any 
$u_t\in\Hom_{\mathrm{tw}}^\pmandnot(\vec{k},\vec{l})$ an object 
$\Hgenboth{u_t}$ via the usual tensor products of the basic complexes.
In particular, we can
let $\Hgenboth{{}_{\vec{s}}\overline{T}_{\vec{t}}}=
\Hgenboth{\web^\pmandnot({}_{\vec{s}}\overline{T}_{\vec{t}})}$.
For $\star$ being $\pm$ or $+$,
we get a homological 
invariant of tangles:

\begin{proposition}\label{proposition:its-an-invariant!}
The assignment $\Hgenboth{{{}_-}}$
defines  a functor
\begin{gather}\label{eq:funny-functor}
\Hgenboth{{}_-}\colon\Tan\to\Komh(\webcatcircboth).
\end{gather}
\end{proposition}

The proof of \fullref{proposition:its-an-invariant!} follows the 
usual strategy. It is given in 
\fullref{subsec:proof-five}.

\begin{corollary}\label{example:our-friends}
We obtain all the link invariants from \eqref{eq:main-specs}:
\begin{gather}\label{eq:based-chain}
\begin{gathered}
\left\llbracket {}_{\vec{s}}\overline{T}_{\vec{t}}\right\rrbracket^\pmandnot_{(\parto,\parha,1,1,1,1)}
\cong_{\Z[\parto,\parha]}
\HKBNboth{{}_{\vec{s}}\overline{T}_{\vec{t}}},
\quad\quad
\left\llbracket {}_{\vec{s}}\overline{T}_{\vec{t}}\right\rrbracket^\pmandnot_{(\parto,\parha,1,1,i,-i)}
\cong_{\Z[\parto,\parha,i]}
\HCaboth{{}_{\vec{s}}\overline{T}_{\vec{t}}},
\\
\left\llbracket {}_{\vec{s}}\overline{T}_{\vec{t}}\right\rrbracket^\pmandnot_{(\parto,\parha,1,1,i,-i)}
\cong_{\Z[\parto,\parha,i]}
\HCMWboth{{}_{\vec{s}}\overline{T}_{\vec{t}}},
\quad\quad
\left\llbracket {}_{\vec{s}}\overline{T}_{\vec{t}}\right\rrbracket^\pmandnot_{(\parto,\parha,1,-1,1,-1)}
\cong_{\Z[\parto,\parha]}
\HBlboth{{}_{\vec{s}}\overline{T}_{\vec{t}}}.
\end{gathered}
\end{gather}
Hereby we denote specializations in 
our usual $6$-term notation, and we have indicated 
the rings one needs to work over for the isomorphisms to make sense.
\end{corollary}
\subsection{Comparison results}
\label{subsec:gl-and-sl-compare}
\makeautorefname{theorem}{Theorems}

We now relate the complexes for various parameters. Let $S$ be a ring and $S_1$ and $S_2$ be two specializations such that the associated parameter complexes $\llbracket{{}_-}\rrbracket_{S_1}$ and $\llbracket{{}_-}\rrbracket_{S_2}$ are defined. By \fullref{proposition:its-an-invariant!} these are tangle invariants.
We call 
these two link invariants \textit{related over $S$}, and write $\llbracket{{}_-}\rrbracket_{S_1}\approx_S\llbracket{{}_-}\rrbracket_{S_2}$, if we have, after scalar extension to $S$, a commuting diagram:
\begin{gather}\label{eq:funny-diagram}
\begin{aligned}
\xymatrix@R4mm{
& \Komh(\Modpgr{\webalg_{\based}[S_1]})\ar[dd]^{\boldsymbol{\Isonew{}{}}}_{\cong}\\
\Tan \ar[ru]|/-.5cm/{\,\llbracket{{}_-}\rrbracket_{S_1}\,}\ar[rd]|/-.5cm/{\,\llbracket{{}_-}\rrbracket_{S_2}\,}&\\
& \Komh(\Modpgr{\webalg_{\based}[S_2]}).
}
\end{aligned}
\end{gather}

\begin{proposition}\label{proposition:all-the-same}
Denote by $\HKBN{{{}_-}}$ the functor obtained via specializing 
$\specS(\parto)=\parto$, $\specS(\parsign)=1$ and $\specS(\paro)=1$ (and scalar extension to $\ring$). Let 
\begin{gather}\label{eq:funny-functor-2}
\boldsymbol{\Isonew{}{}}
\colon
\Komh(\Modpgr{\webalg_{\based}[\KBN]})\stackrel{\cong}{\longrightarrow}
\Komh(\Modpgr{\webalg_{\based}[\parameterS]})
\end{gather}
be the equivalence obtained by combining 
\fullref{theorem:matchalgebras} and \ref{theorem:matchalgebrasnew}. Then 
$\HKBN{{{}_-}}\approx_\ring\llbracket{{}_-}\rrbracket_{\ring}$.
\end{proposition}

\begin{proof}
Let ${}_{\vec{s}}T_{\vec{t}}$ be a tangle. 
We observe that the entries of the two 
complexes $\boldsymbol{\Isonew{}{}}(\HKBN{{}_{\vec{s}}T_{\vec{t}}})$
and $\HGen{{}_{\vec{s}}T_{\vec{t}}}$
can by identified via 
\fullref{theorem:matchalgebras} and \ref{theorem:matchalgebrasnew}. Moreover, the differentials are given by ``multiplication foams'' 
$f\colon\M(u)\to\M(v)$ 
or ``reversed multiplication foams'' $g\colon\M(v)\to\M(u)$
and match up to some small twists.
To be more precise, the differentials restricted to direct summands are related by
\begin{gather}\label{eq:funny-functor-4}
\begin{gathered}
f\colon\M[\KBN](u)\to\M[\KBN](v)
\overset{\boldsymbol{\Isonew{}{}}}{\longmapsto}
\coeff_{v}\circ f\circ\coeff_{u}^{-1}\colon\M[\parameterS](u)\to\M[\parameterS](v),
\\
g\colon\M[\KBN](v)\to\M[\KBN](u)
\overset{\boldsymbol{\Isonew{}{}}}{\longmapsto}
\coeff_{u}\circ g\circ\coeff_{v}^{-1}\colon\M[\parameterS](v)\to\M[\parameterS](u).
\end{gathered}
\end{gather}
Here we use the notation from \fullref{sec:iso}, but now 
for the web algebras side instead of the arc algebra used therein.
In other words, $\boldsymbol{\Isonew{}{}}(f)=f_{\parameterS}$ and $g_{\parameterS}=\coeffb_{v}\circ\boldsymbol{\Isonew{}{}}(g)\circ\coeffb^{-1}_{u}$, as follows by \fullref{proposition:crazy-module-maps} and \fullref{proposition:crazy-module-maps-yet-again}, respectively.

Thus, the differentials of two complexes $\boldsymbol{\Isonew{}{}}(\HKBN{{}_{\vec{s}}T_{\vec{t}}})$
and $\HGen{{}_{\vec{s}}T_{\vec{t}}}$ differ only by ``units'', and we can use unit sprinkling (see \cite[Lemma 4.5]{CMW}) to get a 
chain isomorphism between them. Note that the maps used in this chain isomorphism are entrywise 
$\webalg[\parameterS]$-bimodule homomorphisms by 
\fullref{lemma:signs-are-constant-next}.
\end{proof}

\makeautorefname{theorem}{Theorem}

\begin{theorem}\label{theorem:all-the-same}
Let $S=S_1=\ring$ and consider a specialization $\specS\colon\ring\to S_2$ with $\specS(\parto)=\parto$. Then 
$\llbracket{{}_-}\rrbracket_{S_1}\approx_S\llbracket{{}_-}\rrbracket_{S_2}$.
\end{theorem}

\begin{proof}
Exactly as in the proof of 
\fullref{proposition:all-the-same}, since
we have not used 
the specific form of the parameters in question.
\end{proof}

The following is a consequence of \fullref{theorem:all-the-same}:

\begin{corollary}\label{corollary:all-the-same}
We have (with the last $\approx$ only for $R=\C$ and $\specS(\parto)=0$)
\begin{gather}\label{eq:funny-functor-12}
\HKBNboth{{}_-}
\,
\approx_{\Z[\parto,i]}
\,
\HCaboth{{}_-}
\,
\approx_{\Z[\parto,i]}
\,
\HCMWboth{{}_-}
\,
\approx_{\Z[\parto,i]}
\,
\HBlboth{{}_-}
\,
\approx_{\C,\specS(\parto)=0}
\,
\HOboth{{}_-},
\end{gather}
where the rightmost invariant is the one obtained via category $\mathcal{O}$.
\end{corollary}

This was known for links, but, to the 
best of our knowledge, not for tangles (except for $\HKBNboth{{}_-}
\,
\approx_{\C,\specS(\parto)=0}
\,
\HOboth{{}_-}$, see \cite{Str} together with \cite{BS3}).
\section{Functoriality}\label{subsec:gl-and-sl-funct}
Let us denote by $\twoTan$ 
the $2$\textit{-category of tangle (diagrams)}. This is the $2$-category 
whose underlying $1$-category is $\Tan$, but without relations on tangles (tangle diagram which are equal in $\Tan$ are isomorphic in $\twoTan$), and whose 
$2$-morphisms are certain cobordisms called $2$\textit{-tangles}. 
There is a generator-relation 
description of $\twoTan$, with generators given by 
\textit{Morse and Reidemeister cobordisms} and 
relations given by \textit{movie moves}, see \cite[Section 8]{BN1} 
or \cite[Chapter 1]{CS} for more details.
We will recall the specific facts needed for the proof of 
\fullref{theorem:functor-gl2} in \fullref{subsec:proof-six}.  
We also consider the $2$-subcategory $\twoTanup$ of $\twoTan$ defined in the expected way, with the underlying 
$1$-morphisms being (upwards oriented) braids and $2$-morphisms given by braid-like cobordisms, 
\cite[Section 3.4]{CS}.
\subsection{Functoriality (\texorpdfstring{$\sltwo$}{sl2} versus \texorpdfstring{$\gltwo$}{gl2})}
In light of 
\fullref{proposition:its-an-invariant!}, it makes sense to ask, 
whether there is a $2$-functor
\begin{gather}\label{eq:funny-functor-21}
\Hgenpm{{{}_-}}\colon\twoTan\to\twoKomh(\webcatcircpm).
\end{gather}
Here $\twoKomh(\webcatcircpm)$ is the $2$-category where we take $\Komh(\webcatcircpm)$ on the level of $1$- and $2$-morphisms and take additionally (finite sequences of) boundary points as objects. The existence of 
such a $2$-functor is called \textit{functoriality}.

Let us now consider two families of specializations,
called the \textit{$\sltwo$} respectively \textit{$\gltwo$ specialization}, together with the corresponding  tangle invariants $\Hgenslpm{{{}_-}}$ and $\Hgenglpm{{{}_-}}$.
We define these specializations $\field\to\Z[\parto,\parha,\paro^{\pm 1}]$ by specifying the images of the $6$ parameters as a $6$-tuple:
\begin{gather}\label{eq:slgl-parameters}
\specs_{\sltwo}(\parameter)=
(\parto,\parha,1,\parsign\paro^2=+1,\paro,\parsign\paro)
\quad\text{and}\quad
\specs_{\gltwo}(\parameter)=
(\parto,\parha,1,\parsign\paro^2=-1,\paro,\parsign\paro)
.
\end{gather}
These are $3$-parameter specializations which differ by
the sign in the fourth entry.

\begin{theorem}\label{theorem:functor-gl2}
Functoriality holds for $\gltwo$, i.e. the $\gltwo$ specialization gives rise to 
a $2$-functor
\begin{gather}\label{eq:funny-functor-22}
\Hgenglpm{{{}_-}}\colon\twoTan\to\twoKomh(\webcatcircglpm),
\end{gather}
Restriction to the subcategory $\twoTanup$ defines a $2$-functor
\begin{gather}\label{eq:funny-functor-32}
\Hgengl{{{}_-}}\colon\twoTanup\to\twoKomh(\webcatcircgl).
\end{gather}
For $\sltwo$, similar assignments do not provide a well-defined extension  of $\Hgenslpm{{{}_-}}$.
\end{theorem}

The proof is 
surprisingly easy---using the ideas in \cite{ETWe}---and 
is given in \fullref{subsec:proof-six}.
\subsection{A fix of non-functoriality of Khovanov homology}
We like to emphasize that \fullref{theorem:functor-gl2} gives a way to fix functoriality of Khovanov homology 
without changing the framework of $\KBN$, or 
any other specialization of the 
$\sltwo$ family. 
Namely, use any of the functorial $\gltwo$ invariants from \fullref{theorem:functor-gl2} 
and ``pull it over'' via $\boldsymbol{\Isonew{}{}}$ from \eqref{eq:funny-diagram}. 
To be more precise, one uses the coefficient maps 
from \fullref{definition:the-coefficient-map2} 
on the chain groups (bimodules) to get a different, scalar adjusted, 
cup foam basis. Then one can rearrange the differentials 
(web bimodule homomorphisms).
The resulting complex is functorial by \fullref{theorem:functor-gl2}.
The resulting $2$-functors $\Hgenslo{{{}_-}}$, which we define now, might be called \textit{the functorial versions of $\sltwo$ homology}. 

Similarly as in \eqref{eq:funny-diagram}, we call two link invariants \textit{$2$-related over $S$}, and write $\llbracket{{}_-}\rrbracket_{S_1}\approx_S^2\llbracket{{}_-}\rrbracket_{S_2}$ if we have a commuting diagram
\begin{gather}\label{eq:funny-diagram-two}
\begin{aligned}
\xymatrix@R4mm{
& \twoKomh(\Modpgr{\webalg_{\based}^{\pm}[S_1]})\ar[dd]^{\boldsymbol{\Isonew{}{}}}_{\cong}\\
\twoTan \ar[ru]|/-.5cm/{\,\llbracket{{}_-}\rrbracket_{S_1}\,}\ar[rd]|/-.5cm/{\,\llbracket{{}_-}\rrbracket_{S_2}\,}&\\
& \twoKomh(\Modpgr{\webalg_{\based}^{\pm}[S_2]}).
}
\end{aligned}
\end{gather}

\begin{definition}\label{definition:sl2funct}
A link homology such that $\Hgenslo{{{}_-}}\approx_{\Z[\parto,\parha,\paro^{\pm 1}]}^2\Hgengl{{{}_-}}$ is called a \textit{functorial $\sltwo$ homology}.
\end{definition}

\begin{corollary}\label{corollary:sl2funct}
The link homology $\Hgenslo{{{}_-}}$ is functorial.
\end{corollary}

This construction should not be confused with the (non-functorial) complex $\Hgensl{{{}_-}}$. An example of a functorial $\sltwo$ homology is the classical Khovanov homology, but with
scalar adjusted chain maps.
In contrast to the classical Khovanov homology the merge and split morphisms in the functorial Khovanov homology get signs depending on the underlying shape, see \eqref{eq:mult1}--\eqref{eq:mult4} for examples. Note that this functorial Khovanov homology does not change the underlying combinatorial framework of the classical Khovanov homology, which is in contrast to the fixes  $\Ca$, $\CMW$ and $\Bl$ of non-functoriality.
\section{Main proofs}\label{sec:mainproofs}
In this final section we give the more involved proofs 
of our main statements. 
\subsection{Proof of \texorpdfstring{\fullref{theorem:matchalgebrasold}}{Theorem \ref{theorem:matchalgebrasold}}}\label{subsec:proof-one}
Our proof follows \cite[proof of Theorem 4.18]{EST}. 
That is, we show that 
each step in the multiplication 
procedure of the web algebra locally agrees 
with the one from the multiplication of the arc algebra. 
Here we use \fullref{lemma:matchvs2}, i.e. 
it suffices to show that they agree on the cup foam basis on the side of 
$\webalg[\parameterS]^{\based}$.
That is, throughout the whole proof, we 
first check the multiplication steps for $\webalg[\parameterS]^{\based}$ 
where some rewriting has to be done, and then for 
$\Arcalg$ where we can read off the multiplication directly.

Now, there are four
topologically different situations to check:
\begin{enumerate}
\renewcommand{\theenumi}{(i)}
\renewcommand{\labelenumi}{\theenumi}

\item \label{enum:nnm} \textbf{Non-nested merge.} Two non-nested circles are replaced by one circle.

\renewcommand{\theenumi}{(ii)}
\renewcommand{\labelenumi}{\theenumi}

\item \label{enum:nm} \textbf{Nested merge.} Two nested circles are replaced by one circle.

\renewcommand{\theenumi}{(iii)}
\renewcommand{\labelenumi}{\theenumi}

\item \label{enum:nns} \textbf{Non-nested split.} One circle is replaced by two non-nested circles.

\renewcommand{\theenumi}{(iv)}
\renewcommand{\labelenumi}{\theenumi}

\item \label{enum:ns} \textbf{Nested split.} One circle is replaced by two nested circles.
\end{enumerate}
As in \cite[proof of Theorem 4.18]{EST}, 
we will go through the following cases:

\begin{enumerate}[label=(\Alph*)]

\renewcommand{\theenumi}{(A)}
\renewcommand{\labelenumi}{\theenumi}

\item \label{enum:basics} \textbf{Basic shape.} The involved 
components are as small as possible with 
the minimal number of phantom edges.

\renewcommand{\theenumi}{(B)}
\renewcommand{\labelenumi}{\theenumi}

\item \label{enum:mimimalsad} \textbf{Minimal saddle.} While the 
components themselves are allowed to be 
of any shape, the involved saddle has 
only a single phantom facet.

\renewcommand{\theenumi}{(C)}
\renewcommand{\labelenumi}{\theenumi}

\item \label{enum:gens} \textbf{General case.} Both, the 
shape as well as the saddle, are arbitrary.
\end{enumerate}

We start with \ref{enum:basics}, 
where---including a horizontal flip of the \ref{enum:nm}
case---we have the following pairs of (topological and algebraic) basic shapes corresponding to the cases $(i)$-$(iv)$:
\begin{gather}\label{eq:proofs-part-1}
\underbrace{\!\xy
(0,0)*{\includegraphics[scale=.65]{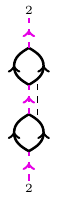}};
\endxy
\leftrightsquigarrow
\xy
(0,0)*{\includegraphics[scale=.65]{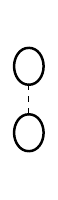}};
\endxy\!}_{\ref{enum:nnm}}
\;,\;
\underbrace{\!\xy
(0,0)*{\includegraphics[scale=.65]{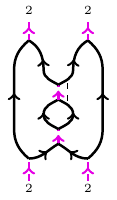}};
\endxy
\leftrightsquigarrow
\xy
(0,0)*{\includegraphics[scale=.65]{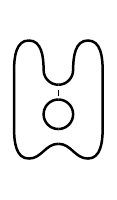}};
\endxy\!}_{\ref{enum:nm}}
\;,\;
\underbrace{\!\xy
(0,0)*{\includegraphics[scale=.65]{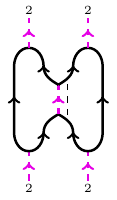}};
\endxy
\leftrightsquigarrow
\xy
(0,0)*{\includegraphics[scale=.65]{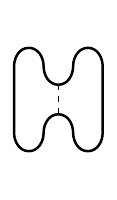}};
\endxy\!}_{\ref{enum:nns}}
\;,\;
\underbrace{\!\xy
(0,0)*{\includegraphics[scale=.65]{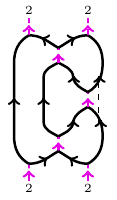}};
\endxy
\leftrightsquigarrow
\xy
(0,0)*{\includegraphics[scale=.65]{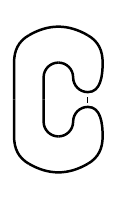}};
\endxy\!}_{\ref{enum:ns}}
\end{gather}
 The two rightmost situations are called
\textit{H-shape} and \textit{C-shape} respectively. (Recall that, by our convention, 
the $\reflectbox{\text{C}}$-shape 
does not occur.)

\noindent\textbf{Non-nested merge - basic shape.}
This case works almost exactly in the same way as in \cite[proof of Theorem 4.18]{EST}. 
That is, multiplication 
of basis cup foams yields topologically basis cup foams again, 
except in case where we start with two dotted basis cup foams. 
But in this case we can use \eqref{eq:theusualrelations1b} to create 
a basis cup foam without dots and a factor $\parto$.
The same happens for $\Arcalg$, see \eqref{eq:mult1}.

\noindent\textbf{Nested merge - basic shape.}
In this case something has to be done on the side of 
$\webalg[\parameter]^{\based}$. 
In fact, this is the most complicated case and we go through 
all details and will be shorter in the other 
cases afterward.
The multiplication step here is given by 
\begin{gather}\label{eq:proofs-part-2}
\xy
\xymatrix@C+=1.3cm@L+=6pt{
\xy
(0,0)*{\includegraphics[scale=.65]{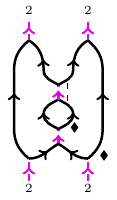}};
\endxy
\ar[rr]|{\;
\xy
(0,0)*{\includegraphics[scale=.65]{figs/fig2-91.pdf}};
\endxy
\;} & &  
\xy
(0,0)*{\includegraphics[scale=.65]{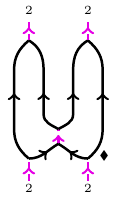}};
\endxy
\ar[rr]|{\;
\xy
(0,0)*{\includegraphics[scale=.65]{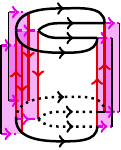}};
\endxy
\;}
& &  
\xy
(0,0)*{\includegraphics[scale=.65]{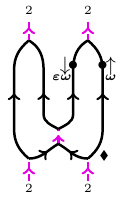}};
\endxy
}
\endxy
\end{gather}
The right foam above
is shown to illustrate the cylinder we cut 
together with the dots, their 
positions---with $\downarrow$ and $\uparrow$ meaning the dot sits on a facet touching 
the corresponding edge under or over
the part where we cut the cylinder---and factors created 
in this cutting procedure.
Now, if a basis cup foam 
is sitting underneath the leftmost picture, 
then the multiplication result is 
topologically not a basis cup foam. 
Thus, we need to turn it into 
a basis cup foam.
In order to do so, we apply \eqref{eq:neckcut} 
to the cylinder 
illustrated above. Here we have to use \eqref{eq:squeezing} 
first, which gives an overall 
factor $\parsign$. (We squeeze the left part of the cylinder.) 
Cutting the cylinder gives a sum of two foams, one with a dot on the top 
and one with a dot on the bottom. The one with a dot on the bottom 
will be of importance and it comes with a factor $\parsign\paro$, 
as illustrated in \eqref{eq:proofs-part-2}.
After neck cutting the 
cylinder we create a ``bubble'', recalling that a basis cup foam is sitting underneath, 
with two internal phantom 
facets in the bottom part of the picture. 
By \eqref{eq:neckcutphantom}, we can remove the 
phantom facets (we remove the left 
phantom facet), we pick up a factor $\parsign$ and create an ``honest'' bubble 
instead. Thus, by \eqref{eq:bubble1}, only the term 
in \eqref{eq:neckcut} with the dot on the bottom 
survives. By \eqref{eq:bubble1} 
the remaining bubble evaluates to $\parsign\paro^{-1}$. Hence, we get in total 
a basis cup foam without dots and a factor
$\parsign\cdot\parsign\paro\cdot\parsign\cdot\parsign\paro^{-1}
=\parsign^4=1$. 
This is the same as for $\Arcalg$, cf. \eqref{eq:mult2}. 
If we start with a dotted basis cup foam, then we can move the dot 
topologically aside and proceed as above. In particular, we 
pick up the same coefficients. 
After the topological rearrangement, we have to move the 
dot to the rightmost facets to produce a basis 
cup foam again. Thus,
using dot migration \eqref{eq:dotmigration}, 
we get the same result as in \eqref{eq:mult2}, 
because the dot moving sign from \eqref{eq:thesigns} 
reflects the dot migration \eqref{eq:dotmigration}. 
More precisely, the dot migration 
gives a factor $\parsign$ 
which is as in \eqref{eq:thesigns}, 
since $\length_\Lambda(\gamma_i^{\rm dot})=1$. 
(We have to move the dot across one cap of length $1$.)  This establishes the situation displayed in \eqref{eq:proofs-part-1} (ii). The same works word-by-word in the 
horizontally flipped cases as well.

\noindent\textbf{Non-nested split - basic shape.} 
The multiplication step is 
\begin{gather}\label{eq:proofs-part-3}
\xy
\xymatrix@C+=1.3cm@L+=6pt{
\xy
(0,0)*{\includegraphics[scale=.65]{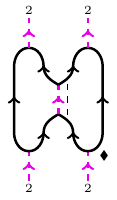}};
\endxy
\ar[rr]|{\;
\xy
(0,0)*{\includegraphics[scale=.65]{figs/fig2-91.pdf}};
\endxy
\;}
& &
\xy
(0,0)*{\includegraphics[scale=.65]{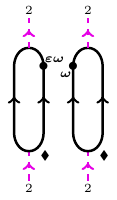}};
\endxy
}
\endxy
\end{gather}
We have again illustrated the dots which are created while 
topologically
rearranging the resulting foam.
Assuming that a basis cup foam without dots is sitting underneath the 
leftmost picture, we see that we almost get 
a basis cup foam after stacking the saddle on top: we get two 
cup foams sitting underneath the left and right circle which 
touch each other in the middle in a closed singular seam, and a 
corresponding phantom facet. Thus, by using the 
singular seam removal
\eqref{eq:closedseam}---creating dots as illustrated above---and 
dot migration \eqref{eq:dotmigration}, we get 
two basis cup foams, one with a dot on the rightmost facets of the left circle 
and one with a dot on the rightmost facet of the right circle. 
The singular seam removal gives a factor $\parsign\paro$ for the first 
and a factor $\paro$ for the second basic cup foam. Additionally, the second gets 
a factor $\parsign$ coming from the dot migration. 
Recalling $\parsign=\pm 1$, this matches the side of $\Arcalg$ which was computed 
in the first line of  \eqref{eq:mult3}.
On the other hand, if a basis cup foam with a dot on the rightmost facet 
is sitting underneath 
the leftmost picture, we can move the dot topologically aside, 
proceed as above and create, after using the 
singular seam removal \eqref{eq:closedseam} and 
dot migration \eqref{eq:dotmigration}, two basis cup foams. 
Remembering that we started with a dot, we see that 
these two are now a basis cup foam with one dot on the rightmost facets 
of the two circles and a foam that is topological a basis 
cup foam, but with two dots on the rightmost facet 
of the right circles. Thus, using \eqref{eq:theusualrelations1b},  we get an extra factor $\alpha$ exactly as 
for 
$\Arcalg$, see \eqref{eq:mult3}.

\noindent\textbf{Nested split - basic shape.} 
The multiplication foam is now 
(indicating again the cylinder we want to cut and the dots we create via cutting)
\begin{gather}\label{eq:proofs-part-4}
\begin{xy}
\xymatrix@C+=1.3cm@L+=6pt{
\xy
(0,0)*{\includegraphics[scale=.65]{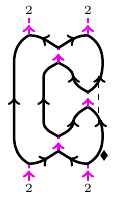}};
\endxy
\ar[rr]|{\;
	\xy
	(0,0)*{\includegraphics[scale=.65]{figs/fig2-91.pdf}};
	\endxy
	\;}& &
\xy
(0,0)*{\includegraphics[scale=.65]{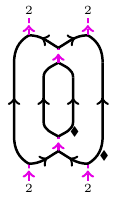}};
\endxy
\ar[rr]|{\;
	\xy
	(0,0)*{\includegraphics[scale=.65]{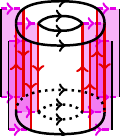}};
	\endxy
	\;}& &
\xy
(0,0)*{\includegraphics[scale=.65]{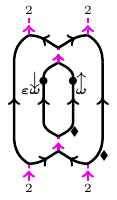}};
\endxy
}
\end{xy}
\end{gather}
We again apply neck cutting; this time to 
the internal cylinder in the second foam between the middle web and the rightmost 
web connecting the two nested circles that we 
can cut using \eqref{eq:neckcut}. 
First assume that the original basis cup foam sitting underneath has no dots.
After neck cutting we get a sum of two basis cup foams, 
so nothing needs to be done topologically. One has a factor $\paro$ 
and a dot sitting on the rightmost facet of the nested circle, 
the other has a factor $\parsign\paro$ and a dot sitting 
on the next to leftmost facets of the 
outer circle, as illustrated above. Moving this dot across 
two phantom facets to the rightmost facets 
picks up, by dot migration \eqref{eq:dotmigration}, a factor 
$\parsign^2=1$ (note that the dot is sitting underneath the place where 
we applied neck cutting and hence, is on a foam 
with a generic slice as in the leftmost picture above). 
We end up with the same as for $\Arcalg$, see \eqref{eq:mult4}.
Similarly, starting with a basis cup foam sitting underneath having a dot on the rightmost 
facet, we can move the dot topologically aside and proceed as before. 
As in the non-nested split case, an additional factor $\alpha$ appears for one of the two basis cup foams via  \eqref{eq:theusualrelations1b}, consistent with \eqref{eq:mult4}.

The case \ref{enum:mimimalsad} can be proven by 
copying the arguments from \cite[Proof of Theorem 4.18]{EST}. 
In particular, non-interfering foam parts 
can be topologically moved away and do not matter in the rewriting process. 
The only thing that changes is that the 
dot moving signs, the topological sign and the saddle 
sign from \eqref{eq:thesigns} are now powers 
of $\parsign$ instead of powers of $-1$. The remaining case is now as follows.

\noindent\textbf{General shape, but minimal saddle.}
The dot moving signs are precisely the same on both sides 
(recalling that moving across phantom facets always gives $\parsign$). 
Furthermore, we can always move existing dots topologically aside 
and we do not have to worry about them until the very end where we 
possible apply \eqref{eq:theusualrelations1} to remove two of them. 
In particular, if we understand the undotted case, then the dotted follows.
So let us consider only basis cup foams without dots.
In case of the non-nested merge, the resulting foams 
are topologically basis cup foams and we are done. 
In case of the nested merge we have to 
topologically manipulate the result until it is a basis cup foam again. 
This can be done as in \cite[Proof of Theorem 4.18]{EST} 
with the difference that the formula \cite[(43)]{EST} gives
\begin{gather}\label{eq:proofs-part-5}
\parsign^{\neatfrac{1}{4}(\length_\Lambda(C_{\rm in})-2)}\cdot\parsign^{1}
\quad\text{instead of}\quad
(-1)^{\neatfrac{1}{4}(\length_\Lambda(C_{\rm in})-2)}\cdot(-1)^{1}.
\end{gather}
This matches the side of $\Arcalg$. 
For the case of the non-nested split we can proceed as above 
and we get the same factors which matches the case of $\Arcalg$. 
Last, for the nested split we copy the argument 
in \cite[Proof of Theorem 4.18]{EST}, but picking up 
\begin{gather}\label{eq:proofs-part-6}
\parsign^{\neatfrac{1}{4}(\length_\Lambda(C_{\rm in})-2)}
\quad\text{instead of}\quad
(-1)^{\neatfrac{1}{4}(\length_\Lambda(C_{\rm in})-2)}.
\end{gather}
Again, this is as in case $\Arcalg$.

\noindent\textbf{General shape.}
The non-nested merge works as above, since the behavior  is independent of the ``size'' of the saddle.
Incorporating a general saddle in the cases of a nested merge 
is as in \cite[Proof of Theorem 4.18]{EST} 
but with
\begin{gather}\label{eq:proofs-part-7}
\parsign^{s_\Lambda(\gamma)}
\quad\text{instead of}\quad
(-1)^{s_\Lambda(\gamma)}.
\end{gather}
The non-nested split case can be done 
as above for the basic shapes. In the case $s_\Lambda(\gamma)=2$ the two 
cup foams touch each other now locally as
\begin{gather}\label{eq:proofs-part-8}
\xy
(0,0)*{\includegraphics[scale=.65]{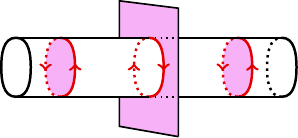}};
\endxy
\xrightarrow[\eqref{eq:neckcutphantom},\eqref{eq:dotmigration}]{\eqref{eq:closedseam}}
\parsign^{2}
\left(
\paro\,
\xy
(0,0)*{\includegraphics[scale=.65]{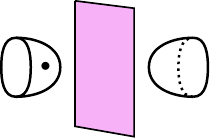}};
\endxy
+\parsign\paro\,
\xy
(0,0)*{\includegraphics[scale=.65]{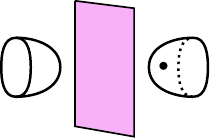}};
\endxy
\right)
\end{gather}
The general case looks similar.
Using \eqref{eq:closedseam} 
once followed by $s_\Lambda(\gamma)-1$ applications
of \eqref{eq:neckcutphantom} 
(as well as $2(s_\Lambda(\gamma)-1)$ applications 
of \eqref{eq:dotmigration} which do not contribute because 
$\parsign=\pm 1$)
gives the above,
where again $s_\Lambda(\gamma)=2$---the left $\parsign$
has an exponent $s_\Lambda(\gamma)$ in general. These scalars are the same as for $\Arcalg$, which we can read off in \fullref{subsec:arc-algebras}.

The case of a nested split does not depend on the saddle and can be done as above 
in case of the minimal saddle.
This finishes the proof of \fullref{theorem:matchalgebrasold}.
\subsection{Proof of \texorpdfstring{\fullref{proposition:crazy-isomorphism}}{Proposition  \ref{proposition:crazy-isomorphism}}}\label{subsec:proof-two}
\subsubsection{The main idea and claim}
First note that the maps from \eqref{eq:the-crazy-map} 
are $\ring$-linear and homogeneous.
It suffices to show the isomorphism 
for some fixed, but arbitrary, $\Lambda\in\bblock$.

The main idea of the 
proof is to show that the maps 
$\coeff_{\stacked}$ successively intertwine the 
two multiplication rules for $\ArcalgS[\KBN]_{\Lambda}$ and 
$\Arcalg_{\Lambda}$. Consequently, we compare two 
intermediate multiplications 
steps in the following fashion:
\begin{gather}\label{eq:proofs-part-9}
\begin{gathered}
\begin{xy}
\xymatrix{
\stacked_{l} 
\ar[rr]^{\boldsymbol{\rm mult}^{\KBN}_{\stacked_{l},\stacked_{l+1}}} 
\ar[d]_{\coeff_{\stacked_{l}}} & &  \stacked_{l+1} \ar[d]^{\coeff_{\stacked_{l+1}}}  \\
\stacked_{l} 
\ar[rr]_{\boldsymbol{\rm mult}^{\parameterS}_{\stacked_{l},\stacked_{l+1}}}   
&     &   \stacked_{l+1},  
}
\end{xy}
\end{gathered}
\end{gather}
with the notation as for the arc algebra multiplication.
The goal is to show that each such diagram, i.e. 
for each $\stacked_{l}$ and $\stacked_{l+1}$, commutes. Since the 
multiplication in $\ArcalgS[\KBN]_{\Lambda}$ has always 
trivial coefficients---up to a factor $\parto$---and
\begin{gather}\label{eq:proofs-part-10}
\boldsymbol{\rm mult}^{\parameterS}_{\stacked_{l},\stacked_{l+1}}(\stacked_l^{\rm or})
=\coeff(\parameterS)\cdot\stackedd_{l+1}^{\rm or}+\text{lin. comb. of other diagrams},
\end{gather}
where $\coeff(\parameterS)$ is the coefficient coming from 
$\Arcalg_{\Lambda}$, this amounts to prove
\begin{gather}\label{eq:show-me-please}
\coeff_{\stacked_{l}}(\stacked_l^{\rm or})\cdot\coeff(\parameterS)
=\coeff_{\stacked_{l+1}}(\stackedd_{l+1}^{\rm or})
\end{gather}
holds up to a factor $\parto$ which always appears on neither side or 
on both sides of \eqref{eq:show-me-please}.


\begin{example}\label{example:to-test}
In \fullref{example:needed!} we calculated 
$\coeff(C_{\mathrm{out}})=\paro^{-1}$, 
$\coeff(C_{\mathrm{in}})=\paro^{-1}$  
and $\coeff(C)=\parsign\paro^{-2}$ for the three circles 
appearing in the diagram on the right-hand side of \eqref{eq:mult2}. 
Moreover, $\coeff(\parameterS)=\parsign$. Thus, \eqref{eq:show-me-please} 
holds.
\end{example}
\subsubsection{Combinatorial preparations for the proof of Claim \eqref{eq:show-me-please}}
We fix $\Lambda\in\bblock$.

\begin{lemma}\label{lemma:compareexteriorinterior}
For a circle $C\in\stacked$ it holds that
\begin{gather}\label{eq:both-sides}
\#\left(\icup(C)\right) + 1 = \#\left(\ecap(C)\right) \quad\text{and}\quad
\#\left(\icap(C)\right) + 1 = \#\left(\ecup(C)\right).
\end{gather}
\end{lemma}

\begin{proof}
This is clear for a circle 
containing only a single cup and cap. 
Any other circle can be constructed from 
such a small circle by successively adding ``zigzags'':
\begin{gather}\label{eq:dull-lemma-1}
\xy
(0,0)*{\includegraphics[scale=.65]{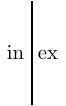}};
\endxy
\rightsquigarrow
\raisebox{0.075cm}{\reflectbox{
\xy
(0,0)*{\includegraphics[scale=.65]{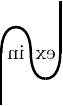}};
\endxy
}}
\;,\;
\xy
(0,0)*{\includegraphics[scale=.65]{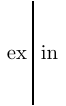}};
\endxy
\rightsquigarrow
\xy
(0,0)*{\includegraphics[scale=.65]{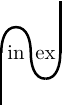}};
\endxy
\;,\;
\xy
(0,0)*{\includegraphics[scale=.65]{figs/fig4-33.pdf}};
\endxy
\rightsquigarrow
\xy
(0,0)*{\includegraphics[scale=.65]{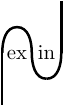}};
\endxy
\;,\;
\xy
(0,0)*{\includegraphics[scale=.65]{figs/fig4-35.pdf}};
\endxy
\rightsquigarrow
\raisebox{0.075cm}{\reflectbox{
\xy
(0,0)*{\includegraphics[scale=.65]{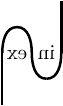}};
\endxy
}}
\end{gather} 
This increases both sides of the equalities from \eqref{eq:both-sides} 
by $0$ or $1$. The claim follows.
\end{proof}

\begin{lemma}\label{lemma:posandtag}
Let $C$ be any circle in a stacked diagram $\stacked$.
\begin{enumerate}[label=(\alph*)]
\item If $\gamma \in \icup(C)$, then 
$\pos_{\Lambda}(\gamma) \equiv \op{t}(C) \,{\rm mod}\, 2$.
\item If $\gamma \in \icap(C)$, then 
$\pos_{\Lambda}(\gamma) \equiv \op{t}(C) \,{\rm mod}\, 2$.
\item If $\gamma \in \ecup(C)$, then 
$\pos_{\Lambda}(\gamma) \equiv \op{t}(C) + 1\,{\rm mod}\, 2$.
\item If $\gamma \in \ecap(C)$, then 
$\pos_{\Lambda}(\gamma) \equiv \op{t}(C) + 1 \,{\rm mod}\, 2$.
\end{enumerate}
\end{lemma}

\begin{proof}
All four statements are clear for a 
circle $C^{\prime}$ with a single cup and cap. 
The circle $C$ is obtained by adding
successively ``zigzags'' to $C^{\prime}$. Adding such a 
zigzag somewhere gives the 
following
(we have illustrated where to read off 
$\pos_{\Lambda}(\gamma)$ and $\op{t}$)
\begin{gather}\label{eq:dull-lemma-2}
(a)\colon
\raisebox{0.075cm}{\reflectbox{
\xy
(0,0)*{\includegraphics[scale=.65]{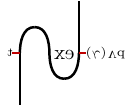}};
\endxy}
}
\quad,\quad
(b)\colon
\xy
(0,0)*{\includegraphics[scale=.65]{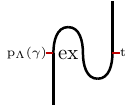}};
\endxy
\quad,\quad
(c)\colon
\xy
(0,0)*{\includegraphics[scale=.65]{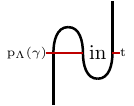}};
\endxy
\quad,\quad
(d)\colon
\raisebox{0.075cm}{\reflectbox{
\xy
(0,0)*{\includegraphics[scale=.65]{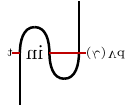}};
\endxy
}}
\end{gather}
Observe that $\op{t}$ might not be the 
rightmost point $\op{t}(C)$ on the circle $C$.
But since clearly $\op{t}\equiv \op{t}(C) \,\, {\rm mod} \,\, 2$, 
these do not change the parity and we are done.
\end{proof}

\begin{lemma}\label{lemma:comparesaddles}
We have
\begin{gather*}\label{eq:dull-lemma-3}
-
{\textstyle\sum_{\gamma \in \icups(C)}}\;
s_{\Lambda}(\gamma)
+
{\textstyle\sum_{\gamma \in \icaps(C)}}\;
(s_{\Lambda}(\gamma)-1)
=1-
{\textstyle\sum_{\gamma \in \ecups(C)}}\;
s_{\Lambda}(\gamma) +
{\textstyle\sum_{\gamma \in \ecaps(C)}}\;
(s_{\Lambda}(\gamma)-1)
\end{gather*}
for any circle $C$ in a stacked diagram $\stacked$.
\end{lemma}

\begin{proof}
By comparison of definitions, we get
\begin{gather}\label{eq:dull-lemma-4}
\begin{gathered}
{\textstyle\sum_{\gamma \in \icups(C)}}\;
s_{\Lambda}(\gamma) + 
{\textstyle\sum_{\gamma \in \ecaps(C)}}\; 
(s_{\Lambda}(\gamma) -1) = \neatfrac{1}{4}\left(\length_\Lambda(C)-2\right),
\\
{\textstyle\sum_{\gamma \in \icaps(C)}}\;
s_{\Lambda}(\gamma) + 
{\textstyle\sum_{\gamma \in \ecups(C)}}\; 
(s_{\Lambda}(\gamma) -1) = 
\neatfrac{1}{4}\left(\length_\Lambda(C)-2\right).
\end{gathered}
\end{gather}
Now we apply \fullref{lemma:compareexteriorinterior}.
\end{proof}

\begin{lemma}\label{lemma:comparecircleinversion}
Let $C\in\stacked$. Then
\begin{gather*}\label{eq:dull-lemma-5}
\begin{gathered}
{\textstyle\sum_{\gamma \in \icups(C_{\rm in})}}\; 
(s_{\Lambda}(\gamma)+1)\pos_{\Lambda}(\gamma) + 
{\textstyle\sum_{\gamma \in \icaps(C_{\rm in})}}\; 
s_{\Lambda}(\gamma)(\pos_{\Lambda}(\gamma)+1) + \op{t}(C) 
+ \neatfrac{1}{4}\left(\length_\Lambda(C)-2\right)
\\
\equiv 
{\textstyle\sum_{\gamma \in \ecups(C_{\rm in})}}\; 
(s_{\Lambda}(\gamma)+1)\pos_{\Lambda}(\gamma) +
{\textstyle\sum_{\gamma \in \ecaps(C_{\rm in})}}\;
s_{\Lambda}(\gamma)(\pos_{\Lambda}(\gamma)+1) \,\, {\rm mod} \,\, 2.
\end{gathered}
\end{gather*}
\end{lemma}

\begin{proof}
Via a direct calculation: 
one starts with the first line and rewrites all 
$\pos_{\Lambda}(\gamma)$ in terms of $\op{t}(C)$ using 
\fullref{lemma:posandtag}. Then we use the same equalities 
as in the proof of \fullref{lemma:comparesaddles}. 
Finally one has to use \fullref{lemma:compareexteriorinterior} 
to arrive at the second line.
\end{proof}
\subsubsection{Proof of Claim \eqref{eq:show-me-please} and thus of \fullref{proposition:crazy-isomorphism}}
We need to 
check the same 
cases \ref{enum:nnm}-\ref{enum:ns} as 
in the proof of \fullref{theorem:matchalgebrasold} and additionally need to distinguish the cases 
with different orientations of the circles in question. 
All components 
not involved in the surgery from $\stacked_{l}$ 
to $\stacked_{l+1}$ remain unchanged,
and we ignore them in the following.

\noindent\textbf{Non-nested merge.} 
Assume that circles $C_i$ and $C_j$ are merged into a circle $C$.
In this case we have (as one easily verifies)
\begin{gather}\label{eq:first-cupcapthingy}
\icup(C_i)\,{\scriptstyle\cup}\,\icap(C_i)\,{\scriptstyle\cup}\,\icup(C_j)\,{\scriptstyle\cup}\,\icap(C_j)=\icup(C)\,{\scriptstyle\cup}\,\icap(C).
\end{gather}
For an example see \eqref{eq:mult1}.
Now let us look at possible orientations.

\textit{Both, $C_i$ and $C_j$, are oriented anticlockwise.} By \eqref{eq:first-cupcapthingy}, 
we directly obtain
\begin{gather}\label{eq:nnmergeanti}
\coeff(C_i^{\rm anti}) \cdot \coeff(C_j^{\rm anti}) = \coeff(C^{\rm anti}).
\end{gather}
Since $\coeff(\parameterS)=1$ 
in this case, we see that \eqref{eq:show-me-please} holds.

\textit{One circle is oriented anticlockwise, the other clockwise.} If without loos of generality 
$C_i$ is oriented clockwise, then the left-hand 
side of \eqref{eq:nnmergeanti} picks up the coefficient 
$\parsign^{\length_\Lambda(\gamma_i^{\rm dot})}=\parsign^{\op{t}(C)-\op{t}(C_i)}$
from the multiplication rule 
for $\boldsymbol{\rm mult}^{\parameterS}_{\stacked_{l},\stacked_{l+1}}$. 
We again obtain \eqref{eq:show-me-please}, since
\begin{gather*}
\begin{gathered}
\coeff(C_i^{\rm cl}) \cdot \coeff(C_j^{\rm anti}) 
\cdot \parsign^{\op{t}(C)-\op{t}(C_i)}
\!\overset{\eqref{eq:from-anti-to-clock}}{=}\! \coeff(C_i^{\rm anti}) \cdot \parsign^{\op{t}(C_i)} 
\cdot \coeff(C_j^{\rm anti}) \cdot \parsign^{\op{t}(C)-\op{t}(C_i)}
\\
\overset{\eqref{eq:nnmergeanti}}{=}\coeff(C^{\rm anti}) \cdot \parsign^{\op{t}(C)}
\;\overset{\eqref{eq:from-anti-to-clock}}{=} \coeff(C^{\rm cl}).
\end{gathered}
\end{gather*}

\textit{Both, $C_i$ and $C_j$, are oriented clockwise.} In this case we have 
\begin{gather}\label{eq:proofs-part-12}
\begin{aligned}
& \coeff(C_i^{\rm cl}) \cdot \coeff(C_j^{\rm cl}) \cdot \parto \cdot \parsign^{\op{t}(C)-\op{t}(C_i)} \cdot \parsign^{\op{t}(C)-\op{t}(C_j)}\\
\overset{\eqref{eq:from-anti-to-clock}}{=}\;& \coeff(C_i^{\rm anti}) \cdot \coeff(C_j^{\rm anti}) \cdot \parto 
\;\overset{\eqref{eq:nnmergeanti}}{=}\; \coeff(C^{\rm anti}) \cdot \parto,
\end{aligned}
\end{gather}
which again give \eqref{eq:show-me-please}, because 
the multiplication rule 
for $\boldsymbol{\rm mult}^{\parameterS}_{\stacked_{l},\stacked_{l+1}}$ picks 
up the coefficient 
$\coeff(\parameterS)=
\parto\parsign^{\length_\Lambda(\gamma_i^{\rm dot})}\parsign^{\length_\Lambda(\gamma_j^{\rm dot})}
=\parto\parsign^{\op{t}(C)-\op{t}(C_i)}\parsign^{\op{t}(C)-\op{t}(C_j)}$.

\noindent\textbf{Nested merge.}
In this case two nested circles, $C_{\rm out}$ and 
$C_{\rm in}$, are merged into one circle $C$. In the nested 
situation---also for the nested split below---the 
notion of exterior and interior swaps for the nested circle $C_{\rm in}$. 
Moreover, in case of the nested merge, the cup-cap pair involved in the surgery 
is of the form $\icup$\,-\,$\ecap$ or 
of the form $\ecup$\,-\hspace*{0.06cm}$\icap$ and hence, is ``lost'' 
after the surgery.
That is, we have altogether
\begin{gather}\label{eq:second-cupcapthingy}
\left(\icup(C_{\mathrm{out}})\,{\scriptstyle\cup}\,\icap(C_{\mathrm{out}})
\,{\scriptstyle\cup}\,\ecup(C_{\mathrm{in}})\,{\scriptstyle\cup}\,\ecap(C_{\mathrm{in}})\right)\setminus\text{surg}
=\icup(C)\,{\scriptstyle\cup}\,\icap(C).
\end{gather}
Here ``surg'' is the set containing the cup-cap of the surgery, see e.g. \eqref{eq:mult2}.

\textit{Both, $C_{\rm out}$ and $C_{\rm in}$, are oriented anticlockwise.} 
Now we get the coefficient 
$\coeff(\parameterS)
=\parsign \cdot \parsign^{\neatfrac{1}{4}\left(\length_\Lambda(C_{\mathrm{in}})-2\right)}
\cdot \parsign^{s_{\Lambda}(\gamma)}$ from 
$\boldsymbol{\rm mult}^{\parameterS}_{\stacked_{l},\stacked_{l+1}}$,
and then \eqref{eq:show-me-please} follows from the calculation
\begin{gather}\label{eq:the-above}
\begin{aligned}
&\coeff(C^{\rm anti}_{\rm out}) \cdot \coeff(C^{\rm anti}_{\rm in}) \cdot \parsign  \cdot \parsign^{\neatfrac{1}{4}\left(\length_\Lambda(C_{\mathrm{in}})-2\right)} 
\cdot \parsign^{s_{\Lambda}(\gamma)}
\\
=\;\;& \coeff(C^{\rm anti}_{\rm out})  
\cdot 
{\textstyle\prod_{\gamma^{\prime} \in \icups(C_{\rm in})}}\; 
\parsign^{(s_{\Lambda}(\gamma^{\prime})+1)\pos_\Lambda(\gamma^{\prime})}
\cdot 
{\textstyle\prod_{\gamma^{\prime} \in \icaps(C_{\rm in})}}\; 
\parsign^{s_{\Lambda}(\gamma^{\prime})(\pos_\Lambda(\gamma^{\prime})+1)}
\\
&  \cdot 
{\textstyle\prod_{\gamma^{\prime} \in \icups(C_{\rm in})}}\;
\paro^{-s_{\Lambda}(\gamma^{\prime})} 
\cdot 
{\textstyle\prod_{\gamma^{\prime} \in \icaps(C_{\rm in})}}\; 
\paro^{s_{\Lambda}(\gamma^{\prime})-1} \cdot \parsign \cdot \parsign^{\neatfrac{1}{4}\left(\length_\Lambda(C_{\mathrm{in}})-2\right)} \cdot \parsign^{s_{\Lambda}(\gamma)}
\\
\overset{(\mathrm{I})}{=}\;\;& 
\coeff_{\parsign}(C^{\rm anti}_{\rm out}) 
\cdot 
{\textstyle\prod_{\gamma^{\prime} \in \ecups(C_{\rm in})}}\;
\parsign^{(s_{\Lambda}(\gamma^{\prime})+1)\pos_{\Lambda}(\gamma^{\prime})} 
\cdot 
{\textstyle\prod_{\gamma^{\prime} \in \ecaps(C_{\rm in})}}\; 
\parsign^{s(\gamma^{\prime})(\pos_{\Lambda}(\gamma^{\prime})+1)}
\\
&\!\!\!\!\!\cdot \;\coeff_{\paro}(C^{\rm anti}_{\rm out}) 
\cdot 
{\textstyle\prod_{\gamma^{\prime} \in \icups(C_{\rm in})}}\; 
\paro^{-s_{\Lambda}(\gamma^{\prime})} 
\cdot 
{\textstyle\prod_{\gamma^{\prime} \in \icaps(C_{\rm in})}}\; 
\paro^{s_{\Lambda}(\gamma^{\prime})-1}
\cdot \parsign^{\pos_{\Lambda}(\gamma)+s_{\Lambda}(\gamma)}
\\
\overset{(\mathrm{II})}{=}\;\;& 
\coeff_{\parsign}(C^{\rm anti}_{\rm out}) 
\cdot 
{\textstyle\prod_{\gamma^{\prime} \in \ecups(C_{\rm in})}}\; 
\parsign^{(s_{\Lambda}(\gamma^{\prime})+1)\pos_{\Lambda}(\gamma^{\prime})} \cdot 
{\textstyle\prod_{\gamma^{\prime} \in \ecaps(C_{\rm in})}}\; 
\parsign^{s_{\Lambda}(\gamma^{\prime})(\pos_{\Lambda}(\gamma^{\prime})+1)}
\\
&\!\!\!\!\cdot \;\coeff_{\paro}(C^{\rm anti}_{\rm out}) 
\cdot 
{\textstyle\prod_{\gamma^{\prime} \in \ecups(C_{\rm in})}}\; 
\paro^{-s_{\Lambda}(\gamma^{\prime})} \cdot 
{\textstyle\prod_{\gamma^{\prime} \in \ecaps(C_{\rm in})}}\; 
\paro^{s_{\Lambda}(\gamma^{\prime})-1} \cdot \parsign^{\pos_{\Lambda}(\gamma)+s_{\Lambda}(\gamma)}\cdot \paro
\\
\overset{\eqref{eq:second-cupcapthingy}}{=}\;\;&
\coeff_{\parsign}(C^{\rm anti})\cdot\coeff_{\paro}(C^{\rm anti})=\coeff(C^{\rm anti}).
\end{aligned}
\end{gather} \makeautorefname{lemma}{Lemmas}
Here (I) follows from \fullref{lemma:posandtag} 
and \ref{lemma:comparecircleinversion} 
(since $\parsign=\pm 1$),\makeautorefname{lemma}{Lemma} 
and (II) from \fullref{lemma:comparesaddles}. 
Moreover, note that 
$\parsign^{\pos_{\Lambda}(\gamma)+s_{\Lambda}(\gamma)}\paro$ 
is the inverse of the coefficient coming 
from the cup-cap pair in the surgery (counting them both).

\textit{$C_{\rm out}$ is oriented clockwise and $C_{\rm in}$ anticlockwise.} 
In this case both sides are just 
multiplied with $\parsign^{\op{t}(C)}=\parsign^{\op{t}(C_{\rm out})}$. Hence, 
the calculation from \eqref{eq:the-above} gives
\begin{gather}\label{eq:proofs-part-13}
\coeff(C^{\rm cl}_{\rm out}) \cdot \coeff(C^{\rm anti}_{\rm in}) \cdot \parsign  \cdot \parsign^{\neatfrac{1}{4}\left(\length_\Lambda(C_{\mathrm{in}})-2\right)} \cdot \parsign^{s_{\Lambda}(\gamma)} =  \coeff(C^{\rm cl}).
\end{gather}
Thus, we again obtain \eqref{eq:show-me-please}, since 
$\boldsymbol{\rm mult}^{\parameterS}_{\stacked_{l},\stacked_{l+1}}$ 
does not give extra factors additionally to the ones from above.

\textit{$C_{\rm in}$ is oriented clockwise and $C_{\rm out}$ anticlockwise.} In 
this case the coefficient of $C$ is multiplied with 
$\parsign^{\op{t}(C)}$, while the one for $C_{\rm in}$ 
is multiplied with $\parsign^{\op{t}(C_{\rm in})}$. But in addition 
the multiplication also introduces a dot moving. Hence, by \eqref{eq:the-above},
\begin{gather}\label{eq:nestedmerge}
\coeff(C^{\rm anti}_{\rm out}) \cdot \coeff(C^{\rm cl}_{\rm in}) \cdot \parsign \cdot \parsign^{\op{t}(C)-\op{t}(C_{\rm in})} \cdot \parsign^{\neatfrac{1}{4}\left(\length_\Lambda(C_{\mathrm{in}})-2\right)} \cdot \parsign^{s_{\Lambda}(\gamma)} =  \coeff(C^{\rm cl}),
\end{gather}
which again gives \eqref{eq:show-me-please}, since 
$\boldsymbol{\rm mult}^{\parameterS}_{\stacked_{l},\stacked_{l+1}}$ 
gives, additionally to the factors from above, the extra coefficient 
$\parsign^{\length_\Lambda(\gamma^{\mathrm{dot}}_{\mathrm{in}})}
=\parsign^{\op{t}(C)-\op{t}(C_{\rm in})}$.

\textit{Both, $C_{\rm in}$ and $C_{\rm out}$, are oriented clockwise.} 
In this case we obtain two dot moving signs, but 
the one for $C_{\rm out}$ is, as before, equal to $1$. 
Thus, we obtain the same as in \eqref{eq:nestedmerge}, but multiplied 
on both sides as in \eqref{eq:second-cupcapthingy} with 
$\parto \cdot \parsign^{\op{t}(C)}$ 
which shows \eqref{eq:show-me-please}.

\noindent\textbf{Non-nested split.}
In this case a circle $C$ is split into two non-nested 
circles $C_i$ 
and $C_j$ (containing the vertexes at positions $i$ or $j$). 
We clearly have 
\begin{gather}\label{eq:third-cupcapthingy}
\icup(C)\,{\scriptstyle\cup}\,\icap(C)
=
\icup(C_{i})\,{\scriptstyle\cup}\,\icap(C_{i})
\,{\scriptstyle\cup}\,\icup(C_{j})\,{\scriptstyle\cup}\,\icap(C_{j})\,{\scriptstyle\cup}\,\text{surg},
\end{gather}
where ``surg'' is as above.
For an example see \eqref{eq:mult3}.

\textit{$C$ is oriented anticlockwise.} 
By \eqref{eq:third-cupcapthingy}, we get
\begin{gather}\label{eq:some-useful-guy}
\coeff(C^{\rm anti})
=
\coeff(C_i^{\rm anti}) \cdot \coeff(C_j^{\rm anti}) 
\cdot \parsign^{\pos_{\Lambda}(\gamma)+s_{\Lambda}(\gamma)} \cdot \paro^{-1},
\end{gather}
since, as above, $\parsign^{\pos_{\Lambda}(\gamma)+s_{\Lambda}(\gamma)}\paro^{-1}$ 
is the coefficient coming from the cup-cap pair in the surgery 
(recalling that $\parsign=\pm 1$). 
Now, we have 
$\coeff(\parameterS)=\paro
\parsign^{\length_{\Lambda}(\gamma_i^{\rm ndot})}
\parsign^{s_{\Lambda}(\gamma)}$. 
By \fullref{lemma:posandtag} and $\parsign=\pm 1$ we have 
$\parsign^{\length_{\Lambda}(\gamma_i^{\rm ndot})}=
\parsign^{\op{t}(C_i)-\pos_{\Lambda}(\gamma)}$.
This in turn gives
\begin{gather}\label{eq:proofs-part-14}
\begin{aligned}
& \coeff(C^{\rm anti})\cdot\paro\cdot\parsign^{\length_{\Lambda}(\gamma_i^{\rm ndot})}\cdot\parsign^{s_{\Lambda}(\gamma)}\\
\overset{\eqref{eq:some-useful-guy}}{=}\;\;& 
\coeff(C_i^{\rm anti})\cdot\coeff(C_j^{\rm anti})\cdot\paro\cdot\parsign^{\op{t}(C_i)-\pos_{\Lambda}(\gamma)}\cdot\parsign^{s_{\Lambda}(\gamma)}
\cdot \parsign^{\pos_{\Lambda}(\gamma)+s_{\Lambda}(\gamma)}
\cdot \paro^{-1}
\\
=\;\;\;& \coeff(C_i^{\rm anti})\cdot\parsign^{\op{t}(C_i)} \cdot \coeff(C_j^{\rm anti})
\overset{\eqref{eq:from-anti-to-clock}}{=} \coeff(C_i^{\rm cl}) \cdot \coeff(C_j^{\rm anti}).
\end{aligned}
\end{gather}
The term with $C_j$ is oriented clockwise instead is dealt 
with completely analogously using the fact that 
$\parsign^{\pos_{\Lambda}(j)} = 
\parsign^{\pos_{\Lambda}(\gamma)+1}$ (by definition). 
We obtain \eqref{eq:show-me-please}.

\textit{$C$ is oriented clockwise.} We 
first compare the coefficients for the term where 
both, $C_i$ and $C_j$, are oriented clockwise 
(thus,  
$\coeff(\parameterS)=\paro
\parsign^{\length_{\Lambda}(\gamma_j^{\rm dot})}
\parsign^{\length_{\Lambda}(\gamma_i^{\rm ndot})}
\parsign^{s_{\Lambda}(\gamma)}$) and obtain 
by rewriting the dot moving signs similar as above (using $\parsign=\pm 1$)
\begin{gather}\label{eq:proofs-part-15}
\begin{gathered}
\coeff(C^{\rm cl}) \cdot \paro \cdot \parsign^{\length_{\Lambda}(\gamma_j^{\rm dot})} \cdot \parsign^{\length_{\Lambda}(\gamma_i^{\rm ndot})}\cdot \parsign^{s_{\Lambda}(\gamma)}
\\
=\; \coeff(C^{\rm cl}) \cdot \paro \cdot 
\parsign^{\op{t}(C)-\op{t}(C_j)}\cdot
\parsign^{\op{t}(C_i)-\pos_{\Lambda}(\gamma)} \cdot \parsign^{s_{\Lambda}(\gamma)}
\\
\overset{\eqref{eq:from-anti-to-clock}}{=}\; \coeff(C^{\rm anti}) \cdot \paro \cdot \parsign^{\op{t}(C_j)} \cdot \parsign^{\op{t}(C_i)-\pos_{\Lambda}(\gamma)} \cdot \parsign^{s_{\Lambda}(\gamma)}
\\
\overset{\eqref{eq:some-useful-guy}}{=}\; \coeff(C_i^{\rm anti})   \cdot \parsign^{\op{t}(C_i)} \cdot \coeff(C_j^{\rm anti}) \cdot \parsign^{\op{t}(C_j)}
\;\overset{\eqref{eq:from-anti-to-clock}}{=}\; \coeff(C_i^{\rm cl}) \cdot \coeff(C_j^{\rm cl}).
\end{gathered}
\end{gather}
Hence, we have \eqref{eq:show-me-please}.
For the term where both $C_i$ and $C_j$ are oriented anticlockwise 
(where we have 
$\coeff(\parameterS)=\parto\parsign\paro
\parsign^{\length_{\Lambda}(\gamma_j^{\rm dot})}
\parsign^{\length_{\Lambda}(\gamma_j^{\rm ndot})}
\parsign^{s_{\Lambda}(\gamma)}$) we obtain
\begin{gather}\label{eq:proofs-part-16}
\begin{gathered}
\coeff(C^{\rm cl}) \cdot \parto \cdot \parsign \cdot \paro \cdot \parsign^{\length_{\Lambda}(\gamma_j^{\rm dot})} \cdot \parsign^{\length_{\Lambda}(\gamma_j^{\rm ndot})}\cdot \parsign^{s_{\Lambda}(\gamma)}
\\
=\; \coeff(C^{\rm cl}) \cdot \parto \cdot \parsign \cdot \paro \cdot \parsign^{\op{t}(C_j)-\pos_{\Lambda}(j)} \cdot \parsign^{\op{t}(C)-\op{t}(C_j)}\cdot \parsign^{s_{\Lambda}(\gamma)}
\\
\overset{\eqref{eq:from-anti-to-clock}}{=}\; \coeff(C^{\rm anti}) \cdot \parto \cdot \parsign \cdot \paro \cdot \parsign^{\pos_{\Lambda}(j)} \cdot \parsign^{s(\gamma)}
\;=\; \coeff(C^{\rm anti}) \cdot \parto \cdot \parsign^{\pos_{\Lambda}(\gamma)+s_{\Lambda}(\gamma)}\cdot \paro
\\
\overset{\eqref{eq:some-useful-guy}}{=}\; \coeff(C_i^{\rm anti}) \cdot \coeff(C_j^{\rm anti})  \cdot \parto,
\end{gathered}
\end{gather}
where we again use the crucial fact that $\parsign=\pm 1$.
Thus, we obtain \eqref{eq:show-me-please}.

\noindent\textbf{Nested split.}
In this case one circle $C$ is split into 
two nested circles $C_{\rm out}$ and $C_{\rm in}$. The steps are very similar 
to the case of the nested merge before, with the main 
difference that, instead of \eqref{eq:second-cupcapthingy}, 
we have
\begin{gather}\label{eq:last-cupcapthingy}
\icup(C)\,{\scriptstyle\cup}\,\icap(C)=
\icup(C_{\mathrm{out}})\,{\scriptstyle\cup}\,\icap(C_{\mathrm{out}})
\,{\scriptstyle\cup}\,\ecup(C_{\mathrm{in}})\,{\scriptstyle\cup}\,\ecap(C_{\mathrm{in}}).
\end{gather}
For an example see \eqref{eq:mult4}.
By \eqref{eq:last-cupcapthingy}, we obtain 
(with (III) similar as in \eqref{eq:the-above})
\begin{gather}\label{eq:proofs-part-17}
\begin{aligned}
& \coeff(C^{\rm anti})\\
=\;\;\,& \coeff_{\parsign}(C^{\rm anti}_{\rm out}) \cdot 
{\textstyle\prod_{\gamma^{\prime} \in \ecaps(C_{\rm in})}}\; 
\parsign^{s_{\Lambda}(\gamma^{\prime})(\pos_{\Lambda}(\gamma^{\prime})+1)} \cdot 
{\textstyle\prod_{\gamma^{\prime} \in \ecups(C_{\rm in})}}\;
\parsign^{(s_{\Lambda}(\gamma^{\prime})+1)\pos_{\Lambda}(\gamma^{\prime})}
\\ & \!\!\!\!\cdot\;\coeff_{\paro}(C^{\rm anti}_{\rm out}) \cdot 
{\textstyle\prod_{\gamma^{\prime} \in \ecaps(C_{\rm in})}}\;
\paro^{s_{\Lambda}(\gamma^{\prime})-1} \cdot 
{\textstyle\prod_{\gamma^{\prime} \in \ecups(C_{\rm in})}}\;
\paro^{-s_{\Lambda}(\gamma^{\prime})}\\
\overset{(\mathrm{III})}{=}\;&\coeff(C^{\rm anti}_{\rm out}) \cdot \coeff(C^{\rm anti}_{\rm in}) \cdot \paro^{-1} \cdot \parsign^{\neatfrac{1}{4}\left(\length_\Lambda(C_{\mathrm{in}})-2\right)} \cdot \parsign^{\op{t}(C_{\rm in})}.
\end{aligned}
\end{gather}

\textit{$C$ is oriented anticlockwise.}  
We have to multiply the coefficient $\coeff(C^{\rm anti})$ 
by
\begin{gather}\label{eq:proofs-part-18}
\coeff(\parameterS)=\paro\cdot\parsign^{\neatfrac{1}{4}
\left(\length_\Lambda(C_{\mathrm{in}})-2\right)}
\quad\text{respectively}\quad
\coeff(\parameterS)=\parsign\cdot\paro\cdot\parsign^{\neatfrac{1}{4}\left(\length_\Lambda(C_{\mathrm{in}})-2\right)}
\end{gather}
and compare 
it to the coefficient $\coeff(C^{\rm anti}_{\rm out}) 
\coeff(C^{\rm cl}_{\rm in})$ respectively to
$\coeff(C^{\rm cl}_{\rm out})\coeff(C^{\rm anti}_{\rm in})$. 
In both cases \eqref{eq:show-me-please} follows then by 
\fullref{lemma:comparecircleinversion}.

\textit{$C$ is oriented clockwise.} This is done in an 
analogous way. Since we have $\coeff(C^{\rm cl}) = 
\coeff(C^{\rm anti})\parsign^{\op{t}(C_{\rm out})}$, 
this fits for both appearing terms.

\noindent
Altogether we verified Claim \eqref{eq:show-me-please} and thus  \fullref{proposition:crazy-isomorphism} follows.
\subsection{Proof of \texorpdfstring{\fullref{proposition:crazy-module-maps}}{Proposition \ref{proposition:crazy-module-maps}}}\label{subsec:proof-three}
The proof is done in complete analogy to 
the proof of \fullref{proposition:crazy-isomorphism}. 
We show that in each step the coefficient maps defined 
above for stacked diagrams intertwine the multiplication 
steps, i.e. in each step \eqref{eq:show-me-please} 
holds true. Since the coefficient map is only modified slightly, 
it is clear that all arguments for the non-nested merge and 
non-nested split are valid in the exact same way as before. 
For the nested cases the swap of exterior and interior of the 
inner circle $C_{\rm in}$ is more involved and we need some preparatory lemmas which also establish \fullref{remark:new-coeff-intertwines} and \fullref{lemma:signs-are-constant-next}.
\subsubsection{Combinatorial preparations for the proof of \texorpdfstring{\fullref{proposition:crazy-module-maps}}{Proposition \ref{proposition:crazy-module-maps}}}
Fix $\Lambda\in\bblock$.

\begin{lemma}\label{lemma:posandtag-forrays}
Let $C$ be any circle in a stacked diagram $\stacked$.
\begin{enumerate}[label=(\alph*)]
\item If $\gamma \in\eright(C)$, then 
$\pos_{\Lambda}(\gamma) \equiv \op{t}(C) + 1 \,{\rm mod}\, 2$.
\item If $\gamma \in\eleft(C)$, then 
$\pos_{\Lambda}(\gamma) \equiv \op{t}(C)\,{\rm mod}\, 2$.
\item If $\gamma \in\ileft(C)$, then 
$\pos_{\Lambda}(\gamma) \equiv \op{t}(C)+1 \,{\rm mod}\, 2$.
\item If $\gamma \in\iright(C)$, then 
$\pos_{\Lambda}(\gamma) \equiv \op{t}(C) \,{\rm mod}\, 2$.
\end{enumerate}
\end{lemma}

\begin{proof}
Recall that the symbol $\times$ counts as being of length $2$. 
Hence, moving across parts in (a)-(d) 
preserves the parity. Thus, the claim follows 
as in \fullref{lemma:posandtag}.
\end{proof}

\begin{lemma}\label{lemma:comparesaddles-bimodules}
Let $C$ be any circle in a stacked diagram $\stacked$. Then
\begin{gather}\label{eq:more-dull-1}
\begin{gathered}
-
{\textstyle\sum_{\gamma \in \ecups(C)}}\; 
(s_{\Lambda}(\gamma) -1)  + 
{\textstyle\sum_{\gamma \in \ecaps(C)}}\; 
(s_{\Lambda}(\gamma) -1) + 
{\textstyle\sum_{\gamma \in\ilefts(C)}}\; 
1 + 
{\textstyle\sum_{\gamma \in\irights(C)}}\; 
1
\\
=-
{\textstyle\sum_{\gamma \in \icups(C)}}\;
s_{\Lambda}(\gamma) + 
{\textstyle\sum_{\gamma \in \icaps(C)}}\;
s_{\Lambda}(\gamma) + 
{\textstyle\sum_{\gamma \in\elefts(C)}}\; 
1 + 
{\textstyle\sum_{\gamma \in\erights(C)}}\; 
1,
\end{gathered}
\end{gather}
and both equal $\neatfrac{1}{4}\left(\length_\Lambda(C)-2\right)$.
\end{lemma}

\begin{proof}
This follows immediately by interpreting 
$\neatfrac{1}{4}\left(\length_\Lambda(C)-2\right)$ 
as the number of internal phantom edges of 
the circle as done in \cite[Lemma 4.10]{EST}.
\end{proof}

\begin{lemma}\label{lemma:signs-are-constant}
For a stacked diagram $\stacked^{\mathrm{or}}$ 
and a circle $C$ in it, we have
\begin{gather}\label{eq:more-dull-2}
\coeffb_{\parsign}(C,\stacked^{\mathrm{or}}) = 
\coeff_{\parsign}(C,\stacked^{\mathrm{or}}) \cdot \factor_{\parsign}(C),\;
\coeffb_{\paro}(C,\stacked^{\mathrm{or}}) = 
\coeff_{\paro}(C,\stacked^{\mathrm{or}}) \cdot \factor_{\paro}(C),\\
\begin{aligned}
\factor_{\parsign}(C) =\;\;& 
{\textstyle\prod_{\gamma \in \icups(C) \,\cup\, \ecups(C)}}\;
\parsign^{\pos_\Lambda(\gamma) + s_\Lambda(\gamma)} \\
=\;\;& 
{\textstyle\prod_{\gamma \in \icaps(C) \,\cup\, \ecaps(C)}}\; 
\parsign^{\pos_\Lambda(\gamma) + s_\Lambda(\gamma)} \cdot \parsign^{ \#\left( \eright(C) \,\cup\, \eleft(C) \,\cup\, \iright(C) \,\cup\, \ileft(C)\right)},\\
\factor_{\paro}(C) =\;\;& \paro^{ \#\left( \icup(C) \,\cup\, \ecup(C)\right)} = \paro^{ \#\left( \icap(C) \,\cup\, \ecap(C)\right)}.
\end{aligned}
\end{gather}
\end{lemma}

\begin{proof}\makeautorefname{lemma}{Lemmas}
Consider $\coeffb_{\parsign}(C,\stacked^{\mathrm{or}})$. After rewriting 
all positions with respect to the rightmost point using \fullref{lemma:posandtag} and \ref{lemma:posandtag-forrays}, \makeautorefname{lemma}{Lemma} we obtain with \fullref{lemma:comparesaddles-bimodules} for the left had side the following
\begin{gather}\label{eq:more-dull-3}
\begin{gathered}
{\textstyle\sum_{\gamma \in \ecups(C)}}\; 
s_\Lambda(\gamma){\rm t}(C) + 
{\textstyle\sum_{\gamma \in \ecaps(C)}}\; 
(s_\Lambda(\gamma)-1)({\rm t}(C)+1) 
\\
+ 
{\textstyle\sum_{\gamma \in \irights(C)}}\;
({\rm t}(C)+1) + 
{\textstyle\sum_{\gamma \in \ilefts(C)}}\;
({\rm t}(C)+1)
\\
\equiv\left( 
{\textstyle\sum_{\gamma \in \ecups(C)}}\; 
(s_\Lambda(\gamma)-1) +
{\textstyle\sum_{\gamma \in \ecaps(C)}}\; 
(s_\Lambda(\gamma)-1) + \# 
\left( \iright(C) \,{\scriptstyle \cup}\, \ileft(C)\right) \right) {\rm t}(C) 
\\
+ \#(\ecup(C)) {\rm t}(C) + 
{\textstyle\sum_{\gamma \in \ecaps(C)}}\; 
(s_\Lambda(\gamma)-1) + \# \left( \iright(C) \,{\scriptstyle \cup}\,\ileft(C) \right) 
\\
\equiv\left( 
{\textstyle\sum_{\gamma \in \icups(C)}}\; 
s_\Lambda(\gamma) +
{\textstyle\sum_{\gamma \in \icaps(C)}}\; 
s_\Lambda(\gamma) + \# \left( \eright(C) 
\,{\scriptstyle \cup}\, \eleft(C)\right) \right) {\rm t}(C) 
\\
+ \#(\ecup(C)) {\rm t}(C) + 
{\textstyle\sum_{\gamma \in \ecaps(C)}}\; 
(s_\Lambda(\gamma)-1) + \# \left( \iright(C) \,{\scriptstyle \cup}\, \ileft(C) \right).
\end{gathered}
\end{gather}
with all congruences modulo $2$.
Collecting terms belonging to $\coeff_{\parsign}(C,\stacked^{\mathrm{or}})$ we are left with
\begin{gather}\label{eq:more-dull-4}
\begin{gathered}
\# \left( \eright(C) \,{\scriptstyle \cup}\, \eleft(C) 
\,{\scriptstyle \cup}\, \iright(C) \,{\scriptstyle \cup}\, \ileft(C)\right) 
\\
+ \#(\icap(C)) {\rm t}(C) + 
{\textstyle\sum_{\gamma \in \icaps(C)}}\; 
s_\Lambda(\gamma) + \#(\ecap(C)) {\rm t}(C) + 
{\textstyle\sum_{\gamma \in \ecaps(C)}}\; 
(s_\Lambda(\gamma)-1)
\\
\overset{\fullref{lemma:posandtag}}{\equiv}
\# \left( \eright(C) \,{\scriptstyle \cup}\, \eleft(C) 
\,{\scriptstyle \cup}\, \iright(C) \,{\scriptstyle \cup}\, \ileft(C)\right) 
\\
+ 
{\textstyle\sum_{\gamma \in \icaps(C)}}\; 
(\pos_\Lambda(\gamma) + s_\Lambda(\gamma)) + 
{\textstyle\sum_{\gamma \in \ecaps(C)}}\; 
(\pos_\Lambda(\gamma) + s_\Lambda(\gamma)).
\end{gathered}
\end{gather}
That this can be rewritten with 
respect to cups instead is just an 
application of 
\fullref{lemma:comparesaddles-bimodules} 
to both sums in the first line above.

Next, the (easier) $\paro$-term:
$\coeffb_{\paro}(C,\stacked^{\mathrm{or}})$ can be rewritten as
\begin{gather}\label{eq:more-dull-5}
\begin{gathered}
-
{\textstyle\sum_{\gamma \in \ecups(C)}}\;
(s_\Lambda(\gamma) - 1) + 
{\textstyle\sum_{\gamma \in \ecaps(C)}}\; 
s_\Lambda(\gamma)  + \# \left( \iright(C)\, {\scriptstyle \cup}\, \ileft(C) \right)
\\
\overset{\fullref{lemma:comparesaddles-bimodules}}{=}
-
{\textstyle\sum_{\gamma \in \icups(C)}}\; 
s_\Lambda(\gamma) + 
{\textstyle\sum_{\gamma \in \icaps(C)}}\; 
(s_\Lambda(\gamma) -1) + \#\left( \eright(C)\,{\scriptstyle \cup}\, \eleft(C) \right) 
\\
+\# \left( \icap(C) \,{\scriptstyle \cup}\, \ecap(C)\right).
\end{gathered}
\end{gather}
The first three summands are the 
powers in $\coeff_{\paro}(C,\stacked^{\mathrm{or}})$, while 
the last term is the power in $\chi_{\paro}(C)$. 
That $\chi_{\paro}(C)$ can be written in the two 
ways is evident.
\end{proof}
\subsubsection{Proof of \eqref{eq:show-me-please} and thus of  \fullref{proposition:crazy-module-maps}}
The proof of the proposition is totally analogously to \fullref{proposition:crazy-isomorphism} except of the case of nested merges which we treat now carefully. Assume therefore that two nested circles $C_{\rm out}$ and $C_{\rm in}$ 
are merged into one circle $C$. As in the proof of 
\fullref{proposition:crazy-isomorphism}, 
the notion of exterior and interior swaps for the nested 
circle $C_{\rm in}$. 
Overall the situation is similar in the sense that the 
cup-cap pair involved in the surgery is of the form 
$\icup$\,-\,$\ecap$ or of the form $\ecup$\,-\hspace*{0.06cm}$\icap$. That is, we have
\begin{gather}\label{eq:cupcapthingy-module}
\left(\icup(C_{\mathrm{out}})\,{\scriptstyle\cup}\,\icap(C_{\mathrm{out}})
\,{\scriptstyle\cup}\,\ecup(C_{\mathrm{in}})\,{\scriptstyle\cup}\,\ecap(C_{\mathrm{in}})\right)\setminus\text{surg}
=\icup(C)\,{\scriptstyle\cup}\,\icap(C),
\end{gather}
where ``surg'' is the set containing the cup and cap of the surgery, and we have
\begin{gather}\label{eq:second-cupcapthingy-module}
\eright(C_{\mathrm{out}})\,{\scriptstyle{\scriptstyle\cup}}\,\eleft(C_{\mathrm{out}})
\,{\scriptstyle\cup}\,
\iright(C_{\mathrm{in}})\,{\scriptstyle\cup}\,\ileft(C_{\mathrm{in}})=\eright(C)\,{\scriptstyle\cup}\,\eleft(C).
\end{gather}

\textit{Both, $C_{\rm out}$ and $C_{\rm in}$, are oriented anticlockwise.} 
Similarly as before, we get \eqref{eq:show-me-please}:
\begin{gather}\label{eq:crazy-equation}
\begin{gathered}
\coeff_{\parsign}(C^{\rm anti}_{\rm out}) \cdot \coeff_{\parsign}(C^{\rm anti}_{\rm in}) 
\cdot \parsign  \cdot \parsign^{\neatfrac{1}{4}\left(\length_\Lambda(C_{\mathrm{in}})-2\right)} 
\cdot \parsign^{s_{\Lambda}(\gamma)}
\\
=\coeff_{\parsign}(C^{\rm anti}_{\rm out})  
\cdot 
{\textstyle\prod_{\gamma^{\prime} \in \icups(C_{\rm in})}}\; 
\parsign^{(s_{\Lambda}(\gamma^{\prime})+1)\pos_{\Lambda}(\gamma^{\prime})}
\cdot 
{\textstyle\prod_{\gamma^{\prime} \in \icaps(C_{\rm in})}}\; 
\parsign^{s_{\Lambda}(\gamma^{\prime})(\pos_{\Lambda}(\gamma^{\prime})+1)}
\\
\cdot
{\textstyle\prod_{\gamma^{\prime} \in\erights(C_{\rm in})}}\; 
\parsign^{\pos_{\Lambda}(\gamma^{\prime})}
\cdot 
{\textstyle\prod_{\gamma^{\prime} \in\elefts(C_{\rm in})}}\;
\parsign^{\pos_{\Lambda}(\gamma^{\prime})+1}
\cdot \parsign \cdot 
\parsign^{\neatfrac{1}{4}\left(\length_\Lambda(C_{\mathrm{in}})-2\right)} \cdot \parsign^{s_{\Lambda}(\gamma)}
\\
\overset{(\mathrm{I})}{=}\coeff_{\parsign}(C^{\rm anti}_{\rm out})  
\cdot 
{\textstyle\prod_{\gamma^{\prime} \in \ecups(C_{\rm in})}}\; 
\parsign^{(s_{\Lambda}(\gamma^{\prime})+1)\pos_{\Lambda}(\gamma^{\prime})}
\cdot
{\textstyle\prod_{\gamma^{\prime} \in \ecaps(C_{\rm in})}}\; 
\parsign^{s_{\Lambda}(\gamma^{\prime})(\pos_{\Lambda}(\gamma^{\prime})+1)} 
\\
\cdot
{\textstyle\prod_{\gamma^{\prime}\in\irights(C_{\rm in})}}\;
\parsign^{\pos_{\Lambda}(\gamma^{\prime})}
\cdot 
{\textstyle\prod_{\gamma^{\prime}\in\ilefts(C_{\rm in})}}\;
\parsign^{\pos_{\Lambda}(\gamma^{\prime})+1}
\cdot 
\parsign^{\pos_{\Lambda}(\gamma)+s_{\Lambda}(\gamma)}
\overset{\eqref{eq:cupcapthingy-module} + \eqref{eq:second-cupcapthingy-module}}{=}
\coeff_{\parsign}(C^{\rm anti}),
\end{gathered}
\end{gather}
\begin{gather}\label{eq:crazy-equation-next}
\begin{gathered}
\coeff_{\parsign}(C^{\rm anti}_{\rm out}) 
\cdot \coeff_{\parsign}(C^{\rm anti}_{\rm in}) \cdot \parsign  \cdot \parsign^{\neatfrac{1}{4}\left(\length_\Lambda(C_{\mathrm{in}})-2\right)} 
\cdot \parsign^{s_{\Lambda}(\gamma)}
\\
=\coeff_{\parsign}(C^{\rm anti}_{\rm out})  
\cdot 
{\textstyle\prod_{\gamma^{\prime} \in \icups(C_{\rm in})}}\; 
\parsign^{(s_{\Lambda}(\gamma^{\prime})+1)\pos_{\Lambda}(\gamma^{\prime})}
\cdot 
{\textstyle\prod_{\gamma^{\prime} \in \icaps(C_{\rm in})}}\; 
\parsign^{s_{\Lambda}(\gamma^{\prime})(\pos_{\Lambda}(\gamma^{\prime})+1)}
\\
\cdot
{\textstyle\prod_{\gamma^{\prime} \in\erights(C_{\rm in})}}\; 
\parsign^{\pos_{\Lambda}(\gamma^{\prime})}
\cdot 
{\textstyle\prod_{\gamma^{\prime} \in\elefts(C_{\rm in})}}\;
\parsign^{\pos_{\Lambda}(\gamma^{\prime})+1}
\cdot \parsign \cdot \parsign^{\neatfrac{1}{4}\left(\length_\Lambda(C_{\mathrm{in}})-2\right)} \cdot \parsign^{s_{\Lambda}(\gamma)}
\\
\overset{(\mathrm{I})}{=} \coeff_{\parsign}(C^{\rm anti}_{\rm out})  
\cdot 
{\textstyle\prod_{\gamma^{\prime} \in \ecups(C_{\rm in})}}\; 
\parsign^{(s_{\Lambda}(\gamma^{\prime})+1)\pos_{\Lambda}(\gamma^{\prime})}
\cdot
{\textstyle\prod_{\gamma^{\prime} \in \ecaps(C_{\rm in})}}\; 
\parsign^{s_{\Lambda}(\gamma^{\prime})(\pos_{\Lambda}(\gamma^{\prime})+1)}
\\
\cdot
{\textstyle\prod_{\gamma^{\prime}\in\irights(C_{\rm in})}}\;
\parsign^{\pos_{\Lambda}(\gamma^{\prime})}
\cdot 
{\textstyle\prod_{\gamma^{\prime}\in\ilefts(C_{\rm in})}}\;
\parsign^{\pos_{\Lambda}(\gamma^{\prime})+1}
\cdot 
\parsign^{\pos_{\Lambda}(\gamma)+s_{\Lambda}(\gamma)}
\overset{\eqref{eq:cupcapthingy-module} + \eqref{eq:second-cupcapthingy-module}}{=}
\coeff_{\parsign}(C^{\rm anti}),
\end{gathered}
\end{gather}
\begin{gather}\label{eq:crazy-equation-2}
\begin{gathered}
\coeff_{\paro}(C^{\rm anti}_{\rm out}) \cdot \coeff_{\paro}(C^{\rm anti}_{\rm in})
\\
=\coeff_{\paro}(C^{\rm anti}_{\rm out}) \cdot 
{\textstyle\prod_{\gamma^{\prime} \in \icups(C_{\rm in})}}\; 
\paro^{-s_{\Lambda}(\gamma^{\prime})}
\cdot 
{\textstyle\prod_{\gamma^{\prime} \in \icaps(C_{\rm in})}}\; 
\paro^{s_{\Lambda}(\gamma^{\prime})-1} \cdot
\paro^{\# \left(\eright(C)\,\cup\,\eleft(C)\right)}
\\
\overset{(\mathrm{II})}{=} \coeff_{\paro}(C^{\rm anti}_{\rm out}) \cdot 
{\textstyle\prod_{\gamma^{\prime} \in \ecups(C_{\rm in})}} 
\paro^{-s_{\Lambda}(\gamma^{\prime})} \cdot 
{\textstyle\prod_{\gamma^{\prime} \in \ecaps(C_{\rm in})}} 
\paro^{s_{\Lambda}(\gamma^{\prime})-1} \cdot
\paro^{\#\left( \iright(C)\,\cup\,\ileft(C)\right)} \cdot \paro
\\
=\coeff_{\paro}(C^{\rm anti}).
\end{gathered}
\end{gather}\makeautorefname{lemma}{Lemmas}
Here (I) follows 
from \fullref{lemma:comparecircleinversion} and \ref{lemma:posandtag-forrays},
while (II) follow from\makeautorefname{lemma}{Lemma}  
\fullref{lemma:comparesaddles-bimodules}.
The arguments for the other orientations are as 
for \fullref{proposition:crazy-isomorphism}.
\subsection{Proof of \texorpdfstring{\fullref{proposition:cats-are-equal-yes}}{Proposition \ref{proposition:cats-are-equal-yes}}}\label{subsec:proof-four}
We are only concerned about $\ring$-ranks of hom-spaces 
in this proof. Consequently, by web evaluation 
\eqref{eq:web-eval} and braid closing, we can stick 
to the upwards oriented setup. Let us write $\webalg[g]^{\based}$ respectively $\webalg[s]^{\based}$ 
if we are in the generic situation or in case \ref{enum:generic} respectively in 
case \ref{enum:semisimple}. We also write $\webalg[*]^{\based}$ if we mean both cases 
(for simplicity of notation we extend scalars to $\ring$).
\subsubsection{The problem}
Because of \fullref{proposition:cats-are-equal} it suffices to 
show that the $\ring$-modules
$\Hom_{\webalg[*]^{\based}}(\M[*]^{\based}(u),\M[*]^{\based}(v))$ are
free of finite rank and to match their ranks. 
The first task is clear because of the cup foam bases. 
The main difficulty is thus to control the number 
of $\webalg[*]^{\based}$-bimodule homomorphisms. We do so by analyzing the 
decomposition structure. 

To this end, 
recall from \fullref{subsec:algmodel} that, 
given a web $u$, then
we can associate to it a $\compmatch$-composite matching $\cc{u}$
by erasing orientations and phantom edges. 
Here choose any presentation of the associated 
$\compmatch$-composite matching $\cc{u}$ in terms of the
basic moves from \cite[(31)]{EST}. 
From this we obtain an $\ArcalgS[*]$-bimodule 
$\Mo[*](\cc{u})$ associated to $\M[*]^{\based}(u)$. 
(The careful reader might want to check 
that different choices in terms of basic moves 
give isomorphic $\ArcalgS[*]$-bimodules.)

The main ingredient in order to control the number of 
$\webalg[a]^{\based}$-bimodule homomorphisms 
is to first use the results from \fullref{subsec:algmodel} 
to identify $\webalg[a]^{\based}$ and its web bimodules with $\ArcalgS[a]$ 
and its arc bimodules. Then further 
identify $\ArcalgS[a]$ with $\ArcalgS[\KBN]$ and their 
arc bimodules by using \makeautorefname{subsection}{Sections}
the results from \fullref{subsec:iso-alg} and \ref{subsec:iso-bimod}.
Then we can use statements obtained 
in \cite{BS2} and \cite{BS3} 
as we explain below.
Hereby we note that these results work, mutatis mutandis, 
in the generic case as well. \makeautorefname{subsection}{Section}

We have 
$\twoHom_{\F[a]}(u,v)
\cong\twoHom_{\F[a]}(\oneinsert{2\omega_{\ell}},\clapping(u)\clapping(v)^*)\langle(\vec{k})\rangle$
as graded, free $\ring$-modules,
cf. the proof of \fullref{proposition:cats-are-equal}. Thus, using the cup foam 
basis and the translation to the side of $\ArcalgS[a]$ 
from \fullref{lemma:matchvs}, 
the (graded) rank of $\twoHom_{\F[a]}(u,v)$ is precisely 
given by all orientations of the composite matching for 
$\cc{\clapping(u)\clapping(v)^*}$ 
and their degrees. 
Thus, we have to show the same on the side of $\ArcalgS[\KBN]$:
\begin{gather}\label{eq:stupid-identity}
\begin{aligned}
\mathrm{rank}_{\ring}&\left(\Hom_{\ArcalgS[\KBN]}(\Mo[\KBN](\cc{u}),\Mo[\KBN](\cc{v})\langle s\rangle)\right)
\\
&
\overunder{\text{need to}}{\text{show}}{=}
\#\{\text{orientations of }\cc{\clapping(u)\clapping(v)^*}\text{ of degree }s\}.
\end{aligned}
\end{gather}
\subsubsection{The two cases in the proposition}\hfill\\
\noindent\textbf{Case \ref{enum:generic}.} 
Assume first that neither $u$ nor $v$ have internal circles. 
Then $\cc{u}$ and $\cc{v}$ fit into the framework  
from \cite[Section 4]{BS2}, i.e. \cite[Theorems 3.6 and 4.14]{BS2} show that
\begin{gather}\label{eq:stupid-identity-2}
\Mo[\KBN](\cc{u})
\text{ is indecomposable 
iff }\cc{u}\text{ does not contain internal circles.}
\end{gather}
It follows now from \cite[Theorem 3.5]{BS3} that \eqref{eq:stupid-identity} 
holds in case $\parto=0$ and $R=\C$. 
Scrutiny of the arguments used in \cite{BS2} and \cite{BS3} 
shows that these work under the circumstances of case \ref{enum:generic} as well. 
Next, if $C$ is any circle in $u$, and thus, in $\cc{u}$, then
\begin{gather}\label{eq:stupid-identity-3}
\Mo[\KBN](\cc{u})\cong\Mo[\KBN](\cc{u}{-}C)\langle+1\rangle\oplus\Mo[\KBN](\cc{u}{-}C)\langle-1\rangle,
\end{gather}
which follows as in \cite[Example 3.22]{EST}, similarly for $v$.
The right-hand side of \eqref{eq:stupid-identity} 
behaves in the same way, i.e. for 
$w=\cc{\clapping(u{-}C)\clapping(v)^*}$
we have
\begin{gather}\label{eq:stupid-identity-4}
\begin{gathered}
\#\{\text{orientations of }\cc{\clapping(u)\clapping(v)^*}\text{ of degree }s\}
\\
=
\#\{\text{orientations of }w\text{ of degree }s+1\}
+
\#\{\text{orientations of }w\text{ of degree }s-1\}.
\end{gathered}
\end{gather}
and similarly for $v$. The claim \eqref{eq:stupid-identity} follows in case \ref{enum:generic}.

\noindent\textbf{Case \ref{enum:semisimple}.} 
As before, it suffices to study the case where $u$ and $v$ do not have internal circles. 
In this case $\twoEnd_{\F[b]}(u)$ 
has a basis which locally looks like
\begin{gather}\label{eq:the-idempotents}
\xy
(0,0)*{\includegraphics[scale=.65]{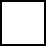}};
\endxy
\quad,\quad
\xy
(0,0)*{\includegraphics[scale=.65]{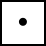}};
\endxy
\quad,\quad
\xy
(0,0)*{\includegraphics[scale=.65]{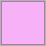}};
\endxy
\end{gather}
This can be shown by using the cup foam basis. 
Because of \eqref{eq:theusualrelations1}--\eqref{eq:theusualrelations2}, we can ignore the phantom parts of any foam $f\in\twoEnd_{\F[a]}(u)$. A direct calculation shows that 
\begin{gather}\label{eq:stupid-identity-5}
e_+=\neatfrac{1}{2}\left(\,
\xy
(0,0)*{\includegraphics[scale=.65]{figs/fig4-50.pdf}};
\endxy
+\sqrt{\specS(\parto)^{-1}}\cdot
\xy
(0,0)*{\includegraphics[scale=.65]{figs/fig4-51.pdf}};
\endxy
\,
\right)
\quad\text{and}\quad
e_-=\neatfrac{1}{2}\left(\,
\xy
(0,0)*{\includegraphics[scale=.65]{figs/fig4-50.pdf}};
\endxy
-\sqrt{\specS(\parto)^{-1}}\cdot
\xy
(0,0)*{\includegraphics[scale=.65]{figs/fig4-51.pdf}};
\endxy
\,
\right)
\end{gather}
are idempotents satisfying $e_+e_-=0=e_-e_+$ and $1=e_++e_-$. 
If $u$ has $c$ connected components---ignoring phantom edges, but 
counting both adjacent usual edges---then
\begin{gather}\label{eq:stupid-identity-6}
E=\{\vec{e}=(e_{1},\dots,e_{2^c})\mid e_i=e_{\pm}, i=1,\dots,2^c\}\setminus\{0\}
\end{gather}
(some $\vec{e}\,$'s might be zero, see below)
gives a complete set of pairwise orthogonal 
idempotents in $\twoEnd_{\F[b]}(u)$. 
Here the idempotents $\vec{e}$ 
are obtained 
by spreading the idempotents $e_+$ and $e_-$ 
locally around a trivalent vertex as 
follows:
\begin{gather}\label{eq:funny-idempotents}
\parsign=1\colon
\xy
(0,0)*{\includegraphics[scale=.65]{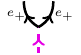}};
\endxy
\quad
\text{and}
\quad
\xy
(0,0)*{\includegraphics[scale=.65]{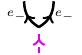}};
\endxy
\quad
,
\quad
\parsign=-1\colon
\xy
(0,0)*{\includegraphics[scale=.65]{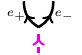}};
\endxy
\quad
\text{and}
\quad
\xy
(0,0)*{\includegraphics[scale=.65]{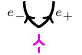}};
\endxy
\end{gather}
These are idempotents as one easily checks, while the other possible 
combinations give zero. 
This shows that, with $w=\clapping(u)\clapping(v)^*$ 
being the clapped web,
\begin{gather}\label{eq:match-some-idempotent-homs}
\begin{aligned}
\mathrm{rank}_{\ring}(\twoHom_{\F[b]}(u,v))
=\#\{\text{non-zero }\text{``colorings'' of }w\text{ with the }e_{\pm}\text{'s}\}.
\end{aligned}
\end{gather}
Using the idempotents $E$ one shows as in \cite[Proposition 3.13]{MPT}
that $\webalg[b]_{\vec{k}}$ 
is semisimple for all $\vec{k}\in\bY$. 

A web bimodule $\M[b](u)$ 
for $u$ having $c$ connected components decomposes 
into pairwise non-isomorphic copies of $\ring$, i.e.
$\M[b](u)\cong {\textstyle\bigoplus_{\vec{e}\in E}}\; \vec{e}\,\M[b](u)\vec{e}$.
The claim follows, since---for 
$w=\clapping(u)\clapping(v)^*$---we get
that the $\ring$-rank of $\Hom_{\webalg[b]}(\M[b](u),\M[b](v))$ 
equals the number of non-zero ``colorings'' of $w$ with the $e_{\pm}$'s.
By \eqref{eq:match-some-idempotent-homs}, 
$\mathrm{rank}_{\ring}(\twoHom_{\F[b]}(u,v))
=\mathrm{rank}_{\ring}(\Hom_{\webalg[b]}(\M[b](u),\M[b](v)))$.

\noindent Now, \eqref{eq:stupid-identity} is verified and thus \fullref{proposition:cats-are-equal-yes} is proven.
\subsection{Proof of \texorpdfstring{\fullref{proposition:its-an-invariant!}}{Proposition \ref{proposition:its-an-invariant!}}}\label{subsec:proof-five}
Let us start with the proof in the upwards oriented 
situation $\Hgen{{{}_-}}$ since---as we will see below---it is slightly more involved. 
Afterwards we can be short in case of $\Hgenpm{{{}_-}}$.
\subsubsection{The two tasks}
Note that  
composition of tangles corresponds, by construction, to 
tensoring of bimodule complexes. 
(The careful reader might want to 
copy \cite[Proposition 13]{Khov} to see this.)
Thus, there are two things to show: 
we have to show that 
$\Hgen{{{}_-}}$ does not depend on our choice of the map 
$\web({}_-)$, 
and we have to show invariance 
under the tangle Reidemeister moves from \fullref{definition:tanglecat}. 
\subsubsection{Independence of choice}  
In case of $\Hgenpm{{{}_-}}$ this task is empty, since there are no choices involved in its definition.
Given some ${}_{\vec{s}}\overline{T}_{\vec{t}}$. 
Assume that we have two different choices $\web_1({}_{\vec{s}}\overline{T}_{\vec{t}})$ and
$\web_2({}_{\vec{s}}\overline{T}_{\vec{t}})$.
To show independence we have to prove that
\begin{gather}\label{eq:new-stuff-1}
\Hgen{\web_1({}_{\vec{s}}\overline{T}_{\vec{t}})}=
\Hgen{\web_2({}_{\vec{s}}\overline{T}_{\vec{t}})}
\text{ as complexes in }\Komh(\webcatcirc),
\end{gather}
Now, two different choices $\web_1({}_{\vec{s}}\overline{T}_{\vec{t}})$ and 
$\web_2({}_{\vec{s}}\overline{T}_{\vec{t}})$ can only 
differ by the following local moves: 
The \textit{ordinary-ordinary-phantom} \textbf{R3} \textit{moves} 
(similarly for a negative crossing)
\begin{gather}\label{eq:move-to-show1}
\xy
(0,0)*{\includegraphics[scale=.65]{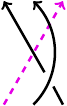}};
\endxy
\,\sim\,
\xy
(0,0)*{\includegraphics[scale=.65]{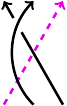}};
\endxy
\quad,\quad
\xy
(0,0)*{\includegraphics[scale=.65]{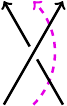}};
\endxy
\,\sim\,
\xy
(0,0)*{\includegraphics[scale=.65]{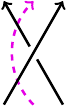}};
\endxy
\quad,\quad
\xy
(0,0)*{\includegraphics[scale=.65]{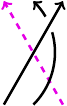}};
\endxy
\,\sim\,
\xy
(0,0)*{\includegraphics[scale=.65]{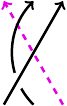}};
\endxy
\end{gather}
The \textit{ordinary-phantom} \textbf{R2} \textit{moves} 
and the \textit{pure-phantom}
\textbf{R2} \textit{move}
\begin{gather}\label{eq:move-to-show2}
\xy
(0,0)*{\includegraphics[scale=.65]{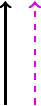}};
\endxy
\,\sim\,
\xy
(0,0)*{\includegraphics[scale=.65]{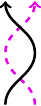}};
\endxy
\quad,\quad
\xy
(0,0)*{\includegraphics[scale=.65]{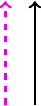}};
\endxy
\,\sim\,
\xy
(0,0)*{\includegraphics[scale=.65]{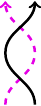}};
\endxy
\quad,\quad
\xy
(0,0)*{\includegraphics[scale=.65]{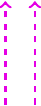}};
\endxy
\,\sim\,
\xy
(0,0)*{\includegraphics[scale=.65]{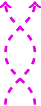}};
\endxy
\end{gather}
Third, the \textit{ordinary-phantom-phantom} 
\textbf{R3} \textit{moves} and the \textit{pure-phantom} 
\textbf{R3} \textit{move}
\begin{gather}\label{eq:move-to-show3}
\xy
(0,0)*{\includegraphics[scale=.65]{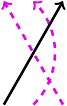}};
\endxy
\sim
\xy
(0,0)*{\includegraphics[scale=.65]{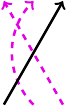}};
\endxy
\quad,\quad
\xy
(0,0)*{\includegraphics[scale=.65]{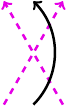}};
\endxy
\sim
\xy
(0,0)*{\includegraphics[scale=.65]{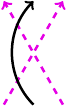}};
\endxy
\quad,\quad
\xy
(0,0)*{\includegraphics[scale=.65]{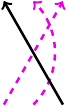}};
\endxy
\sim
\xy
(0,0)*{\includegraphics[scale=.65]{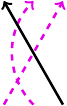}};
\endxy
\quad,\quad
\xy
(0,0)*{\includegraphics[scale=.65]{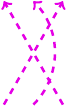}};
\endxy
\sim
\xy
(0,0)*{\includegraphics[scale=.65]{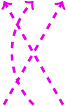}};
\endxy
\end{gather}
Thus, it suffices to show that the chain complex stays the 
same---up to chain homotopy---if the two choices 
$\web_1({}_{\vec{s}}\overline{T}_{\vec{t}})$ and 
$\web_2({}_{\vec{s}}\overline{T}_{\vec{t}})$ differ by one of these moves. 
For this purpose, we have by \cite[Lemmas 4.4 and 4.5]{EST} 
(which work for $\parameter$ as well)
that the chain groups 
are isomorphic. (The slogan is ``isotopic webs give isomorphic 
web bimodules). 
These isomorphisms are given by the evident foams, 
which are clearly $\webalg^\based[\parameter]$-bimodule homomorphisms. 
Thus, it remains to show that these commute with 
the differentials 
in the associated complexes. 
This is clear for the moves \eqref{eq:move-to-show2} 
or \eqref{eq:move-to-show3}. For  \eqref{eq:move-to-show1} 
one easily checks that these isomorphisms 
commute with 
the differentials. Indeed, via clapping one sees 
that the corresponding hom-spaces between the two sides are free 
of rank one in the appropriated degree. Thus, since the 
chain isomorphism are given by foams which are identities on
the ordinary parts, everything has to work out 
up to units. The involved complexes are 
however of length two, and so the foams can be scaled appropriately.
\subsubsection{Tangle Reidemeister moves (upwards case)}
Denote by 
${}_{\vec{s}}\overline{T}_{\vec{t}}$ 
and ${}_{\vec{s}}\overline{\mathcal{T}}_{\vec{t}}$ 
two tangle diagrams that differ by 
one of the moves \eqref{enum:reide1}--\eqref{enum:reide3}. Again, 
if we show that
\begin{gather}\label{eq:new-stuff-2}
\Hgen{{}_{\vec{s}}\overline{T}_{\vec{t}}}=\Hgen{{}_{\vec{s}}\overline{\mathcal{T}}_{\vec{t}}}
\text{ as complexes in }\Komh(\webcatcirc),
\end{gather}
then we are done. The main point is invariance 
under a move from \eqref{enum:reide2} or \eqref{enum:reide3}. Indeed, invariance 
under a move from \eqref{enum:reide1} 
follows as above, 
because e.g. we can by the above assume that a zigzag move looks 
locally as in \fullref{example:from-tangles-to-webs} 
and then use the same arguments as before (``isotopic webs give isomorphic web bimodules''). 
Hence, it remains 
to show invariance under the moves \eqref{enum:reide2} 
and \eqref{enum:reide3}, i.e. we have to 
check the following, together with variations of these:
\begin{gather}\label{eq:wrong-r2}
\xy
(0,0)*{
\xy
(0,0)*{\includegraphics[scale=.65]{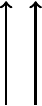}};
\endxy
\,
\stackrel{\boldsymbol{\mathrm{R2}}}{=}
\,
\xy
(0,0)*{\includegraphics[scale=.65]{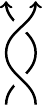}};
\endxy
\quad,\quad
\xy
(0,0)*{\includegraphics[scale=.65]{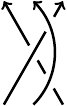}};
\endxy
\,
\stackrel{\boldsymbol{\mathrm{R3}}}{=}
\,
\xy
(0,0)*{\includegraphics[scale=.65]{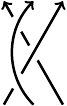}};
\endxy};
(0,-8)*{\text{\tiny braid-like}};
\endxy
\quad;\quad
\xy
(0,0)*{
\xy
(0,0)*{\includegraphics[scale=.65]{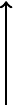}};
\endxy
\,
\stackrel{\boldsymbol{\mathrm{R1}}}{=}
\,
\xy
(0,0)*{\includegraphics[scale=.65]{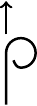}};
\endxy
\quad,\quad
\xy
(0,0)*{\reflectbox{\includegraphics[scale=.65]{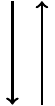}}};
\endxy
\,
\stackrel{\boldsymbol{\mathrm{mR2}}}{=}
\,
\xy
(0,0)*{\reflectbox{\includegraphics[scale=.65]{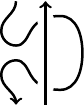}}};
\endxy};
(0,-8)*{\text{\tiny non-braid-like}};
\endxy
\end{gather}

We start with the \textit{braid-like moves}, which are 
much easier than the \textit{non-braid-like moves}. For these we can basically copy the usual 
(known) homotopies.

\begin{enumerate}[label=(\roman*)]

\item \textbf{R2} \textit{move, both versions.} We use the same 
cobordisms and coefficients 
as in \cite[Sections 3.4 and 3.5]{Bla}.

\item \textbf{R3} \textit{move, all versions.} This can be showed 
using the usual (abstract) Gauss elimination argument, 
as pioneered in \cite{BN2}. To be precise, 
one uses the $\parameter$-analog of \cite[Lemma 4.3]{EST}---the ``circle 
removal''---and then twice the Gauss elimination from \cite[Lemma 3.2]{BN2}. 
One obtains that the two complexes for both sides of the \textbf{R3} move 
have isomorphic chain groups. These can then be matched directly. We leave the details 
to the reader, where we note that all appearing coefficients are trivial, 
because the ``complicated'' maps in the Gauss elimination 
are at extremal parts of the complexes.

\item \textbf{R1} \textit{move, right curl with a 
positive or negative crossing.} 
In the setup of upwards oriented webs, 
this move is quite different from the one in e.g. \cite{Bla}. Thus, 
we give the chain maps and homotopies explicitly (omitting the obvious shifts). 
First, we want  chain maps to establish the first Reidemeister move for the positive crossing:
\begin{equation}\label{eq:new-stuff-3}
\begin{tikzpicture}[anchorbase]
\matrix (m) [matrix of math nodes, row sep=4em, column
sep=3em, text height=1.8ex, text depth=0.25ex] {
\Hgen{
	\xy
	(0,0)*{\includegraphics[scale=.65]{figs/fig5-57.pdf}};
	\endxy
}
=
&
\phantom{.}
\xy
(0,0)*{\includegraphics[scale=.65]{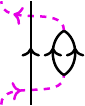}};
\endxy
&
\phantom{.}
\xy
(0,0)*{\includegraphics[scale=.65]{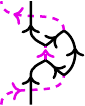}};
\endxy
\\
\phantom{aa}
\Hgen{
	\xy
	(0,0)*{\;\includegraphics[scale=.65]{figs/fig5-56.pdf}};
	\endxy
}
=
&\phantom{.}
\xy
(0,0)*{\includegraphics[scale=.65]{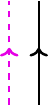}};
\endxy
&
0
\\
};
\path[<-,blue] ($(m-1-2.east) + (0,.065)$) edge node[above] {\text{\tiny$h$}} ($(m-1-3.west) + (.2,.065)$);
\path[<-] ($(m-1-3.west) + (.2,-.065)$) edge node[below] {\text{\tiny$d$}} ($(m-1-2.east) + (0,-.065)$);
\path[->, mycolor] ($(m-1-2.south) + (-.1,-.5)$) edge ($(m-1-2.south) + (-.1,-.95)$);
\path[<-, mycolor] ($(m-1-2.south) + (.1,-.5)$) edge ($(m-1-2.south) + (.1,-.95)$);
\node[mycolor] at ($(m-1-2.south) + (-.3,-.725)$) {\text{\tiny$f$}};
\node[mycolor] at ($(m-1-2.south) + (.3,-.725)$) {\text{\tiny$g$}};
\path[->] (m-2-2) edge (m-2-3);
\end{tikzpicture}
\end{equation}
\vspace*{0.25cm}
The three maps $f$, $g$ and $h$ are given by the following scaled maps of cobordisms:
\begin{gather}\label{eq:new-stuff-4}
\xymatrix@C2.0cm{
\text{$\phantom{.}$}
\xy
(0,0)*{\includegraphics[scale=0.75]{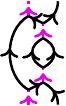}};
(0,-8)*{\text{\tiny\text{start }$f$}};
(0,8)*{\text{\tiny\text{end }$g$}};
\endxy
\ar@<-2pt>@[mycolor][rrr]_{\xy
	(0,0)*{\includegraphics[scale=0.325]{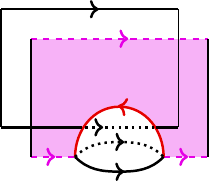}};
	\endxy}
\ar@{.>}@<-27pt>@[mycolor]@/_.52cm/[rrrr]|{\,{\color{mycolor}\parep^2\parom^{-1}\parome^{-2}}\,}
& & &
\text{$\phantom{.}$}
\xy
(0,0)*{\includegraphics[scale=0.75]{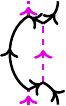}};
\endxy
\ar@<-2pt>@[mycolor][r]_{\xy
	(0,0)*{\includegraphics[scale=0.325]{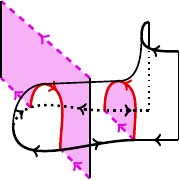}};
	\endxy}
\ar@<-2pt>@[mycolor][lll]_{
	\xy
	(0,0)*{\includegraphics[scale=0.325]{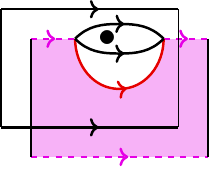}};
	\endxy
	-\parom^{-1}\parome
	\xy
	(0,0)*{\includegraphics[scale=0.325]{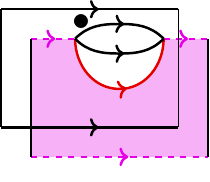}};
	\endxy
	-\parom^{-1}\parome\parha
	\xy
	(0,0)*{\includegraphics[scale=0.325]{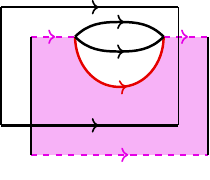}};
	\endxy
}
&
\text{$\phantom{.}$}
\xy
(0,0)*{\includegraphics[scale=0.75]{figs/fig5-rm1-check1.pdf}};
(0,-8)*{\text{\tiny\text{end }$f$}};
(0,8)*{\text{\tiny\text{start }$g$}};
\endxy
\ar@<-2pt>@[mycolor][l]_{\xy
	(0,0)*{\includegraphics[scale=0.325]{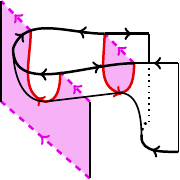}};
	\endxy}
}
\end{gather}
\begin{gather}\label{eq:new-stuff-5}
\xymatrix@C2cm{
\text{$\phantom{.}$}
\xy
(0,0)*{\includegraphics[scale=0.75]{figs/fig5-rm1-check5.pdf}};
(0,-8)*{\text{\tiny\text{end }$h$}};
\endxy
&
\text{$\phantom{.}$}
\xy
(0,0)*{\includegraphics[scale=0.75]{figs/fig5-rm1-check4.pdf}};
\endxy
\ar@[myblue][l]|{
	\;
	\xy
	(0,0)*{\includegraphics[scale=0.325]{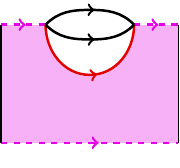}};
	\endxy\;}
&
\text{$\phantom{.}$}
\xy
(0,0)*{\includegraphics[scale=0.75]{figs/fig5-rm1-check5.pdf}};
\endxy
\ar@[myblue][l]|{
	\;
	\xy
	(0,0)*{\includegraphics[scale=0.325]{figs/fig2-61.pdf}};
	\endxy\;}
&
\text{$\phantom{.}$}
\xy
(0,0)*{\includegraphics[scale=0.75]{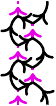}};
(0,-8)*{\text{\tiny\text{start }$h$}};
\endxy
\ar@[myblue][l]|{
	\;
	\xy
	(0,0)*{\includegraphics[scale=0.325]{figs/fig2-60.pdf}};
	\endxy\;}
\ar@{.>}@<-15pt>@[myblue]@/_.3cm/[lll]|{\,{\color{myblue}\parep\parom^{-2}}\,}
}
\end{gather}
(The dotted arrows indicate a scaling.) One checks now that $d\circ g=0$, $f\circ g=\mathrm{id}$, 
$\mathrm{id} - g\circ f = h\circ d$ and $d\circ h=\mathrm{id}$.
Similarly for the negative crossing, but with the roles of  $f$ and $g$ exchanged.

\item \textbf{R1} \textit{move, left curl with a 
positive or negative crossing.} Similarly as for the right curl, 
but exchanging the roles of $\parom$ and $\parome$.

\item \textbf{mR2} \textit{move, version in \eqref{eq:wrong-r2}.} 
As in case of \textbf{R1} we give the chain maps and homotopies explicitly:
\begin{equation}\label{eq:new-stuff-11}
\begin{tikzpicture}[anchorbase]
\matrix (m) [matrix of math nodes, row sep=4em, column
sep=3em, text height=1.8ex, text depth=0.25ex] {
\Hgen{
	\xy
	(0,0)*{\reflectbox{\includegraphics[scale=.65]{figs/fig5-mrm0a.pdf}}};
	\endxy
}
=
&
\raisebox{.05cm}{\xy
	(0,0)*{\reflectbox{\includegraphics[scale=.65]{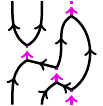}}};
	\endxy}
&
\raisebox{.05cm}{\xy
	(0,0)*{\reflectbox{\includegraphics[scale=.65]{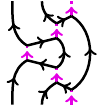}}};
	\endxy}
\oplus
\xy
(0,0)*{\reflectbox{\includegraphics[scale=.65]{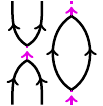}}};
\endxy
&
\raisebox{.05cm}{\xy
	(0,0)*{\reflectbox{\includegraphics[scale=.65]{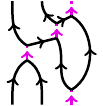}}};
	\endxy}
\\
\phantom{aa}
\raisebox{-.1cm}{$\Hgen{
		\xy
		(0,0)*{\reflectbox{\includegraphics[scale=.65]{figs/fig5-mrm0b.pdf}}};
		\endxy
	}$}
=
&
0
&
\xy
(0,0)*{\reflectbox{\includegraphics[scale=.65]{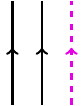}}};
\endxy
&
0
\\
};
\path[->] (m-1-2) edge node[below] {\text{\tiny$\begin{pmatrix}d_1 \\ d_2 \end{pmatrix}$}} (m-1-3);
\path[->] (m-1-3) edge node[below] {\text{\tiny$\begin{pmatrix}d_3, -d_4 \end{pmatrix}$}} (m-1-4);
\path[<-, mycolor] ($(m-1-3.south) + (-.7,-.5)$) edge ($(m-2-3.north) + (-.45,.5)$);
\path[<-, mycolor] ($(m-2-3.north) + (-.65,.5)$) edge ($(m-1-3.south) + (-.9,-.5)$);
\node[mycolor] at ($(m-2-3.north) + (-.95,.7)$) {\text{\tiny$f_1$}};
\node[mycolor] at ($(m-2-3.north) + (-.35,.7)$) {\text{\tiny$f_2$}};
\path[->] ($(m-1-3.south) + (.7,-.5)$) edge ($(m-2-3.north) + (.45,.5)$);
\path[->] ($(m-2-3.north) + (.65,.5)$) edge ($(m-1-3.south) + (.9,-.5)$);
\node at ($(m-2-3.north) + (1,.7)$) {\text{\tiny$g_2$}};
\node at ($(m-2-3.north) + (.325,.7)$) {\text{\tiny$g_1$}};
\draw[mygreen, ->] ($(m-1-4.north) + (-.25,.5)$) to  [out=158, in=22] node[above] {
\text{\tiny$h_2$}} ($(m-1-3.north) + (1.25,.5)$);
\draw[myblue, ->] ($(m-1-3.north) + (.75,.5)$) to [out=169, in=11] node[above] {
\text{\tiny$h_1$}} ($(m-1-2.north) + (.25,.5)$);
\path[->] (m-2-2) edge (m-2-3);
\path[->] (m-2-3) edge (m-2-4);
\end{tikzpicture}
\end{equation}
\vspace*{.25cm}
\begin{gather}\label{eq:new-stuff-12}
\xymatrix@C12mm{
\xy
(0,0)*{\reflectbox{\includegraphics[scale=0.75]{figs/fig5-mrm2.pdf}}};
(0,-8)*{\text{\tiny\text{start/end }$f$'s}};
\endxy
\ar@<2pt>@[mycolor][r]^{\xy
	(0,0)*{\includegraphics[scale=0.325]{figs/fig5-squeeze1.pdf}};
	\endxy}
&
\xy
(0,0)*{\reflectbox{\includegraphics[scale=0.75]{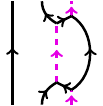}}};
\endxy
\ar@<2pt>@[mycolor][r]^{\xy
	(0,0)*{\includegraphics[scale=0.325]{figs/fig5-squeeze1.pdf}};
	\endxy}
\ar@<2pt>@[mycolor][l]^{\xy
	(0,0)*{\includegraphics[scale=0.325]{figs/fig5-squeeze2.pdf}};
	\endxy}
&
\xy
(0,0)*{\reflectbox{\includegraphics[scale=0.75]{figs/fig5-mrm-equi2.pdf}}};
(0,-8)*{\text{\tiny\text{end/start }$f$'s}};
(0,8)*{\text{\tiny\text{end/start }$g$'s}};
\endxy
\ar@<-2pt>[r]_{\xy
	(0,0)*{\includegraphics[scale=0.325]{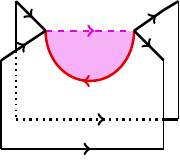}};
	\endxy}
\ar@<2pt>@[mycolor][l]^{\xy
	(0,0)*{\includegraphics[scale=0.325]{figs/fig5-squeeze2.pdf}};
	\endxy}
\ar@{.>}@<25pt>@[mycolor]@/^.5cm/[ll]|{\,{\color{mycolor}\parep^{-2}\parome^{4}}\,}
&
\xy
(0,0)*{\reflectbox{\includegraphics[scale=0.75]{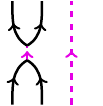}}};
\endxy
\ar@<-2pt>[r]_{\xy
	(0,0)*{\includegraphics[scale=0.325]{figs/fig2-62.pdf}};
	\endxy}
\ar@<-2pt>[l]_{\xy
	(0,0)*{\includegraphics[scale=0.325]{figs/fig2-60.pdf}};
	\endxy}
&
\xy
(0,0)*{\reflectbox{\includegraphics[scale=0.75]{figs/fig5-mrm3.pdf}}};
(0,8)*{\text{\tiny\text{start/end }$g$'s}};
\endxy
\ar@<-2pt>[l]_{\xy
	(0,0)*{\includegraphics[scale=0.325]{figs/fig2-61.pdf}};
	\endxy}
\ar@{.>}@<-25pt>@/^-.5cm/[ll]|{\,-\parep^{-2}\parome^{4}\,}
}
\end{gather}
\begin{gather}\label{eq:new-stuff-13}
\xymatrix@C12mm{
\xy
(0,0)*{\reflectbox{\includegraphics[scale=0.75]{figs/fig5-mrm4.pdf}}};
(0,-8)*{\text{\tiny\text{end }$h_1$}};
\endxy
&
\xy
(0,0)*{\reflectbox{\includegraphics[scale=0.75]{figs/fig5-mrm-equi3.pdf}}};
\endxy
\ar@[myblue][l]^{\xy
	(0,0)*{\includegraphics[scale=0.325]{figs/fig5-squeeze2.pdf}};
	\endxy}
&
\xy
(0,0)*{\reflectbox{\includegraphics[scale=0.75]{figs/fig5-mrm3.pdf}}};
(0,-8)*{\text{\tiny\text{start }$h_1$}};
(0,8)*{\text{\tiny\text{end }$h_2$}};
\endxy
\ar@[myblue][l]^{\xy
	(0,0)*{\includegraphics[scale=0.325]{figs/fig2-61.pdf}};
	\endxy}
\ar@{.>}@<25pt>@[myblue]@/^.5cm/[ll]|{{\color{myblue}\,\parep\parom^{-2}\parome^{3}\,}}
&
\xy
(0,0)*{\reflectbox{\includegraphics[scale=0.75]{figs/fig5-mrm-equi3.pdf}}};
\endxy
\ar@[mygreen][l]_{\xy
	(0,0)*{\includegraphics[scale=0.325]{figs/fig2-62.pdf}};
	\endxy}
&
\xy
(0,0)*{\reflectbox{\includegraphics[scale=0.75]{figs/fig5-mrm1.pdf}}};
(0,8)*{\text{\tiny\text{start }$h_2$}};
\endxy
\ar@[mygreen][l]_{\xy
	(0,0)*{\includegraphics[scale=0.325]{figs/fig5-squeeze1.pdf}};
	\endxy}
\ar@[mygreen]@{.>}@<-25pt>@/^-.5cm/[ll]|{{\color{mygreen}\,-\parep\parom^{-2}\parome^{3}\,}}
}
\end{gather}
The second illustration gives the chain maps $f_1,f_2,g_1,g_2$ 
and the third the non-trivial homotopies $h_1,h_2$---all of which are the composites 
of the displayed foams---as well as all their scalars. We leave it to the reader 
to verify that $f_1,f_2,g_1,g_2$ commute with the differentials 
$d_1,d_2,d_3,d_4$---i.e. 
$f_1\circ d_1+g_1\circ d_2=0$ and  
$d_3\circ f_2-d_4\circ g_2=0$---that $f_1,f_2$ are mutually inverses---i.e. 
$f_1\circ f_2=\mathrm{id}$ and $f_2\circ f_1=\mathrm{id}$---that $g_1\circ g_2=0$, 
$d_1\circ h_1 =- f_2\circ g_1$, $h_2\circ d_3 =- g_2\circ f_1$ and 
$d_2\circ h_1 - h_2\circ d_4=\mathrm{id}-g_2\circ g_1$.

\item \textbf{mR2} \textit{move, other versions.} As above but 
keeping in mind that some scalars depend on directions, 
e.g. the ones turning up for the squeeze \eqref{eq:squeezing}.
\end{enumerate} 

The arguments work for any specialization of $\parameter$.
\subsubsection{Tangle Reidemeister moves in general}
The case of $\Hgenpm{{{}_-}}$ can be proven as before with homotopies given as in the upwards oriented 
case but deleting the phantom facets which only 
show up due to the difference between $\web^\pm({}_-)$ and $\web({}_-)$. However, the scalars have to be adjusted as follows:
\begin{itemize}[leftmargin=.8in]

\item[(i)+(ii)] The braid-like moves have the same scalars.

\item[(iii)+(iv)] In case of \textbf{R1} (left and right), $f$
is scaled by $1$ instead of 
$\parep^2\parom^{-1}\parome^{-2}$, $g$ should get the 
same scalars, and $h$ gets $1$ 
instead of $\parep\parom^{-2}$.

\item[(v)+(vi)] In case of \textbf{mR2} (both versions), the only 
needed scalars are the minus signs.\qedhere

\end{itemize}
\subsection{Proof of \texorpdfstring{\fullref{theorem:functor-gl2}}{Theorem \ref{theorem:functor-gl2}}}\label{subsec:proof-six}
\subsubsection{The $2$-category $\twoTan$}
The underlying 
category of this $2$-category is $\Tan$ from \fullref{definition:tanglecat} (without relations) where 
we additionally fix an ordering of the crossings. 
Its 
$2$-morphisms---besides crossing 
reordering isomorphisms---can be seen as two-dimensional cobordisms embedded 
into $\R\times [-1,1]\times [-1,1]$ bounding a tangle at $\R\times [-1,1]\times \{-1\}$ 
and $\R\times [-1,1]\times \{1\}$. 
Categorical compositions are 
given by the evident (gluing) operations, see \cite[Chapters 1 and 2]{CS}.

As explained in \cite[Chapters 1 and 2]{CS}, the $2$-category 
$\twoTan$ admits a generator-relation 
presentation with generators given by certain 
basic cobordisms and the relations given by (oriented versions of)
the \textit{(Roseman--Carter--Saito) movie moves} \textbf{MM}, see 
e.g. \cite[Section 8]{BN1} whose numbering
we adapt. That is, \textbf{MM1} to \textbf{MM5} 
are the \textit{Reidemeister movie moves}, \textbf{MM6} to \textbf{MM10} 
are the \textit{reversible movie moves}, and 
\textbf{MM11} to \textbf{MM15} the \textit{non-reversible} ones. 
All of these come in various flavors depending on 
orientations and relative height of the strands, as well as 
ordering of the crossings.

The generating cobordisms come in two types (with $n=0,1,2$):

\begin{enumerate}[label=$\blacktriangleright$]

\renewcommand{\theenumi}{($\operatorname{Cob}$-A)}
\renewcommand{\labelenumi}{\theenumi}

\item \label{enum:cob-gen1} Morse type cobordisms attaching $n$-handles, i.e. caps, saddles or cups, cf. \eqref{eq:morse-foams}.

\renewcommand{\theenumi}{($\operatorname{Cob}$-B)}
\renewcommand{\labelenumi}{\theenumi}

\item \label{enum:cob-gen2} Cobordisms between different sides of Reidemeister moves
as in \eqref{eq:wrong-r2}.

\end{enumerate}
\subsubsection{The assignments}
Hence, we need to fix what to associate to these. 
For \ref{enum:cob-gen1} this is evident, cf. \eqref{eq:morse-foams}; to the Reidemeister cobordisms 
we associate the chain homotopies fixed in 
the proof of \fullref{theorem:functor-gl2}. 
Important hereby is that a \textit{non-braid like \textbf{R3} move}---with 
upwards and downwards pointing strands---can
be obtained from a braid like \textbf{R3} move, 
with \textbf{R2} and \textbf{mR2} moves as in \eqref{eq:wrong-r2}. 
However, there is 
a choice involved and we fix one, namely:
\begin{gather}\label{eq:last-proof-1}
\xy
(0,0)*{\includegraphics[scale=.65]{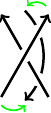}};
\endxy
\;\leftrightsquigarrow\;
\xy
(0,0)*{\includegraphics[scale=.65]{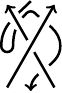}};
\endxy
\;\leftrightsquigarrow\;
\xy
(0,0)*{\includegraphics[scale=.65]{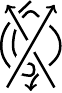}};
\endxy
\;\leftrightsquigarrow\;
\xy
(0,0)*{\includegraphics[scale=.65]{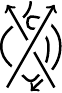}};
\endxy
\;\leftrightsquigarrow\;
\xy
(0,0)*{\includegraphics[scale=.65]{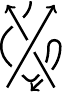}};
\endxy
\;\leftrightsquigarrow\;
\xy
(0,0)*{\includegraphics[scale=.65]{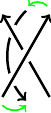}};
\endxy
\end{gather}
(This is the same choice ---and pictures---as in \cite[(3.5) and Example 3.21]{ETWe}. Important for us is  only that all such non-braid like \textbf{R3} moves 
consist of precisely two \textbf{mR2} moves on opposite sides of an 
``active strand''.)
\subsubsection{Reductions and the role of $\gltwo$}
W need to check that the movie moves \textbf{MM1} to \textbf{MM15} 
are preserved.
To this end, let us recall (some of) the ideas developed in \cite{ETWe} 
(which in turn are based on e.g. \cite{Bla}). The two 
main ideas therein can be summarized as follows:

\begin{enumerate}[label=$\blacktriangleright$]

\renewcommand{\theenumi}{(I)}
\renewcommand{\labelenumi}{\theenumi}

\item \label{enum:to-check1} It suffices to check functoriality in semisimple 
deformations.

\renewcommand{\theenumi}{(II)}
\renewcommand{\labelenumi}{\theenumi}

\item \label{enum:to-check2} For these deformations it is enough to check 
that the movie moves are preserved on so-called \textit{simple 
resolutions} only.

\end{enumerate}

In our setup \ref{enum:to-check1} means that we can set 
$\parto$ to $1$ and $\parha$ to $0$. This in turn means that we get the same idempotents 
as in the proof of \fullref{proposition:cats-are-equal-yes}, 
cf. \eqref{eq:the-idempotents}. These are elevant for (II). By writing out the idempotents, 
one checks that one has only the following 8 local colorings, 
four of which we call called \textit{admissible}:
\begin{gather}\label{eq:this-fails}
\xy
(0,0)*{
\xy
(0,0)*{\includegraphics[scale=1]{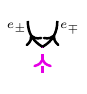}};
\endxy
\quad,\quad
\xy
(0,0)*{\includegraphics[scale=1]{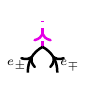}};
\endxy};
(0,-10)*{\text{\tiny admissible}};
\endxy
\quad,\quad
\xy
(0,0)*{
\reflectbox{\xy
(0,0)*{\includegraphics[scale=.65]{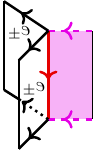}};
\endxy}
=
0\colon
\xy
(0,0)*{\includegraphics[scale=1]{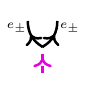}};
\endxy
\to
\xy
(0,0)*{\includegraphics[scale=1]{figs/fig6-func1b.pdf}};
\endxy};
(0,-10)*{\text{\tiny non-admissible}};
\endxy
\end{gather}
The non-admissible configurations are zero as the $\gltwo$ specialization swap the orthogonal idempotents $e_+$ and $e_-$ when moving over phantom edges, see 
\eqref{eq:slgl-parameters}. We 
abbreviate the $\gltwo$ specialization by $[\gltwo]$ in the following.

Following  \cite{BNM}, \cite[Section 3.2]{ETWe} we obtain a 
diagrammatic calculus 
for the Karoubi envelope $\Kar(\F^\pm[\gltwo])$ of 
the semisimplified foam $2$-category. 
In $\Kar(\F^\pm[\gltwo])$ ordinary edges of webs and 
ordinary facets of foams are labeled with the admissible idempotent
configurations. We have for instance the following four basic Morse type foams, all labeled 
with the idempotent $e_+$:
\begin{gather}\label{eq:morse-foams}
\xy
(0,0)*{\includegraphics[scale=.65]{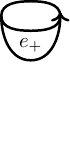}};
\endxy
\quad,\quad
\xy
(0,0)*{\includegraphics[scale=.65]{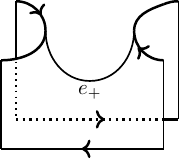}};
\endxy
\quad,\quad
\xy
(0,0)*{\includegraphics[scale=.65]{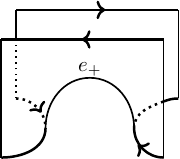}};
\endxy
\quad,\quad
\xy
(0,0)*{\includegraphics[scale=.65]{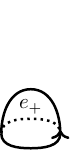}};
\endxy
\end{gather}
On the right we displayed examples of  relations for these foams which can easily be verified by writing out the idempotents.

\begin{gather}\label{eq:idem-rels}
\xy
(0,0)*{\includegraphics[scale=.4]{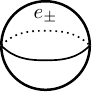}};
\endxy
\,=\,
\pm 1,
\quad\quad
\xy
(0,0)*{\includegraphics[scale=.4]{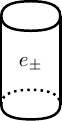}};
\endxy
\,=\,
\xy
(0,0)*{\includegraphics[scale=.4]{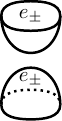}};
\endxy
\end{gather}

Let us denote the homotopy category of bounded complex with values in 
$\Kar(\F^\pm[\gltwo])$ by $\Komh(\Kar(\F^\pm[\gltwo]))$. Then---ignoring 
grading shifts since 
we are in the filtered only case---there are, cf. \cite[Figure 17]{Bla},
isomorphisms in $\Komh(\Kar(\F^\pm[\gltwo]))$ of the form
\begin{gather}\label{eq:last-proof-2}
\Hgenglpm{
\xy
(0,0)*{\includegraphics[scale=.65]{figs/fig5-9.pdf}};
\endxy
}
\!\!\!\!\cong\!\!
\underline{\xy
(0,0)*{\includegraphics[scale=.65]{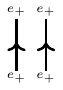}};
(0,6.25)*{\text{\tiny si.re.}};
(0,-5.25)*{\text{\tiny $\phantom{si.re.}$}};
\endxy
\oplus
\xy
(0,0)*{\includegraphics[scale=.65]{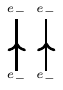}};
\endxy}
\xrightarrow{0}
\xy
(0,0)*{\includegraphics[scale=.65]{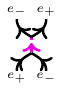}};
\endxy
\oplus
\xy
(0,0)*{\includegraphics[scale=.65]{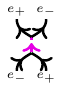}};
\endxy
\;\;\;,\;\;\;
\Hgenglpm{
\xy
(0,0)*{\includegraphics[scale=.65]{figs/fig5-10.pdf}};
\endxy
}
\!\!\!\!\cong\!\!
\xy
(0,0)*{\includegraphics[scale=.65]{figs/fig6-simple-crossing3.pdf}};
\endxy
\oplus
\xy
(0,0)*{\includegraphics[scale=.65]{figs/fig6-simple-crossing4.pdf}};
\endxy
\xrightarrow{0}
\underline{\xy
(0,0)*{\includegraphics[scale=.65]{figs/fig6-simple-crossing1.pdf}};
(0,6.25)*{\text{\tiny si.re.}};
(0,-5.25)*{\text{\tiny $\phantom{si.re.}$}};
\endxy
\oplus
\xy
(0,0)*{\includegraphics[scale=.65]{figs/fig6-simple-crossing2.pdf}};
\endxy}
\end{gather}
(The underlined part of the complex is 
in homological degree zero.) The resolutions, where all strands are colored by the idempotents 
$e_+$ are the \textit{simple resolutions}.

Then \ref{enum:to-check2} 
amounts to say that we can check if the movie moves hold on simple resolutions only. 
Thus, we can prove functoriality 
independent of the relative height of the involved strands 
and the orderings of the crossings, and only the relative 
orientation of the strands matters---which boils down the number of cases.
\subsubsection{Check of the movie moves}
For the Reidemeister movie moves 
\textbf{MM1} to \textbf{MM5}
we get the following simple resolutions. First the 
\textit{braid-like}:
\begin{gather}\label{eq:braid-like}
\xy
(0,0)*{\includegraphics[scale=.85]{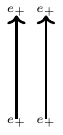}};
\endxy
\xrightarrow{\text{si.re.}}
\xy
(0,0)*{\includegraphics[scale=.85]{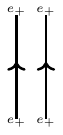}};
\endxy
\xleftarrow{\text{si.re.}}
\xy
(0,0)*{\includegraphics[scale=.85]{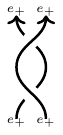}};
\endxy
\quad,\quad
\xy
(0,0)*{\includegraphics[scale=.85]{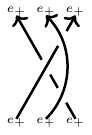}};
\endxy
\xrightarrow{\text{si.re.}}
\xy
(0,0)*{\includegraphics[scale=.85]{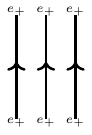}};
\endxy
\xleftarrow{\text{si.re.}}
\xy
(0,0)*{\includegraphics[scale=.85]{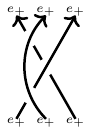}};
\endxy
\end{gather}
Next, the \textit{non-braid-like}: 
\begin{gather}\label{eq:non-braid-like}
\begin{gathered}
\raisebox{.019cm}{$\xy
(0,0)*{\includegraphics[scale=.80]{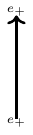}};
\endxy$}
\xrightarrow{\text{si.re.}}
\xy
(0,0)*{\includegraphics[scale=.80]{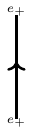}};
\endxy
\quad,\quad
\xy
(0,0)*{\includegraphics[scale=.80]{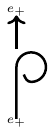}};
\endxy
\xrightarrow{\text{si.re.}}
\xy
(0,0)*{\includegraphics[scale=.80]{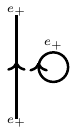}};
\endxy
\quad,\quad
\xy
(0,0)*{\includegraphics[scale=.80]{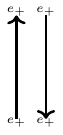}};
\endxy
\xrightarrow{\text{si.re.}}
\xy
(0,0)*{\includegraphics[scale=.80]{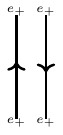}};
\endxy
\\
\raisebox{.017cm}{$
\xy
(0,0)*{\includegraphics[scale=.80]{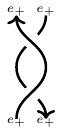}};
\endxy$}
\xrightarrow{\text{si.re.}}
\xy
(0,0)*{\includegraphics[scale=.80]{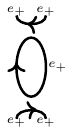}};
\endxy
\quad,\quad
\xy
(0,0)*{\includegraphics[scale=.80]{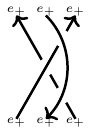}};
\endxy
\xrightarrow{\text{si.re.}}
\xy
(0,0)*{\includegraphics[scale=.80]{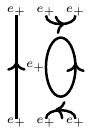}};
\endxy
\quad,\quad
\xy
(0,0)*{\includegraphics[scale=.80]{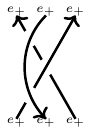}};
\endxy
\xrightarrow{\text{si.re.}}
\xy
(0,0)*{\includegraphics[scale=.80]{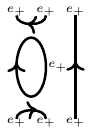}};
\endxy
\end{gathered}
\end{gather}
On this simple resolutions the chosen homotopies between the 
movie moves look as follows. First, the braid-like homotopies are 
all identities
\begin{gather}\label{eq:some-gather}
\raisebox{-.05cm}{$\begin{gathered}
\mathrm{MM}3a
\\
\mathrm{MM}4a
\end{gathered}$}
\colon\!\!\!
\xymatrix{
\xy
(0,0)*{\includegraphics[scale=.65]{figs/fig6-RM2-simple1b.pdf}};
\endxy
\ar@<2pt>[r]^{\text{id foam}}
&
\xy
(0,0)*{\includegraphics[scale=.65]{figs/fig6-RM2-simple1b.pdf}};
\endxy
\ar@<2pt>[r]^{\text{id foam}}
\ar@<2pt>[l]^{\text{id foam}}
&
\xy
(0,0)*{\includegraphics[scale=.65]{figs/fig6-RM2-simple1b.pdf}};
\endxy
\ar@<2pt>[l]^{\text{id foam}}
}
,\;
\mathrm{MM}5a\colon\!\!\!
\xymatrix{
\xy
(0,0)*{\includegraphics[scale=.65]{figs/fig6-RM3-simple1b.pdf}};
\endxy
\ar@<2pt>[r]^{\text{id foam}}
&
\xy
(0,0)*{\includegraphics[scale=.65]{figs/fig6-RM3-simple1b.pdf}};
\endxy
\ar@<2pt>[r]^{\text{id foam}}
\ar@<2pt>[l]^{\text{id foam}}
&
\xy
(0,0)*{\includegraphics[scale=.65]{figs/fig6-RM3-simple1b.pdf}};
\endxy
\ar@<2pt>[l]^{\text{id foam}}
}
\end{gather}
Checking that these compose to the identity---what we need to check 
in order for these movie moves to be satisfied---is trivial. The non-braid-like 
homotopies in the movies moves $\textbf{MM1}$ to $\textbf{MM5}$ 
are a bit ---although not much--- more involved:
\begin{gather}\label{eq:last-proof-11}
\textbf{MM1}\colon\!\!\!
\xymatrix@C1.1cm{
\xy
(0,0)*{\includegraphics[scale=.65]{figs/fig6-RM1-simple1b.pdf}};
\endxy \;
\ar@<2pt>[r]^{
\xy
(0,0)*{\includegraphics[scale=.35]{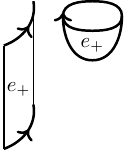}};
\endxy\hspace*{0.45cm}
}
&
\;
\xy
(0,0)*{\includegraphics[scale=.65]{figs/fig6-RM1-simple2b.pdf}};
\endxy \;
\ar@<2pt>[r]^{\hspace*{0.45cm}\xy
(0,0)*{\includegraphics[scale=.35]{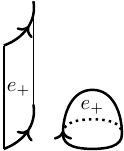}};
\endxy}
\ar@<2pt>[l]^{\xy
(0,0)*{\includegraphics[scale=.35]{figs/fig6-rm1-foam2.pdf}};
\endxy\hspace*{0.45cm}}
&
\;
\xy
(0,0)*{\includegraphics[scale=.65]{figs/fig6-RM1-simple1b.pdf}};
\endxy
\ar@<2pt>[l]^{\hspace*{0.45cm}\xy
(0,0)*{\includegraphics[scale=.35]{figs/fig6-rm1-foam1.pdf}};
\endxy}
}
,\;
\textbf{MM2}\colon\!\!\!
\xymatrix@C1.1cm{
\xy
(0,0)*{\includegraphics[scale=.65]{figs/fig6-RM1-simple2b.pdf}};
\endxy \;
\ar@<2pt>[r]^{\hspace*{0.45cm}
\xy
(0,0)*{\includegraphics[scale=.35]{figs/fig6-rm1-foam2.pdf}};
\endxy
}
&
\;
\xy
(0,0)*{\includegraphics[scale=.65]{figs/fig6-RM1-simple1b.pdf}};
\endxy \;
\ar@<2pt>[r]^{\xy
(0,0)*{\includegraphics[scale=.35]{figs/fig6-rm1-foam1.pdf}};
\endxy\hspace*{0.45cm}}
\ar@<2pt>[l]^{\hspace*{0.45cm}\xy
(0,0)*{\includegraphics[scale=.35]{figs/fig6-rm1-foam1.pdf}};
\endxy}
&
\; 
\xy
(0,0)*{\includegraphics[scale=.65]{figs/fig6-RM1-simple2b.pdf}};
\endxy
\ar@<2pt>[l]^{\xy
(0,0)*{\includegraphics[scale=.35]{figs/fig6-rm1-foam2.pdf}};
\endxy\hspace*{0.45cm}}
}\!,
\end{gather}
\vskip-.4em
\begin{gather}\label{eq:last-proof-12}
\textbf{MM3b}\colon\!\!\!\!\!
\xymatrix@C1.05cm{
\xy
(0,0)*{\includegraphics[scale=.65]{figs/fig6-RM2-simple3b.pdf}};
\endxy \;\;\;
\ar@<2pt>[r]^{
\xy
(0,0)*{\includegraphics[scale=.35]{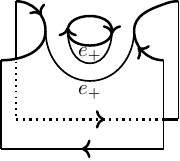}};
\endxy
}
&
\;\;\;\;
\xy
(0,0)*{\includegraphics[scale=.65]{figs/fig6-RM2-simple4b.pdf}};
\endxy
\;\;\;
\ar@<2pt>[r]^{\xy
(0,0)*{\includegraphics[scale=.35]{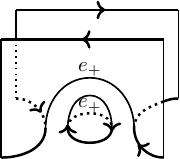}};
\endxy}
\ar@<2pt>[l]^{\xy
(0,0)*{\includegraphics[scale=.35]{figs/fig6-rm2-foam2.pdf}};
\endxy}
&
\;\;\;
\xy
(0,0)*{\includegraphics[scale=.65]{figs/fig6-RM2-simple3b.pdf}};
\endxy
\ar@<2pt>[l]^{\xy
(0,0)*{\includegraphics[scale=.35]{figs/fig6-rm2-foam1.pdf}};
\endxy}
}\!\!,
\end{gather}
\vskip-.7em
\begin{gather}\label{eq:last-proof-13}
\textbf{MM4b}\colon\!\!\!\!\!
\xymatrix@C1.05cm{
\xy
(0,0)*{\includegraphics[scale=.65]{figs/fig6-RM2-simple4b.pdf}};
\endxy\;\;\;
\ar@<2pt>[r]^{
\xy
(0,0)*{\includegraphics[scale=.35]{figs/fig6-rm2-foam2.pdf}};
\endxy
}
&
\;\;\;
\xy
(0,0)*{\includegraphics[scale=.65]{figs/fig6-RM2-simple3b.pdf}};
\endxy \;\;\;
\ar@<2pt>[r]^{\xy
(0,0)*{\includegraphics[scale=.35]{figs/fig6-rm2-foam1.pdf}};
\endxy}
\ar@<2pt>[l]^{\xy
(0,0)*{\includegraphics[scale=.35]{figs/fig6-rm2-foam1.pdf}};
\endxy}
&
\;\;\;
\xy
(0,0)*{\includegraphics[scale=.65]{figs/fig6-RM2-simple4b.pdf}};
\endxy
\ar@<2pt>[l]^{\xy
(0,0)*{\includegraphics[scale=.35]{figs/fig6-rm2-foam2.pdf}};
\endxy}
}
\!\!,
\end{gather}
\vskip-.9em
\begin{gather}\label{eq:last-proof-14}
\textbf{MM5b}\colon\!\!\!
\xymatrix@C1.07cm{
\xy
(0,0)*{\includegraphics[scale=.65]{figs/fig6-RM3-simple3b.pdf}};
\endxy
\ar@<2pt>[r]^{
f
}
&
\xy
(0,0)*{\includegraphics[scale=.65]{figs/fig6-RM3-simple4b.pdf}};
\endxy
\ar@<2pt>[r]^{g}
\ar@<2pt>[l]^{g}
&
\xy
(0,0)*{\includegraphics[scale=.65]{figs/fig6-RM3-simple3b.pdf}};
\endxy
\ar@<2pt>[l]^{f}
}
\end{gather}
with $f,g$ being ``monkey saddles'' as in \cite[Figure 9]{BN1}. 
Using \eqref{eq:idem-rels}, a visualization exercise 
shows that these also compose to the identities. In total, \textbf{MM1} to \textbf{MM5} are preserved.

That the remaining movie moves are preserved is not much harder to 
show---keeping in mind how easy the simple 
resolutions actually are, cf. \eqref{eq:braid-like} and \eqref{eq:non-braid-like}. It will
use only the relations in \eqref{eq:idem-rels}. Let us comment on two 
of them:
\begin{enumerate}

\item[\textbf{MM6}] There are essentially two versions, one of 
it contains only braid-like Reidemeister moves and is thus immediate. The other 
contains two \textbf{mR2} moves and a non-braid like 
\textbf{R3} move and is by far the hardest---but not really hard---move to be checked.

\item[\textbf{MM10}] By \cite[Proof of \textbf{MM10}]{CMW}, far-commutativity 
and \textbf{MM6} this move needs 
only be checked in its braid version and is therefore 
immediate.

\end{enumerate} 
The other movie moves hold in an almost trivial 
fashion and we omit these for brevity. 
Note that all involved 
foams are ordinary. This in turn ensures that 
working with either $\F^\pm[\gltwo]$ or with any further 
specializations does not affect functoriality.

\begin{remark} 
We like to finally point out, that  in the case of a $\sltwo$ choice of parameters, 
the ``admissibility trick'' as in \eqref{eq:this-fails} fails 
and one can find sign ambiguities exactly as 
Jacobsson did for Khovanov homology \cite{Jac}---noting again that 
everything in the semisimple deformation boils down to parameter free 
foam theories with only ordinary parts.
\end{remark}

\bibliographystyle{alphaurl}
\bibliography{generic-gl2}
\end{document}